\documentclass[article]{article}
\usepackage{color}
\usepackage{graphicx}
\usepackage{amsmath}
\usepackage{amssymb}
\usepackage{mathrsfs}
\usepackage[numbers,sort&compress]{natbib}
\usepackage{amsthm}
\usepackage{extarrows}
\usepackage{tikz}
\usepackage{float}
\usepackage{appendix}
\usepackage{hyperref}
\hypersetup{hypertex=true,
            colorlinks=true,
            linkcolor=blue,
            anchorcolor=blue,
            citecolor=red}
\usepackage{indentfirst}
\newtheorem{theorem}{Theorem}[section]
\newtheorem{lemma}{Lemma}[section]
\newtheorem{corollary}{Corollary}[section]
\newtheorem{proposition}{Proposition}[section]

\newtheorem{remark}{Remark}[section]
\newtheorem{RHP}{RHP}[section]
\newtheorem{Dbar}{$\bar{\partial}$-Problem}[section]
\usepackage{caption}
\usepackage{float}
\usepackage{subfigure}
\makeatletter
\@addtoreset{equation}{section}
\makeatother
\allowdisplaybreaks

\begin{document}
\title{Long time asymptotic behavior for the nonlocal mKdV equation in space-time solitonic regions-II}
\author{Xuan Zhou$^{a}$, Engui Fan$^{a}$\thanks{Corresponding author, with e-mail address
as faneg@fudan.edu.cn} }
\footnotetext[1]{ \  School of Mathematical Sciences, Fudan University, Shanghai 200433, P.R. China.}

\date{ }
\baselineskip=16pt
\maketitle
\begin{abstract}
  \baselineskip=16pt
  We study the long time asymptotic behavior for the Cauchy problem of an integrable real nonlocal mKdV equation with nonzero initial data in the solitonic regions
  \begin{align*}
    &q_t(x,t)-6\sigma q(x,t)q(-x,-t)q_{x}(x,t)+q_{xxx}(x,t)=0, \\
    &q(x,0)=q_{0}(x),\ \ \lim_{x\to \pm\infty} q_{0}(x)=q_{\pm},
  \end{align*}
  where $|q_{\pm}|=1$ and $q_{+}=\delta q_{-}$, $\sigma\delta=-1$. In our previous article, we have obtained long time asymptotics for the nonlocal mKdV equation in the solitonic region $-6<\xi<6$ with $\xi=\frac{x}{t}$. In this paper, we calculate the asymptotic expansion of the solution $q(x,t)$ for   other  solitonic regions $\xi<-6$ and $\xi>6$.
   Based on the Riemann-Hilbert problem of    the the Cauchy problem, further using the $\bar{\partial}$ steepest descent method,
   we derive different long time asymptotic expansions of the solution $q(x,t)$ in above two  different space-time solitonic regions.
   In the region  $\xi<-6$,    phase function $\theta(z)$ has four  stationary phase points on the $\mathbb{R}$.
    Correspondingly, $q(x,t)$ can be characterized with an $\mathcal{N}(\Lambda)$-soliton on discrete spectrum, the leading order term on continuous spectrum and an residual error term,
    which  are affected by a function ${\rm Im}\nu(\zeta_i)$.
  In the region $\xi>6$,    phase function $\theta(z)$  has  four  stationary phase points  on  $i\mathbb{R}$, the corresponding asymptotic approximations can be characterized with an $\mathcal{N}(\Lambda)$-soliton with diverse residual error order $\mathcal{O}(t^{-1})$.
\end{abstract}
\baselineskip=16pt

\newpage
\tableofcontents

\section{Introduction}
The study on the long-time behavior of nonlinear wave equations was first carried out with IST method by Manakov in 1974 \cite{Manakov1974}.
Later, by using this method, Zakharov and Manakov gave the first result on the large-time asymptotic of solutions for the NLS equation with decaying initial value \cite{Zak-Mana1976}. In 1993, Deift and Zhou developed rigorous analytic method to study the oscillatory RH problem associated with the mKdV equation \cite{DZ1993}. Then this method has been widely applied to integrable nonlinear evolution equations, such as KdV equation \cite{KDV},  the NLS equation \cite{ZD1994,ZD2003}, the sine-Gordon equation \cite{SG1,SG2}, the Camassa-Holm equation \cite{CH}, the Degasperis-Procesi \cite{DP}, the Fokas-Lenells equation\cite{FL} etc.

Recently, McLaughlin and Miller have developed a nonlinear steepest descent-method for the asymptotic analysis of RH problems based on the analysis of $\bar{\partial}$-problems, rather than the asymptotic analysis of singular integrals on contours \cite{MM1,MM2}. When it is applied to integrable systems, the $\bar{\partial}$-steepest descent method also has displayed some advantages, such as avoiding delicate estimates involving $L^{p}$ estimates of Cauchy projection operators, and leading the non-analyticity in the RH problem reductions to a $\bar{\partial}$-problem in some sectors of the complex plane. This method was adapted to obtain the long-time asymptotics for solutions to the NLS equation and the derivative NLS equation, with a sharp error bound for the weighted Sobolev initial data\cite{DM,CJ,BJ,JL}. Moreover, the comprehensive promotion of the above methods are also applied to derive the long-time asymptotics for solutions to the modified Camassa-Holm equation \cite{YF}, the focusing Fokas-Lenells equation \cite{CF}, the mKdV equation \cite{XZF,ZXF,CL} and so on.

In this paper, we investigate the long-time asymptotic behavior for the Cauchy problem of an integrable real nonlocal mKdV
equation under nonzero boundary conditions
\begin{align}
    &q_t(x,t)-6\sigma q(x,t)q(-x,-t)q_{x}(x,t)+q_{xxx}(x,t)=0, \label{nmkdv}\\
    &q(x,0)=q_{0}(x),\lim_{x\to \pm\infty} q_{0}(x)=q_{\pm},\label{datanmkdv}
\end{align}
where $|q_{\pm}|=1$ and $q_{+}=\delta q_{-}$, $\sigma\delta=-1$. The nonlocal equation \eqref{nmkdv} was introduced in \cite{nmkdv1,nmkdv2}, where $\sigma=\pm1$ denote the defocusing and focusing  cases, respectively, and  $q(x,t)$ is a real function.  The nonlocal equation \eqref{nmkdv} can be regarded as the integrable nonlocal extension of the well-known  classical  mKdV equation
\begin{equation}
q_t(x,t)-6\sigma q^{2}(x,t)q_{x}(x,t)+q_{xxx}(x,t)=0,\label{lmkdv}
\end{equation}
which appears in various of physical fields \cite{eg}. From the perspective of their structure, \eqref{lmkdv} can be translated to \eqref{nmkdv}
by replacing $q^{2}(x,t)$ with the \emph{PT}-symmetric term $q(x,t)q(-x,-t)$ (In the real field, the definition of \emph{PT} symmetry is given by \emph{P} : $x\rightarrow-x$ and \emph{T} : $t\rightarrow-t$) \cite{PT}. The general form of the nonlocal mKdV equation was shown to appear in the nonlinear oceanic and atmospheric dynamical system \cite{eg2}.

There is much work on the study of various mathematical properties for the nonlocal mKdV equation. Zhang and Yan rigorously analyzed dynamical behaviors of solitons and their interactions for four distinct cases of the reflectionless potentials for both focusing and defocusing nonlocal mKdV equations with NZBCs \cite{ZY}. The Darboux transformation was used to seek for soliton solutions of the focusing nonlocal mKdV equation \eqref{nmkdv} \cite{Darboux}, and the IST for the focusing nonlocal mKdV equation \eqref{nmkdv} with ZBC was presented \cite{IST}.  Moreover, the long-time asymptotics for the nonlocal defocusing mKdV equation with decaying initial data were investigated via Deift-Zhou steepest-descent method \cite{HF}.

In this work, we apply the $\bar{\partial}$-techniques to obtain the long-time asymptotic behavior of solutions to the nonlocal mKdV equation \eqref{nmkdv} in the case of $\xi<-6$ and $\xi>6$, while the case of $-6<\xi<6$ was discussed in detail in our previous article \cite{ZF}. Different from \cite{ZF}, in the case of $\xi<-6$, the value of $\nu(z)$ at four phase points $\zeta_i,i=1,2,5,6$ is not real. Moreover, the second leading term and the error term in the asymptotic expansion of the solution $q(x,t)$ are closely related to ${\rm Im}\nu(\zeta_i),i=1,2,5,6$. In the case of $\xi>6$, there is no phase point on the deformed jump contour $\Sigma^{(2)}$, which implies
that it is not necessary to consider local solvable RH models near phase points as \cite{ZF}. Here we outline our main results as follows.

\subsection{Main results}
The central results of this work give the long-time asymptotic behavior of the solutions $q(x,t)$ of nonlocal mKdV equation \eqref{nmkdv} in  space-time regions I and III, while the case of region II can refer our previous work \cite{ZF}.
\begin{figure}[htbp]
    \begin{center}
    \begin{tikzpicture}[node distance=2cm]
    \draw[yellow!20, fill=yellow!20] (0,0)--(-5,0)--(-5,1.5)--(0,0);
    \draw [blue!20,fill=blue!20]
    (0,0)--(-5,1.5)--(5,1.5)--(0,0);
    \draw [ green!20,fill=green!20]
    (0,0)--(5,0)--(5,1.5)--(0,0);
    \draw[red](0,0)--(-5,1.5)node[above,black]{$\xi=-6$};
    \draw[red](0,0)--( 5,1.5)node[above,black]{$\xi=6$};
    \draw[->](-5.5,0)--(5.5,0)node[right]{$x$};
    \draw[->](0,0)--(0,4)node[above]{$t$};
    \node[below]{$0$};
    \coordinate (A) at (-5, 0.8);
	\fill (A) node[right] {Region I};

    \coordinate (E) at (-7, 0.3);
	\fill (E) node[right] {\tiny{$T(\infty)^{-2}[q^{\Lambda}-\mathrm{i}\underset{i=1,2,5,6}\Sigma t^{-\frac{1}{2}+{\rm
    Im}\nu(\zeta_i)}f_i+R(t;\xi)]$}};

    \coordinate (C) at (0,2.5);
    \fill (C) node[below] {Region II};

    \coordinate (B) at (-2, 1);
	\fill (B) node[right] {\tiny{$cq_{sol}(x,t;\sigma_d^{\lambda})-t^{-\frac{1}{2}}f+\mathcal{O}(t^{-1})$}};

     \coordinate (D) at (5.1, 0.8);
	 \fill (D) node[left] {Region III};

     \coordinate (F) at (4.8, 0.4);
	 \fill (F) node[left] {\tiny{$T(\infty)^{-2}[q^{\Lambda}(x,t)+\mathcal(O)(t^{-1})]$}};

    \end{tikzpicture}
    \caption{\small The different space-time cones for $\xi=\frac{x}{t}$.
    Except phase points  $z=\pm 1$, the   function  $\theta(z)$  has  other four  phase points,   whose distribution depends on different  different space-time cones:
    Region I: $\xi<-6$,   four phase points are  located on real axis $\mathbb{R}$;
    Region II: $-6< \xi <6$,   four phase points  are  located on the unit circle; Region III: $\xi>6$,
     four phase  points are  located on the imaginary axis $\mathrm{i}\mathbb{R}$. } \label{cone}
    \end{center}
    \end{figure}
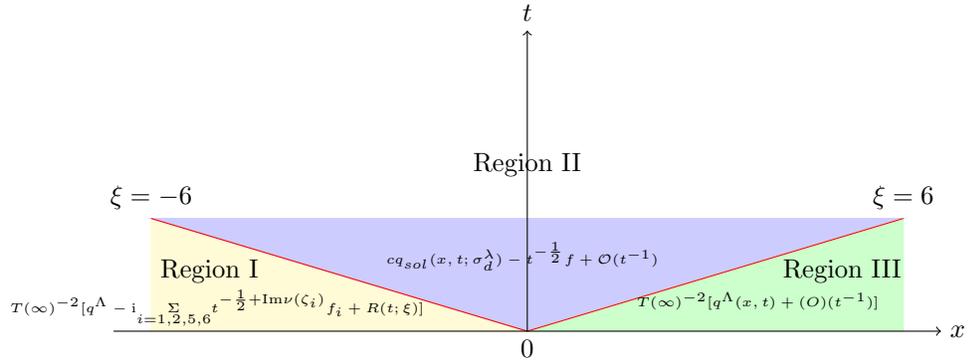
\begin{theorem}\label{th1}
Let $q(x,t)$ be the solution for the initial-value problem \eqref{nmkdv}--\eqref{datanmkdv} with generic data
$q_0(x)\mp q_{\pm}\in L^{1,2}(\mathbb{R})$, $q'\in$ $W^{1,1}(\mathbb{R})$ and scattering data $\Big\{\rho(z), \tilde{\rho}(z), \{\eta_k, A[\eta_k]\}\Big\}_{k=1}^{2N_1+N_2}$. And $q^{\Lambda}(x,t)$ denote $\mathcal{N}(\Lambda)$-soliton solution corresponding to scattering data $\Big\{0, 0, \{\eta_k, A[\eta_k]\}\Big\}_{k\in\Lambda}$ shown in Proposition \ref{mlambda}, where $\Lambda$ is defined in \eqref{symbol}. Then for $t\rightarrow\infty$,
\begin{itemize}
\item[1.] In region I \textnormal{($\xi<-6$)}, we have asymptotic expansion:
\begin{equation}\label{q02}
q(x,t)=T(\infty)^{-2}\left[q^\Lambda(x,t)-\mathrm{i}\underset{i=1,2,5,6}\Sigma t^{-\frac{1}{2}+\rm{Im}\nu(\zeta_i)}f_i+R(t;\xi)\right].
\end{equation}
The above $R(t;\xi)$ in \eqref{q02} can be written as:
\begin{align}\label{R}
  R(t,\xi)=\left\{
    \begin{aligned}
    &\mathcal{O}(t^{-1+a+c}), \quad 0<b\leq a=c<\frac{1}{2} \ or \ -\frac{1}{6}<2a-\frac{1}{2}<b<0<a=c<\frac{1}{2},\\
    &\mathcal{O}(t^{-\frac{3}{4}+\frac{a}{2}}), \quad \frac{a}{2}-\frac{1}{4}<b<0<a=c\leq\frac{1}{6} \ or \ \frac{a}{2}-\frac{1}{4}<b<0<a<c=-b<\frac{1}{2},\\
    &\mathcal{O}(t^{-\frac{3}{4}}), \quad -\frac{1}{4}<b\leq a<0<c=-b<\frac{1}{2},
    \end{aligned}
           \right.
\end{align}
where $a=\underset{i=1,2,\cdots,6}\max{\rm{Im}\nu(\zeta_i)}$, $b=\underset{i=1,2,\cdots,6}\min{\rm{Im}\nu(\zeta_i)}$,  $c=\underset{i=1,2,\cdots,6}\max{|\rm{Im}\nu(\zeta_i)|}$.
\item[2.] In region III \textnormal{($\xi>6$)}, we have asymptotic expansion:
\begin{equation}
q(x,t)=T(\infty)^{-2}\left[q^\Lambda(x,t)+\mathcal{O}(t^{-1})\right].
\end{equation}
\end{itemize}
\end{theorem}
\subsection{Outline of this paper}
This  paper is organized as follows.

In section \ref{RHconstruct}, we get down to the spectral analysis on the Lax pair. Based on the analyticity, symmetry and asymptotics of the Jost solutions and scattering data, the RH problem for the Cauchy problem \eqref{nmkdv}-\eqref{datanmkdv} is established.

In section \ref{disphasepoint}, we analyze the distribution of saddle points for different $\xi$ and depict the decay regions of
$\left|e^{\pm2\mathrm{i}\theta(z)}\right|$ by some figures.

The section \ref{openlens1} shows key technical processing for deforming and decomposing the RH problem in the case of $\xi<-6$. In section \ref{1stdeform}, we introduce a matrix-valued function $T(z)$ to define a new RH problem for $m^{(1)}(z)$, which admits a regular discrete spectrum and two triangular decompositions of the jump matrix. In section \ref{mixedRHP}, we introduce $\mathcal{R}^{(2)}(z)$ to make continuous extension for the jump matrix and remove the jump from $\Sigma$ in such away that the new problem takes advantage of the decay of $\left|e^{\pm2\mathrm{i}\theta(z)}\right|$ for $z\notin\Sigma$. Consequently, a mixed $\bar{\partial}$-RH problem is set up in subsection \ref{422}. We further decompose the mixed $\bar{\partial}$-RH problem for $m^{(2)}(z)$ into a pure RH problem $m^{rhp}(z)$ and a pure $\bar{\partial}$-problem for $m^{(3)}(z)$ in subsection \ref{423}. In section \ref{purerhp}, we obtain the solution of the pure RH problem for $m^{rhp}(z)$ via an outer model $m^{sol}(z)$ for the soliton components to be solved in subsection \ref{out}, and an inner model $m^{lo}(z)$ which are approximated by solvable models for $m_i^{pc}, i=1,2,5,6$ obtained in subsection \ref{in}. The error function $E(z)$ between $m^{rhp}(z)$ and $m^{sol}(z)$ satisfies a small norm RH problem which is shown in subsection \ref{erro}. As for the pure $\bar{\partial}$-problem $m^{(3)}= m^{(2)}(m^{rhp})^{-1}$, we will present details in the section \ref{puredbar}.

In section \ref{openlens2}, we discuss related properties of the RH problem in the case of $\xi>6$ similar to section \ref{openlens1}.

Finally, in section \ref{results}, based on a series of transformations above, a decomposition formula for $m(z)$ is found
\begin{equation}
m(z)=T(\infty)^{-\sigma_{3}}m^{(3)}(z)E(z)m^{sol}(z)T(\infty)^{\sigma_{3}}\left[I+z^{-1}T_1^{\sigma_3}+\mathcal{O}(z^{-2})\right],
\end{equation}
from which we obtain the long-time asymptotic behavior for the solutions of the Cauchy problem \eqref{nmkdv}--\eqref{datanmkdv} of the nonlocal mKdV equation in regions $\xi<-6$ and $\xi>6$.
\hspace*{\parindent}

\section{The Spectral Analysis and the RH Problem}\label{RHconstruct}
\subsection{Notations}
We recall  some notations. $\sigma_1, \sigma_2,\sigma_3$ are classical Pauli matrices
as follows
\begin{equation*}
    \sigma_1=\begin{bmatrix} 0 & 1 \\ 1 & 0 \end{bmatrix},\quad
    \sigma_2=\begin{bmatrix} 0 & -i \\ i & 0 \end{bmatrix}, \quad
    \sigma_3=\begin{bmatrix} 1 & 0 \\ 0 & -1 \end{bmatrix}.
\end{equation*}

We introduce the Japanese bracket $\langle x\rangle:=\sqrt{1+|x|^2}$ and the normed spaces:
\begin{itemize}
    \item A weighted $L^{p,s}(\mathbb{R})$ is defined by
    \begin{equation*}
        L^{p,s}(\mathbb{R})=\{u\in L^{p}(\mathbb{R}):\langle x\rangle^s u(x)\in L^{p}(\mathbb{R})\},
    \end{equation*}
    whose norm is defined by $\Vert u \Vert_{L^{p,s}(\mathbb{R})}:=\Vert\langle x \rangle^{s}u \Vert_{L^{p}(\mathbb{R})}$.
    \item A Sobolev space is defined by
    \begin{equation*}
        W^{m,p}(\mathbb{R})=\{u\in L^{p}(\mathbb{R}): \partial ^{j}u(x)\in L^{p}(\mathbb{R}) \quad{\rm for}\quad j=0,1,2,\dots,m\},
    \end{equation*}
 with the  norm  $\Vert u \Vert_{W^{m,p}(\mathbb{R})}:=\sum_{j=0}^{m}\Vert \partial^{j}u \Vert_{L^{p}(\mathbb{R})}$. Additionally, we are
    used to expressing $H^{m}(\mathbb{R}):=W^{m,2}(\mathbb{R})$.
    \item A weighted Sobolev space is defined by
    \begin{equation*}
        H^{m,s}(\mathbb{R}):=L^{2,s}(\mathbb{R})\cap H^{m}(\mathbb{R}).
    \end{equation*}
\end{itemize}

In this paper, we use $a\lesssim b$ to express $\exists c=c(\xi)>0$, s.t. $a\leqslant cb$.

\subsection{The Lax pair and spectral analysis}

The nonlocal mKdV equation \eqref{nmkdv} posses the nonlocal Lax pair
\begin{equation}\label{lax pair}
     \Phi_x=X\Phi, \quad \Phi_t=T\Phi,
\end{equation}
where
\begin{equation}
\begin{split}
& X=ik\sigma_3+Q,\\
& T=[4k^{2}+2\sigma q(x,t)q(-x,-t)]X-2ik\sigma_{3}Q_{x}+[Q_{x},Q]-Q_{xx},\\
& Q=\begin{bmatrix} 0 & q(x,t) \\  \sigma q(-x,-t) & 0 \end{bmatrix}, \sigma=\pm1,
\end{split}
\end{equation}
and $\Phi=\Phi(x,t;k)$ is a matrix eigenfunction, $k$ is a spectral parameter.

Taking $q(x,t)=q_{\pm}$ in the Lax pair \eqref{lax pair}, we get the spectral problems
\begin{equation}\label{asy lax}
    \phi_x=X_{\pm}\phi, \quad \phi_t=T_{\pm}\phi,
\end{equation}
with $X_{\pm}=\underset{x\rightarrow\pm\infty}\lim X(x,t;k)=\begin{bmatrix} \mathrm{i}k & q_{\pm} \\  - q_{\pm} & -\mathrm{i}k \end{bmatrix}$, we have the fundamental matrix solution of Eq. (\ref{asy lax}) as
\begin{align}
 \phi_{\pm}(x,t;k)=\left\{
        \begin{aligned}
    &E_{\pm}(k)e^{it\theta (x,t;k)\sigma_{3}}, \quad k\neq\pm i,\\
    &I+[x+(4k^{2}-2)t]X_{\pm}(k), \quad k=\pm i,
    \end{aligned}
        \right.
\end{align}
where
 \begin{equation}
    E_{\pm}(k)=\begin{bmatrix} 1 & \frac{iq_{\pm}}{k+\lambda} \\  \frac{iq_{\pm}}{k+\lambda} & 1 \end{bmatrix}, \hspace{0.5cm} \lambda^{2}-k^{2}=1, \hspace{0.5cm} \theta(x,t;k)=\lambda\left[\frac{x}{t}+(4k^2-2)\right].
 \end{equation}
To avoid multi-valued case of eigenvalue $\lambda$, we introduce a uniformization variable
\begin{equation}
    z=k+\lambda,
\end{equation}
and obtain two single-valued functions
\begin{equation}
    k=\frac{1}{2}(z-\frac{1}{z}),\hspace{0.5cm}\lambda=\frac{1}{2}(z+\frac{1}{z}).
\end{equation}
We can define two domains $D_{+}$, $D_{-}$ and their boundary $\Sigma$ on $z$-plane by
\begin{align*}
    &D_{+}=\{z\in \mathbb{C}: (|z|-1) {\rm Im}z>0\},\\
    &D_{-}=\{z\in \mathbb{C}: (|z|-1) {\rm Im}z<0\},\\
    &\Sigma=\mathbb{R}\cup \{z\in \mathbb{C}: |z|=1\},
\end{align*}
which are  yellow region and white region respectively  shown  in the Figure \ref{analregion&spectrumsdis}.\\
\begin{figure}[H]
\begin{center}
\begin{tikzpicture}[node distance=2cm]
\draw [fill=pink,ultra thick,color=yellow!10] (0,0) rectangle (3,3);
\draw [fill=pink,ultra thick,color=yellow!10] (0,0) rectangle (-3,3);
\draw [fill=white] (0,0) circle [radius=2];
\filldraw[color=yellow!10](0,0)-- (-2,0) arc (180:270:2);
\filldraw[color=yellow!10](0,0)-- (0,-2) arc (270:360:2);
\draw [](-3.5,0)--(3.5,0)  node[right, scale=1] {Rez};
\draw [](0,-3.5)--(0,3.5)  node[above, scale=1] {Imz};
\draw[fill=red!60] (2.8,2.8) circle [radius=0.04];
\draw[fill=red!60] (-2.8,2.8) circle [radius=0.04];
\draw[fill=red!60] (0,2.8) circle [radius=0.04];
\draw[fill=blue!60] (0,1.429) circle [radius=0.04];
\draw[fill=blue!60] (0.714,0.714) circle [radius=0.04];
\draw[fill=blue!60] (-0.714,0.714) circle [radius=0.04];

\draw[fill=blue!60] (0,-2.4) circle [radius=0.04];
\draw[fill=red!60] (0,-1.67) circle [radius=0.04];
\draw[fill=red!60] (1,-1) circle [radius=0.04];
\draw[fill=blue!60] (-2,-2) circle [radius=0.04];
\draw[fill=blue!60] (2,-1) circle [radius=0.04];
\draw[fill=red!60] (-1.6,-0.8) circle [radius=0.04];

\draw (0,0)circle(2cm);
\node  [right]  at ( 2,1.8 ) {$D_+$};
\node  [right]  at ( 2,-1.4 ) {$D_-$};
\node  [right]  at (2.8,2.8) {$z_k$};
\node  [right]  at (-2.8,2.8) {$-\overline{z}_{k}$};
\node  [right]  at (0,2.8) {$\mathrm{i}\omega_{k}$};
\node  [right]  at (0,1.2) {$\frac{\mathrm{i}}{\omega_{k}}$};
\node  [below]  at (1.2,1.2) {$\overline{z}_{k}^{-1}$};
\node  [below]  at (-1.2,1.2) {$-z_{k}^{-1}$};
\end {tikzpicture}
\end{center}
\caption{The complex $z$-plane showing the discrete spectrums [zeros of scattering data $s_{11}(z)$ (red) in yellow region and those
of scattering data $s_{22}(z)$ (blue) in white region]. The yellow and white regions stand for $D_{+}$ and $D_{-}$, respectively.}
\label{analregion&spectrumsdis}
\end{figure}
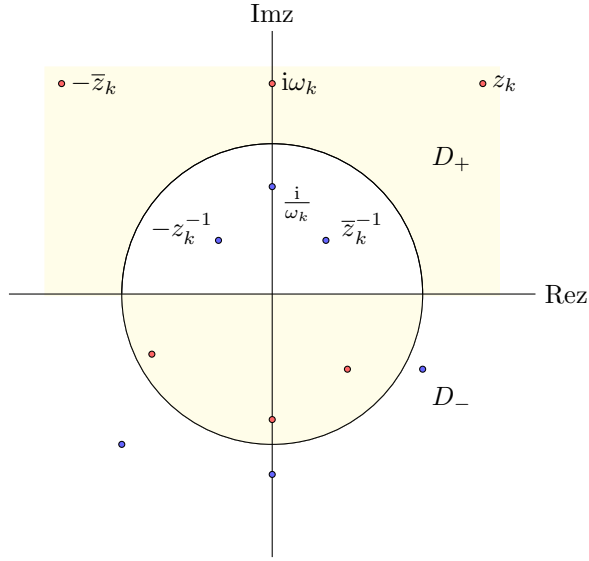

We know that the Jost solutions $\Phi_{\pm}(x,t,z)$ satisfy
\begin{equation}
   \Phi_{\pm}(x,t,z)\sim E_{\pm}(z)e^{\mathrm{i}t\theta(x,t,z)\sigma_{3}}.
\end{equation}
For convenience, we introduce the modified Jost solutions $\mu_{\pm}(x,t,z)$ by eliminating the exponential oscillations
\begin{equation}
   \mu_{\pm}(x,t,z)=\Phi_{\pm}(x,t,z)e^{-\mathrm{i}t\theta(x,t,z)\sigma_{3}},
\end{equation}
such that
\begin{equation}
   \lim_{x\to \pm\infty}\mu_{\pm}(x,t,z)=E_{\pm}(z),
\end{equation}
and $\mu_{\pm}$ admit the Volterra type integral equations
\begin{align}
 \mu_{\pm}(x,t;z)=E_{\pm}(z)+\left\{
        \begin{aligned}
    &\int_{\pm\infty}^{x}E_{\pm}(z)e^{\mathrm{i}\lambda(x-y)\sigma_{3}}[E_{\pm}^{-1}(z)\Delta Q_{\pm}(y,t)\mu_{\pm}(y,t,z)]dy, \quad k\neq\pm i,\\
    &\int_{\pm\infty}^{x}[I+(x-y)X_{\pm}]\Delta Q_{\pm}(y,t)\mu_{\pm}(y,t,z)dy, \quad k=\pm i,
    \end{aligned}
        \right.
\label{mu}
\end{align}
where $\Delta Q_{\pm}(x,t)= Q(x,t)-Q_{\pm}$. For convenience, we let $\Sigma_{0}=\Sigma/{\pm \mathrm{i}}$. Similarly to \cite{XZF,ZXF}, we have the following propositions:
\begin{proposition}\label{analydiff}
    Given $n\in\mathbb{N}_0$, let $q\mp q_{\pm}\in$ $L^{1,n+1}(\mathbb{R})$, $q'\in$ $W^{1,1}(\mathbb{R})$.
    \begin{itemize}
    \item $\mu_{+,1}$ and $\mu_{-,2}$ can be analytically extended to $D_{+}$ and continuously extended to $D_{+}\cup \Sigma_{0}$, $\mu_{-,1}$ and $\mu_{+,2}$ can be analytically extended to $D_{-}$ and continuously extended to $D_{-}\cup \Sigma_{0}$;

    \item The map $q\rightarrow$ $\frac{\partial^n}{\partial z^n}\mu_{\pm,i}(z)$ $(i=1,2, n\geq 0)$ are Lipschitz continuous, specifically, for any $x_0\in\mathbb{R}$, $\mu_{-,1}(z)$ and $\mu_{+,2}(z)$ are continuously differentiable mappings:
        \begin{align}
    &\partial_z^n\mu_{+,1}: \bar{D}_+\setminus\{0, \pm \mathrm{i}\} \rightarrow L^{\infty}_{loc}\{\bar{D}_+\setminus\{0, \pm \mathrm{i}\}, C^1([x_0,\infty),\mathbb{C}^2)\cap W^{1,\infty}([x_0,\infty),\mathbb{C}^2)\}, \\
    &\partial_z^n\mu_{-,2}: \bar{D}_+\setminus\{0, \pm \mathrm{i}\}\rightarrow L^{\infty}_{loc}\{\bar{D}_+\setminus\{0, \pm \mathrm{i}\}, C^1((-\infty,x_0],\mathbb{C}^2)\cap W^{1,\infty}((-\infty,x_0],\mathbb{C}^2)\},
    \end{align}
    $\mu_{+,1}(z)$ and $\mu_{-,2}(z)$ are continuously differentiable mappings:
    \begin{align}
    &\partial_z^n\mu_{-,1}: \bar{D}_-\setminus\{0, \pm \mathrm{i}\} \rightarrow L^{\infty}_{loc}\{\bar{D}_-\setminus\{0, \pm \mathrm{i}\}, C^1((-\infty,x_0],\mathbb{C}^2)\cap W^{1,\infty}((-\infty,x_0],\mathbb{C}^2)\},\\
    &\partial_z^n\mu_{+,2}: \bar{D}_-\setminus\{0, \pm \mathrm{i}\} \rightarrow L^{\infty}_{loc}\{\bar{D}_-\setminus\{0, \pm \mathrm{i}\}, C^1([x_0,\infty),\mathbb{C}^2)\cap W^{1,\infty}([x_0,\infty),\mathbb{C}^2)\}.
    \end{align}
    \item Let $K$ be a compact neighborhood of $\{-\mathrm{i},\mathrm{i}\}$ in $\bar{D}_+\setminus\{0\}$. Set $x^{\pm}=\max\{\pm x,0\}$, then there exists a constant $C$ such that for $z\in K$ we have
    \begin{equation}
    |\mu_{+,1}(z)-(1,\mathrm{i}z^{-1})^{\rm T}|\leq C\langle x^-\rangle e^{C\int_x^\infty \langle y-x\rangle|q-1|dy}\|q-\tilde{q}\|_{L^{1,1}(x,\infty)},
    \end{equation}
    i.e., the map $z\rightarrow \mu_{+,1}(z)$ extends as a continuous map to the points $\pm\mathrm{i}$ with values in $C^1([x_0,\infty),\mathbb{C})\cap W^{1,n}([x_0,\infty),\mathbb{C})$ for any preassigned $x_0\in\mathbb{R}$. Moreover, the map $q\rightarrow \mu_1^+(z)$ is locally Lipschitz continuous from:
    \begin{equation}\label{prop31}
    L^{1,1}(\mathbb{R})\rightarrow L^{\infty}(\bar{D}_+\setminus\{0\},C^1([x_0,\infty),\mathbb{C})\cap W^{1,\infty}([x_0,\infty),\mathbb{C}).
    \end{equation}
    Analogous statements hold for $\mu_{+,2}$ and for $\mu_{-,j}$ $(j=1,2)$.
    Furthermore, the maps $z\rightarrow \partial_z^n\mu_{+,1}(z)$ and $q\rightarrow \partial_z^n\mu_{+,1}(z)$ also satisfy:
    \begin{equation}
    |\partial_z^n\mu_{+,1}(z)|\leq F_n\left[(1+|x|)^{n+1}\| q-1\|_{L^{1,n+1}(x,\infty)}\right],\hspace{0.5cm}z\in K.
    \end{equation}
    \end{itemize}
    \end{proposition}
The asymptotic behavior of $\mu_{\pm, j}$, $j=1,2$ could be described by following proposition.
\begin{proposition}
    Suppose that $q\mp q_{\pm}\in$ $L^{1,n+1}(\mathbb{R})$ and $q'\in$ $W^{1,1}(\mathbb{R})$. Then as $z\rightarrow\infty$, we have
    \begin{align}
    &\mu_{\pm,1}(z)=e_1-\frac{\mathrm{i}}{z}\left(\begin{array}{c}
                                    \int^{x}_{\pm\infty}[\sigma q(y,t)q(-y,-t)+1]dy\\
                                    \sigma q(-x,-t)
                                    \end{array}\right)+\mathcal{O}(z^{-2}),\\
    &\mu_{\pm,2}(z)=e_2+\frac{\mathrm{i}}{z}\left(\begin{array}{c}
                                    q(x,t)\\
                                    \int^{x}_{\pm\infty}[\sigma q(y,t)q(-y,-t)+1]dy
                                    \end{array}\right)+\mathcal{O}(z^{-2}),
    \end{align}
 and as $z\rightarrow0$, we have
     \begin{align}
    &\mu_{\pm,1}(z)=\mathrm{i}\frac{q_{\pm}}{z}e_2+\mathcal{O}(1),\\
    &\mu_{\pm,2}(z)=\mathrm{i}\frac{q_{\pm}}{z}e_1+\mathcal{O}(1),
    \end{align}
\end{proposition}
where $e_1=(1,0)^{\rm T}$, $e_2=(0,1)^{\rm T}$.

It follows  that $\Phi_{\pm}(x,t;z)$ are fundamental solutions of Lax pair \eqref{lax pair} as $z\in \Sigma_{0}$, thus there exists a constant scattering matrix $S(z)=(s_{ij}(z))_{2\times2}$ (independence of    $x$ and $t$) such that
\begin{equation}
   \Phi{+}(x,t,z)=\Phi_{-}(x,t,z)S(z), \quad z\in \Sigma_{0},
\end{equation}
where $s_{ij}(z)$ can be expressed as
\begin{equation}\label{sij}
\begin{split}
& s_{11}(z)=\frac{{\rm det}(\Phi_{+,1},\Phi_{-,2})}{1+z^{-2}}, \quad s_{12}(z)=\frac{{\rm det}(\Phi_{+,2},\Phi_{-,2})}{1+z^{-2}},\\
& s_{21}(z)=\frac{{\rm det}(\Phi_{-,1},\Phi_{+,1})}{1+z^{-2}}, \quad s_{22}(z)=\frac{{\rm det}(\Phi_{-,1},\Phi_{+,2})}{1+z^{-2}}.
\end{split}
\end{equation}

 By calculating, we get the symmetry relations of the Jost solutions and scattering matrix for variables and isospectral parameter.
 \begin{proposition}\label{symmetrypro}
The symmetries for Jost solutions $\Phi_{\pm}(x,t,z)$ and scattering matrix $S(z)$ in $z\in \Sigma$ are listed as follows:
    \begin{itemize}
    \item The first symmtry
      \begin{equation}
      \Phi_{\pm}(x,t,z)=\sigma_{4}\overline{\Phi_{\mp}(-x,-t,-\bar{z})}\sigma_{4}, \quad S(z)=\sigma_{4}[\overline{S(-\bar{z})}]^{-1}\sigma_{4},
      \end{equation}
      where  $\sigma_{4}$ is defined as
      \begin{align}
       \sigma_{4}=\left\{
        \begin{aligned}
      &\sigma_{1}, \quad \sigma=-1,\\
      &\sigma_{2}, \quad \sigma=1.
      \end{aligned}
      \right.
      \end{align}
    \item The second symmetry
      \begin{equation}
      \Phi_{\pm}(x,t,z)=\overline{\Phi_{\pm}(-x,-t,-\bar{z})}, \quad S(z)=\overline{S(-\bar{z})}.
      \end{equation}
    \item The third symmetry
      \begin{equation}
      \Phi_{\pm}(x,t,z)=\frac{\mathrm{i}}{z}\Phi_{\pm}(x,t,-z^{-1})\sigma_{3}Q_{\pm}, \quad S(z)=(\sigma_{3}Q_{-})^{-1}S(-z^{-1})(\sigma_{3}Q_{+}).
      \end{equation}
    \end{itemize}
\end{proposition}
We use the scattering coefficients $s_{ij}(z), i,j=1,2$ to define reflection coefficients
\begin{equation}
   \rho(z)=\frac{s_{21}(z)}{s_{11}(z)}, \quad \tilde{\rho}(z)=\frac{s_{12}(z)}{s_{22}(z)}, \quad  z\in \Sigma.
\end{equation}
The following proposition provides some essential properties for $s_{ij}(z), i,j=1,2$ and $\rho(z), \tilde{\rho}(z)$.
\begin{proposition} Let $q\mp q_{\pm}\in$ $L^{1,n+1}(\mathbb{R})$ and  $q'\in$ $W^{1,1}(\mathbb{R})$, then
    \begin{itemize}
        \item [{\rm 1)}] For $z\in\left\{z\in\mathbb{C}: |z|=1\right\}/\{\pm\mathrm{i}\}$,
          \begin{equation}
          |\rho(z)\tilde{\rho}(z)|\leq 1-|s_{11}(z)|^{-2}<1.
          \end{equation}
        \item [{\rm 2)}] $s_{ij}(z), i,j=1,2$ and the reflection coefficient $\rho(z), \tilde{\rho}(z)$ satisfy the symmetries
        \begin{equation}
         s_{11}(z)=-\sigma s_{22}(-z^{-1}), \quad s_{12}(z)=-\sigma s_{21}(-z^{-1}), \quad  \tilde{\rho}(z)=\rho(-z^{-1}).
        \end{equation}
        \item [{\rm 3)}]The scattering data have the asymptotics
        \begin{align}
        &\lim_{z\rightarrow\infty}(s_{11}(z)-1)z=i\int_{\mathbb{R}}\left[\sigma q(y,t)q(-y,-t)-1\right]dy, \label{32}\\
        &\lim_{z\rightarrow0}s_{11}(z)=-\sigma, \label{33}
        \end{align}
        \begin{align}
        &|s_{21}(z)|=\mathcal{O}(|z|^{-2}),\hspace{0.5cm} \text{as }|z|\rightarrow\infty,\label{34}\\
        &|s_{21}(z)|=\mathcal{O}(|z|^{2}),\hspace{0.5cm} \text{as }|z|\rightarrow0, \label{35}
        \end{align}
         So that
    \begin{align}\label{36}
    \rho(z), \tilde{\rho}(z)\sim z^{-2}, \hspace{0.2cm}|z|\rightarrow\infty;\hspace{0.5cm}\rho(z), \tilde{\rho}(z)\sim 0,\hspace{0.2cm}|z|\rightarrow0.
    \end{align}
    \item [{\rm 4)}] Although $s_{11}(z)$ and $s_{21}(z)$ have singularities at points $\pm \mathrm{i}$, we can claim that the reflection
coefficient $\rho(z), \tilde{\rho}(z)$ remain bounded at $z=\pm \mathrm{i}$ and $|\rho(\pm \mathrm{i})|=1, |\tilde{\rho}(\pm \mathrm{i})|=1$. In fact, by direct calculation, we obtain
    \begin{equation}
    s_{11}(z)=\frac{\mp s^{\pm}}{z\mp \mathrm{i}}+\mathcal{O}(1), \quad s_{21}(z)=\frac{-\sigma s^{\pm}}{z\mp \mathrm{i}}+\mathcal{O}(1),
    \end{equation}
    \begin{equation}\label{rhob2}
   \lim_{z\to \pm\mathrm{i}}\rho(z)=\lim_{z\to \pm\mathrm{i}}\tilde{\rho}(z)=\pm\sigma,
    \end{equation}
    where $s^{\pm}=\frac{1}{2\mathrm{i}}{\rm det}(\Phi_{+,1}(\pm\mathrm{i}),\Phi_{-2}(\pm\mathrm{i}))$.
    \end{itemize}
\end{proposition}

Since one cannot exclude the possibilities of zeros for $s_{11}(z)$ and $s_{22}(z)$ along $\Sigma$. To solve the Riemann-Hilbert problem in the
inverse process, we only consider the potentials without spectral singularities, i.e., $s_{11}(z)\neq0$, $s_{22}(z)\neq0$ for $z\in\Sigma$.

The next proposition shows that, given data $q_{0}(x)$ with sufficient smoothness and decay properties, the reflection coefficients will also be smooth and decaying.
\begin{proposition}
For given $q\mp q_{\pm}\in L^{1,2}(\mathbb{R})$ and $q'\in$ $W^{1,1}(\mathbb{R})$, we then have $\rho(z),\tilde{\rho}(z)\in H^{1}(\Gamma)$, where $\Gamma$ is defined in \eqref{Gamma}.
\end{proposition}
\begin{proof}
Proposition \ref{analydiff} and \eqref{sij} indicate that $s_{11}(z)$ and $s_{21}(z)$ are continuous to $\Sigma/\{0,\pm\mathrm{i}\}$. Noticing $s_{11}(z)\neq0$, $s_{22}(z)\neq0$ for $z\in\Sigma$, then $\rho(z)$ is continuous to $\Sigma/\{0,\pm\mathrm{i}\}$. From \eqref{36} and \eqref{rhob2}, we know that $\rho(z)$ is bounded in the small neighborhood of $\{0,\pm\mathrm{i}\}$ and $\rho(z)\in L^{1}(\Gamma)\cap L^{2}(\Gamma)$. Next, we just need to  prove that $\rho'(z)\in L^{2}(\Gamma)$. For $\delta_{0}>0$ sufficiently small, from Proposition
\ref{analydiff}, the maps
\begin{equation}
q\rightarrow {\rm det}(\Phi_{+,1},\Phi_{-,2}), \quad q\rightarrow {\rm det}(\Phi_{-,1},\Phi_{+,1})
\end{equation}
are locally Lipschitz maps from
\begin{equation}\label{40}
\{q: q\in L^{1,n+1}(\mathbb{R}), q'\in W^{1,1}(\mathbb{R})\}\rightarrow W^{n,\infty}(\mathbb{R}/(-\delta_{0},\delta_{0})), n\geq0.
\end{equation}
Actually, $q\rightarrow \Phi_{+,1}(z,0)$ is, by Proposition \ref{analydiff}, a locally Lipschitz map with values in $W^{n,\infty}(\bar{D}_+/(0,\delta_{0}, \mathbb{C}^2))$. For $q\rightarrow \Phi_{-,2}(z,0)$ and $q\rightarrow \Phi_{-,1}(z,0)$ the same is true.
These and  \eqref{32} imply that $q\rightarrow\rho(z)$ is a locally Lipschitz map from the domain in \eqref{40} into
\begin{equation}
W^{n,\infty}(I_{\delta_{0}})\cap H^{n}(I_{\delta_{0}}),
\end{equation}
where $I_{\delta_{0}}=\Gamma/(-\delta_{0},\delta_{0})\cup dist(\pm\mathrm{i};\delta_{0})$. Now fix $\delta_{0}$ sufficiently small such that the $3$ intervals $dist(z;\pm \mathrm{i})\leq\delta_{0}$ and $|z-0|\leq\delta_{0}$ have no intersection. In the complement of their union
\begin{equation}
|\partial^{j}_{z}\rho(z)|\leq C_{\delta_{0}}\langle z\rangle^{-1}, \quad j=0,1.
\end{equation}
Let $|z-\mathrm{i}|<\delta_{0}$, then using the definition of $s_{+}$ we have
\begin{equation}
\rho(z)=\frac{s_{21}(z)}{s_{11}(z)}=\frac{{\rm det}[\Phi_{-,1},\Phi_{+,1}]}{{\rm det}[\Phi_{+,1},\Phi_{-,2}]}=\frac{\int_{\mathrm{i}}^{z}F(s)ds-\sigma s^{+}}{\int_{\mathrm{i}}^{z}G(s)ds- s^{+}},
\end{equation}
where
\begin{equation}
F(s)=\partial_{s}{\rm det}[\Phi_{-,1}(s)\Phi_{+,1}(s)], \quad G(s)=\partial_{s}{\rm det}[\Phi_{+,1},\Phi_{-,2}].
\end{equation}
If $s_{+}\neq0$ then it is clear from the above formula that $\rho'(z)$ exist and is bounded near $\mathrm{i}$.

If $s_{+}=0$, then $z=\mathrm{i}$ is not the pole of $s_{11}(z)$ and $s_{21}(z)$, so that $s_{11}(z)$ and $s_{21}(z)$ are continuous
at $z=\mathrm{i}$, then
\begin{equation}
\rho(z)=\frac{\int_{\mathrm{i}}^{z}F(s)ds}{\int_{\mathrm{i}}^{z}G(s)ds}.
\end{equation}
From \eqref{sij}, we have
\begin{equation}\label{s11}
(1+z^2)s_{11}(z)=z^2{\rm det}[\Phi_{+,1},\Phi_{-,2}].
\end{equation}
Since $s_{+}=0$ implies that ${\rm det}[\Phi_{+,1}(\mathrm{i}),\Phi_{-,2}(\mathrm{i})]=0$, differentiating \eqref{s11} at $z=\mathrm{i}$ we get
\begin{equation}
2\mathrm{i}s_{11}(\mathrm{i})=-\partial_{z}{\rm det}[\Phi_{+,1}(z),\Phi_{-,2}(z)]|_{z=\mathrm{i}}=-\mathrm{i}G(\mathrm{i}).
\end{equation}
With $|s_{11}(\mathrm{i})|^2=1+|s_{21}(\mathrm{i})|^2>1$, we have $G(\mathrm{i})\neq0$. It follows that the derivative $\rho'(z)$ is
bounded around $\mathrm{i}$. The same proof holds at $z=-\mathrm{i}$ and $z=0$. It follows that $\rho'(z)\in L^{2}(\Gamma)$.

According to the symmetry $\tilde{\rho}(z)=\rho(-z^{-1})$, we have $\tilde{\rho}(z)\in H^{1}(\Gamma)$.
\end{proof}
\begin{corollary}
    For  given $q\mp q_{\pm}\in L^{1,2}(\mathbb{R})$, $q'\in W^{1,1}(\mathbb{R})$,  we then have $\rho(z), \tilde{\rho}(z)\in H^{1,1}(\Gamma)$.
\end{corollary}
\begin{proof}
    Since $H^{1,1}(\Gamma)=L^{2,1}(\Gamma)\cap H^{1}(\Gamma)$, what we need to prove is $\rho(z), \tilde{\rho}(z)\in L^{2,1}(\Gamma)$.
    With \eqref{36}, we can see that
    \begin{equation}
        \vert z\vert^2 \rho^{2}(z), \vert z\vert^2\tilde{\rho}^{2}(z) \sim  \vert z\vert^{-2}, \quad |z|\rightarrow \infty
    \end{equation}
    Thus
    \begin{equation}
        \int_{\Gamma}\vert \langle z\rangle \rho(z)\vert^2, \int_{\Gamma}\vert \langle z\rangle \tilde{\rho}(z)\vert^2 <\infty,
    \end{equation}
    which help us obtain the result.
\end{proof}

Suppose that $s_{11}(z)$ has $N_{1}$ and $N_{2}$ simple zeros, respectively, in $D_{+}\cap \{z\in \mathbb{C}: Rez>0\}$ denoted by $z_{k}$, $k=1,2,\cdot\cdot\cdot, N_{1}$ and in $D_{+}\cap \{z\in \mathbb{C}: Rez=0\}$ denoted by $\mathrm{i}\omega_{k}$, $k=1,2,\cdot\cdot\cdot, N_{2}$. It follows from the symmetry relations of the scattering coefficients that
\begin{equation}
\begin{split}
& s_{11}(z_{k})=s_{11}(-\bar{z}_{k})=s_{22}(-z_{k}^{-1})=s_{22}(\bar{z}_{k}^{-1}), \quad k=1,2,\cdot\cdot\cdot, N_{1},\\
& s_{11}({\mathrm{i}\omega_{k}})=s_{22}(\mathrm{i}\omega_{k}^{-1}), \quad k=1,2,\cdot\cdot\cdot, N_{2}.
\end{split}
\end{equation}
It is convenient to define that
\begin{align}
       \eta_{k}=\left\{
        \begin{aligned}
      &z_{k}, \quad k=1,2,\cdot\cdot\cdot, N_{1},\\
      &-\bar{z}_{k-N_{1}}, \quad k=N_{1}+1,N_{1}+2,\cdot\cdot\cdot, 2N_{1}.\\
      &\mathrm{i}\omega_{k-2N_{1}}, \quad k=2N_{1}+1,2N_{1}+2,\cdot\cdot\cdot, 2N_{1}+N_{2},
      \end{aligned}
      \right.
\end{align}
and
\begin{equation}
\hat{\eta}_{k}=-\eta_{k}^{-1}, \quad k=1,2,\cdot\cdot\cdot, 2N_{1}+N_{2}.
\end{equation}
Thus, the discrete spectrum is given by
\begin{equation}
\mathcal{Z}\cup\hat{\mathcal{Z}}=\{\eta_{k},\hat{\eta}_{k}\}_{k=1}^{2N_{1}+N_{2}},
\end{equation}
and the distribution of $\mathcal{Z}\cup\hat{\mathcal{Z}}$ on the $z$-plane is shown in Figure \ref{analregion&spectrumsdis}.

Given $z_{0}\in\mathcal{Z}$, it follows from the Wronskian representations and $s_{11}(z_{0})=0$ that $\Phi_{+,1}(x,t,z_{0})$ and $\Phi_{-,2}(x,t,z_{0})$ are linearly dependent; Given $z_{0}\in\hat{\mathcal{Z}}$, it follows from the Wronskian representations $s_{22}(z_{0})=0$ and that $\Phi_{+,2}(x,t,z_{0})$ and $\Phi_{-,1}(x,t,z_{0})$ are linearly dependent. For convenience, we denote the proportional coefficient in the following definition of $b[z_{0}]$ by
$\frac{\Phi_{+,1}(x,t,z_{0})}{\Phi_{-,2}(x,t,z_{0})}$ or $\frac{\Phi_{+,2}(x,t,z_{0})}{\Phi_{-,1}(x,t,z_{0})}$ according to the region $z_{0}$ belongs to. Let
\begin{align}
 b[z_{0}]=
 \left\{
    \begin{aligned}
    &\frac{\Phi_{+,1}(x,t,z_{0})}{\Phi_{-,2}(x,t,z_{0})},  \quad  z_{0}\in\mathcal{Z},\\
    & \frac{\Phi_{+,2}(x,t,z_{0})}{\Phi_{-,1}(x,t,z_{0})}, \quad  z_{0}\in\hat{\mathcal{Z}},
    \end{aligned}
        \right.
    \quad
    \left\{
    \begin{aligned}
    &A[z_{0}]=\frac{b[z_{0}]}{s^{'}_{11}(z_{0})}, \quad  z_{0}\in\mathcal{Z},\\
    &A[z_{0}]=\frac{b[z_{0}]}{s^{'}_{22}(z_{0})}, \quad  z_{0}\in\hat{\mathcal{Z}}.
    \end{aligned}
        \right.
\end{align}
\begin{proposition}
For the given $z_{0}\in\mathcal{Z}\cup\hat{\mathcal{Z}}$, there exist three relations for $b[z_{0}]$, $s'_{11}(z_{0})$ and $s'_{22}(z_{0})$:
    \begin{itemize}
    \item The first relation
      \begin{equation}
       b[z_{0}]=-\frac{\sigma}{\overline{b[-\bar{z}_{0}]}}, \quad s'_{11}(z_{0})=-\overline{s'_{11}(-\bar{z}_{0})}, \quad s'_{22}(z_{0})=-\overline{s'_{22}(-\bar{z}_{0})}.
      \end{equation}
    \item The second relation
      \begin{equation}
       b[z_{0}]=-\overline{b[-\bar{z}_{0}]}, \quad s'_{11}(z_{0})=-\overline{s'_{11}(-\bar{z}_{0})}, \quad s'_{22}(z_{0})=-\overline{s'_{22}(-\bar{z}_{0})}.
      \end{equation}
    \item The third relation
      \begin{equation}
       b[z_{0}]=-\sigma b[-z_{0}^{-1}], \quad s'_{11}(z_{0})=-\sigma z_{0}^{-2}s'_{22}(-z_{0}^{-1}).
      \end{equation}
    \end{itemize}
\end{proposition}
From the first relation, one concludes that imaginary discrete spectrum $\mathrm{i}\omega_{k}$ exists if and only if $\sigma=-1$. That is to say that as $\sigma=1$, one has $N_{2}=0$. We give the following relation among discrete spectrum $\mathcal{Z}\cup\hat{\mathcal{Z}}$.
\begin{proposition}
The relations for $b[\cdot]$ and $A[\cdot]$ in $\mathcal{Z}\cup\hat{\mathcal{Z}}$ are given by
\begin{equation}
\begin{split}
& b[\eta_{k}]=\overline{b[-\bar{\eta}_{k}]}=-\sigma b[\hat{\eta}_{k}]=-\sigma b[\bar{\eta}_{k}^{-1}], \quad b^{2}[\eta_{k}]=1\\
& s'_{11}(\eta_{k})=-\overline{s'_{11}(-\bar{\eta}_{k})}=-\sigma \hat{\eta}_{k}^{2}s'_{22}(\hat{\eta}_{k})=\sigma \hat{\eta}_{k}^{2}s'_{22}(\bar{\eta}_{k}^{-1}).
\end{split}
\end{equation}
As $\sigma=-1$, one has
\begin{equation}
b[\mathrm{i} \omega_{k}]=b[\mathrm{i} \omega_{k}^{-1}], \quad s'_{11}(\mathrm{i} \omega_{k})=-\omega_{k}^{2}s'_{22}(\mathrm{i} \omega_{k}^{-1}).
\end{equation}
Then,
\begin{equation}
A[\hat{\eta}_{k}]=\hat{\eta}_{k}^{2}A[\eta_{k}].
\end{equation}
\end{proposition}

\subsection{A RH problem}
Define a sectionally meromorphic matrix as follows
\begin{equation}
    m(z):=m(x,t;z)=\left\{
        \begin{aligned}
        \left(\frac{\mu_{+,1}(x,t;z)}{s_{11}(z)}, \mu_{-,2}(x,t;z)\right), \quad z\in D_{+}, \\
        \left(\mu_{-,1}(x,t;z), \frac{\mu_{+,2}(x,t;z)}{s_{22}(z)}\right), \quad z\in D_{-},
        \end{aligned}
        \right.
\end{equation}
and
\begin{equation}
   m_{\pm}(x,t,z)=\lim_{z'\to z \atop z\in D_{\pm}}m(x,t,z') \quad z\in\Sigma,
\end{equation}
the multiplicative matrix Riemann-Hilbert problem can be proposed as follows.
\begin{RHP}\label{rhp0}
Find a $2\times 2$ matrix-valued function $m(z):=m(x,t;z)$ such that
\begin{itemize}
\item[*] Analyticity: $m(z)$ is analytical in $\mathbb{C}\backslash (\Sigma\cup\mathcal{Z}\cup\hat{\mathcal{Z}})$ and has simple poles in $\mathcal{Z}\cup\hat{\mathcal{Z}}=\{\eta_{k}, \hat{\eta}_{k}\}_{k=1}^{2N_{1}+N_{2}}$.
\item[*] Jump relation: $m_{+}(x,t,z)=m_{-}(x,t,z)v(z)$, where
\begin{equation}\label{rhp0jump}
    v(z)=\begin{bmatrix} 1-\rho(z)\tilde{\rho}(z) & -\tilde{\rho}(z)e^{2\mathrm{i}t\theta(x,t,z)} \\ \rho(z)e^{-2\mathrm{i}t\theta(x,t,z)} & 1 \end{bmatrix}, \quad z\in\Sigma.
\end{equation}
\item[*] Asymptotic behavior:
\begin{align}
    &m(x,t;z)=I+\mathcal{O}(z^{-1}), \quad  z\rightarrow\infty,\\
    &m(x,t;z)=\frac{\mathrm{i}}{z}\sigma_3Q_{-}+\mathcal{O}(1), \quad  z\rightarrow 0.
\end{align}
\item[*]Residue conditions
\begin{align}
&\underset{z=\eta_{k}}{\rm Res}m(z)=\lim_{z\rightarrow\eta_{k}}m(z)\begin{bmatrix} 0 & 0 \\ A[\eta_{k}] e^{-2\mathrm{i}t\theta(x,t,\eta_{k})} & 0 \end{bmatrix},\label{rhp0resa}\\
&\underset{z=\hat{\eta}_{k}}{\rm Res}m(z)=\lim_{z\rightarrow\hat{\eta}_{k}}m(z)\begin{bmatrix} 0 & A[\hat{\eta}_{k}]e^{2\mathrm{i}t\theta(,x,t,\hat{\eta}_{k})} \\ 0 & 0 \end{bmatrix},\label{rhp0resb}
\end{align}
where $\theta(x,t,z)=\lambda(z)[\frac{x}{t}+(4k^{2}(z)-2)]$.
\end{itemize}
\end{RHP}
The potential $q(x,t)$ is found by the reconstruction formula
\begin{equation}\label{resconstructm}
q(x,t)=-i(m_1)_{12}=-i\lim_{z\rightarrow\infty}(zm)_{12},
\end{equation}
where $m_1$ appears in the expansion of $m=I+z^{-1}m_1+O(z^{-2})$ as $z\rightarrow\infty$.
\hspace*{\parindent}

\section{Distribution of Saddle Points and Signature Table }\label{disphasepoint}
We notice that the long-time asymptotic behavior of RHP \ref{rhp0} is influenced by the growth and decay of the exponential
function
\begin{equation}\label{phasefunc}
e^{\pm 2it\theta}, \quad \theta(z)=\frac{1}{2}\left(z+\frac{1}{z}\right)\left[\frac{x}{t}-2+\left(z-\frac{1}{z}\right)^2\right],
\end{equation}
which not only appear in jump matrix $v(z)$ but also in the residue condition.  Based on this observation, we shall make analysis for the real part of
 $\pm 2it\theta(z)$ to ensure the exponential decaying property. Let $\xi=\frac{x}{t}$, we consider the stationary phase points and the real part of
$2\mathrm{i}t\theta(z)$:
\begin{equation}\label{theta'}
\theta'(z)=-\frac{(1-z^{2})(3z^{4}+\xi z^{2}+3)}{2z^{4}},
\end{equation}
\begin{equation}\label{Re 2itheta}
{\rm Re}[2\mathrm{i}t\theta(z)]=-2t{\rm Im}\theta(z)=-t{\rm Im}z(1-|z|^{-2})\Big[\xi-3+(1+|z|^{-2}+|z|^{-4})(3{\rm Re}^{2}z-{\rm Im}^{2}z)\Big].
\end{equation}
From Eq. \eqref{theta'}, we find six stationary phase points of $\theta(z)$:
\begin{equation}\label{Re 2itheta}
\pm 1, \quad \pm\frac{\sqrt{-\xi\pm\sqrt{-36+\xi^{2}}}}{\sqrt{6}}.
\end{equation}
Except to $z=\pm 1$, there are also four stationary phase points, whose distribution depends on different $\xi$ is as follows:
    \begin{itemize}
       \item[(i)]  For $\xi<-6$, the four phase points are located on real axis $\mathbb{R}$ corresponding to
         Figure \ref{theta'xiaoyu}. Among them, two are inside the unit circle and the other two are outside the unit circle;
        \item[(ii)]  For $-6<\xi<6$, the four phase points are all located on the unit circle and they are symmetrical to each other, which is
            corresponded to Figure \ref{theta'jieyu};
        \item[(iii)] For $\xi>6$, the four phase points are located on  the imaginary axis $\mathrm{i}\mathbb{R}$ corresponding to \textmd{Figure \ref{theta'dayu}}. The two of them are inside the unit circle and the other two are outside the unit circle;.
    \end{itemize}
\begin{figure}[htbp]
	\centering
	\subfigure[]{\includegraphics[width=0.3\linewidth]{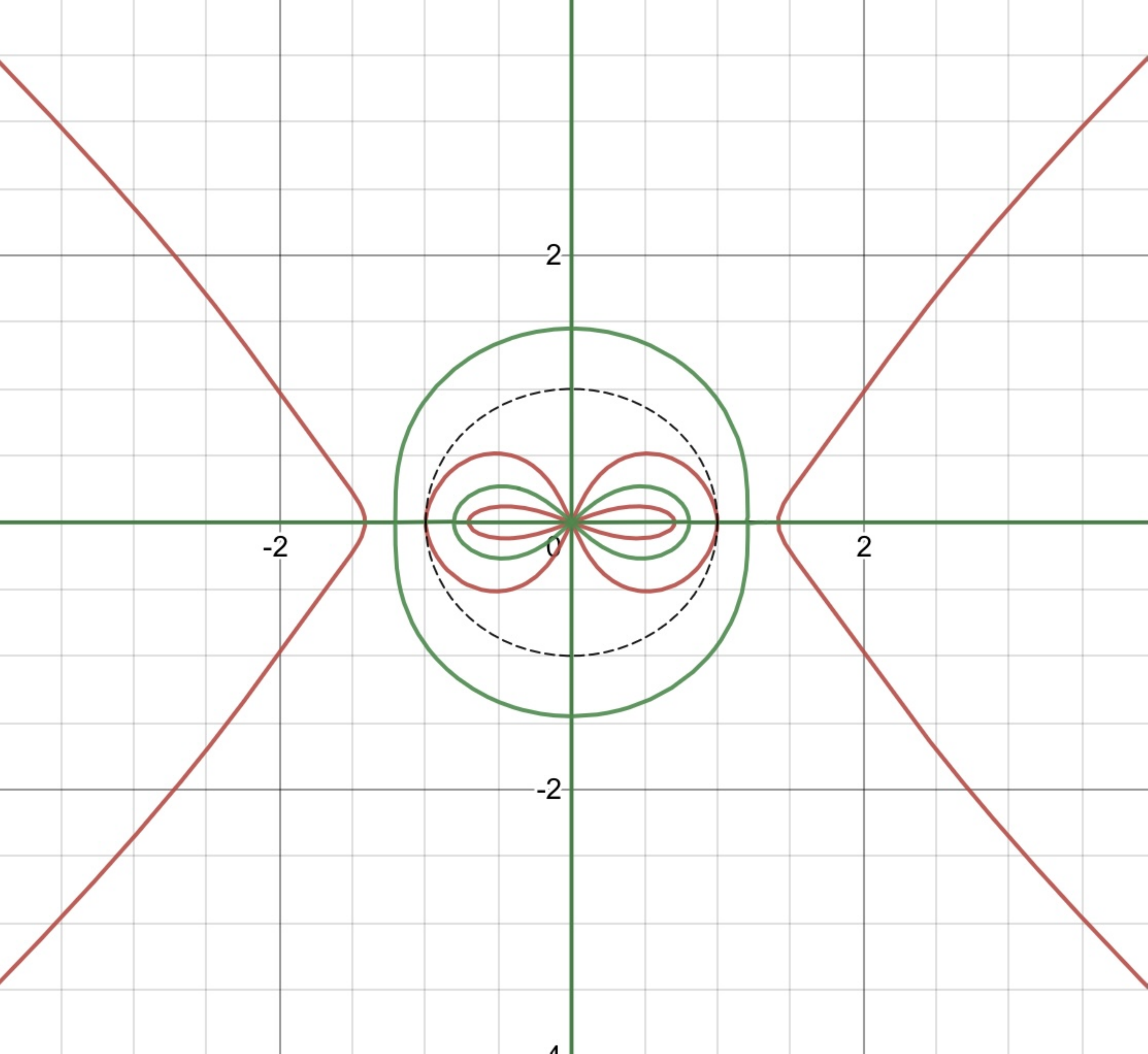}\hspace{0.5cm}\
	\label{theta'xiaoyu}}
	\subfigure[]{\includegraphics[width=0.3\linewidth]{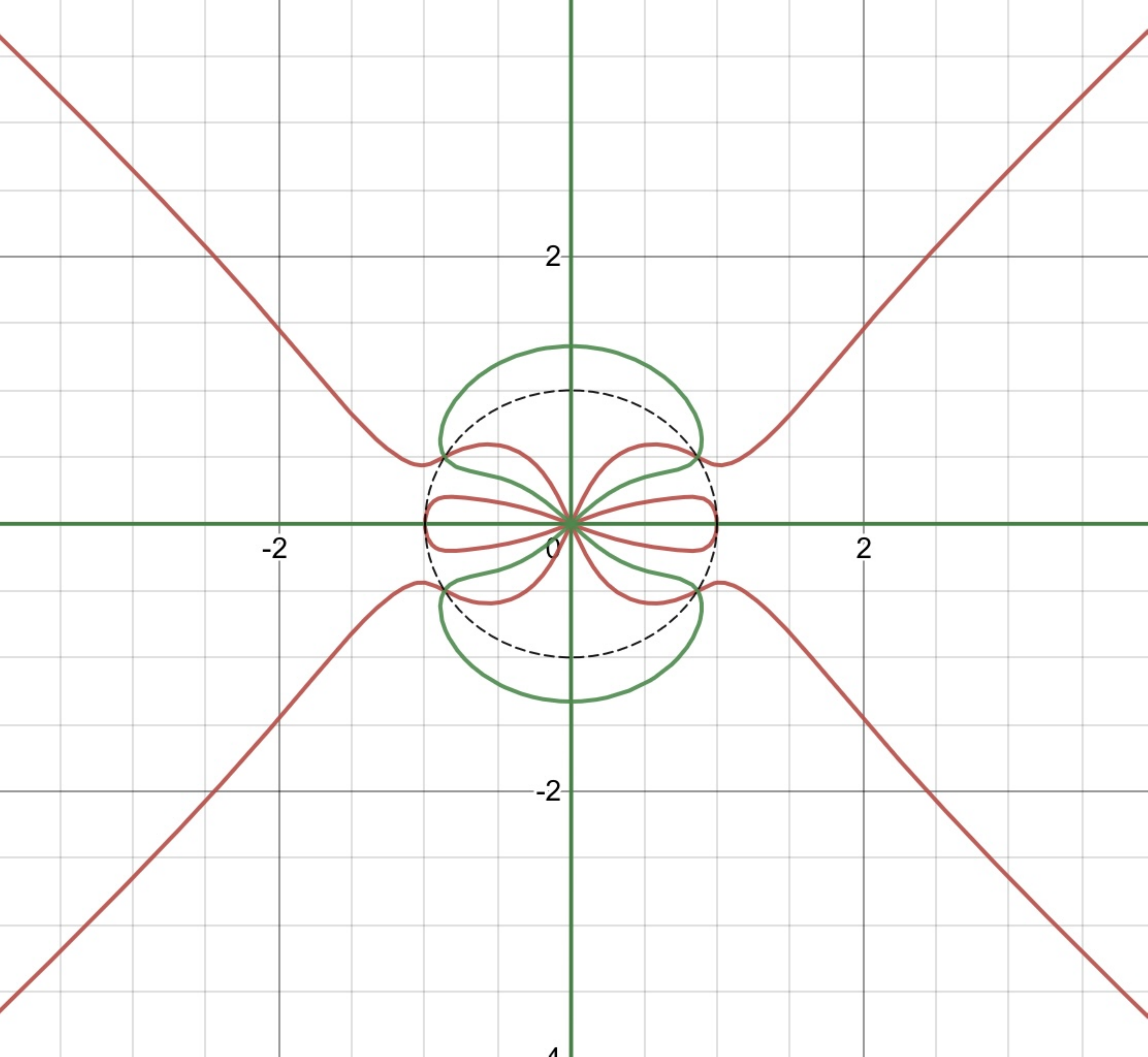}
	\label{theta'jieyu}}	
	\subfigure[]{\includegraphics[width=0.3\linewidth]{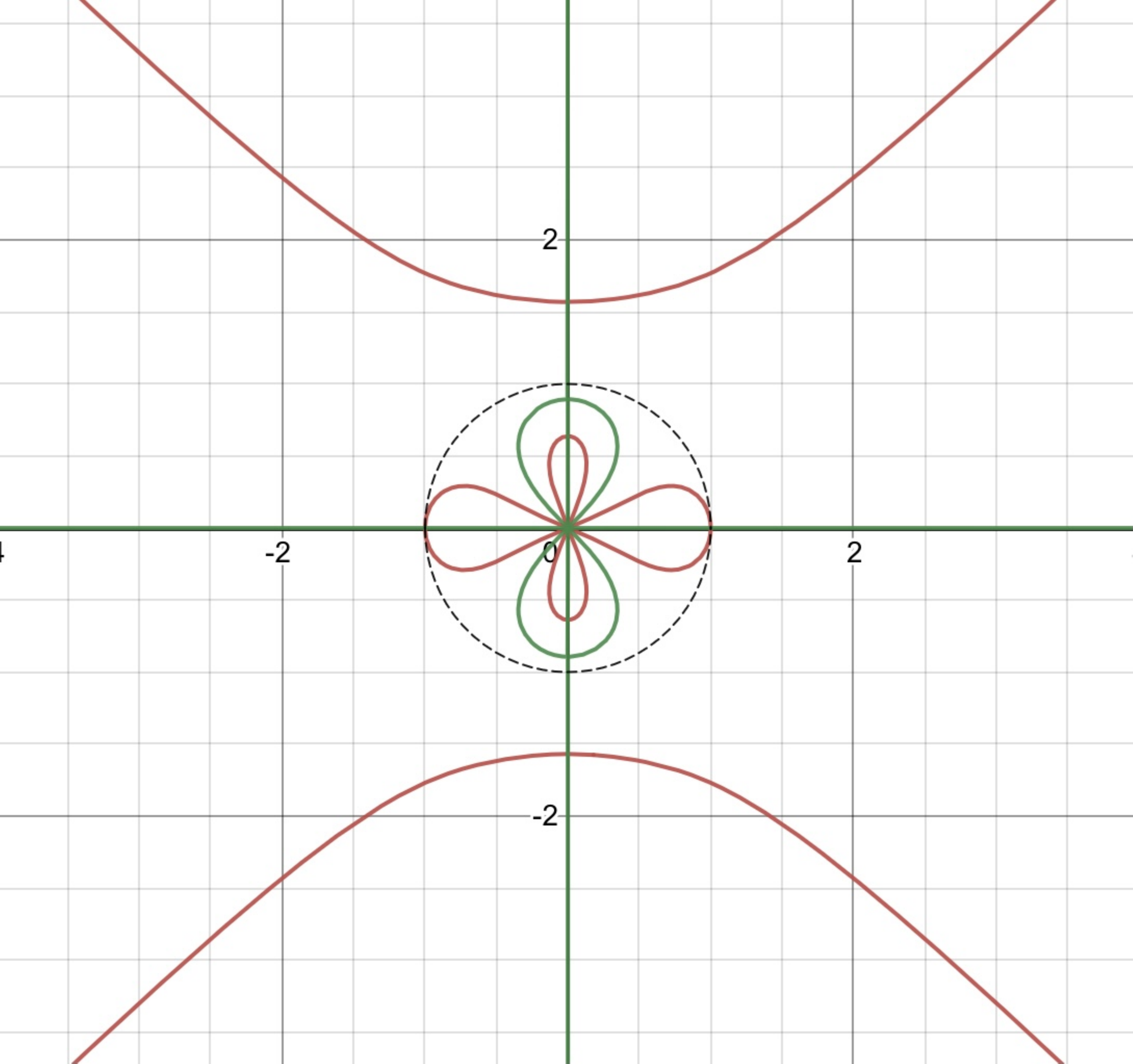}
	\label{theta'dayu}}
	\caption{\footnotesize Plots of the distributions for saddle points:
    $\textbf{(a)}$ $\xi<-6$,
    $\textbf{(b)}$ $-6<\xi<6$,
    $\textbf{(c)}$ $\xi>6$. The black dotted curve is an unit circle. The red curve shows the ${\rm Re} \theta'(z)=0$, and the green curve shows the ${\rm Im} \theta'(z)=0$.
     The intersection points are the saddle points which express $\theta'(z)=0$.}
	\label{figsaddle}
\end{figure}
\begin{figure}[htbp]
	\centering
	\subfigure[]{\includegraphics[width=0.3\linewidth]{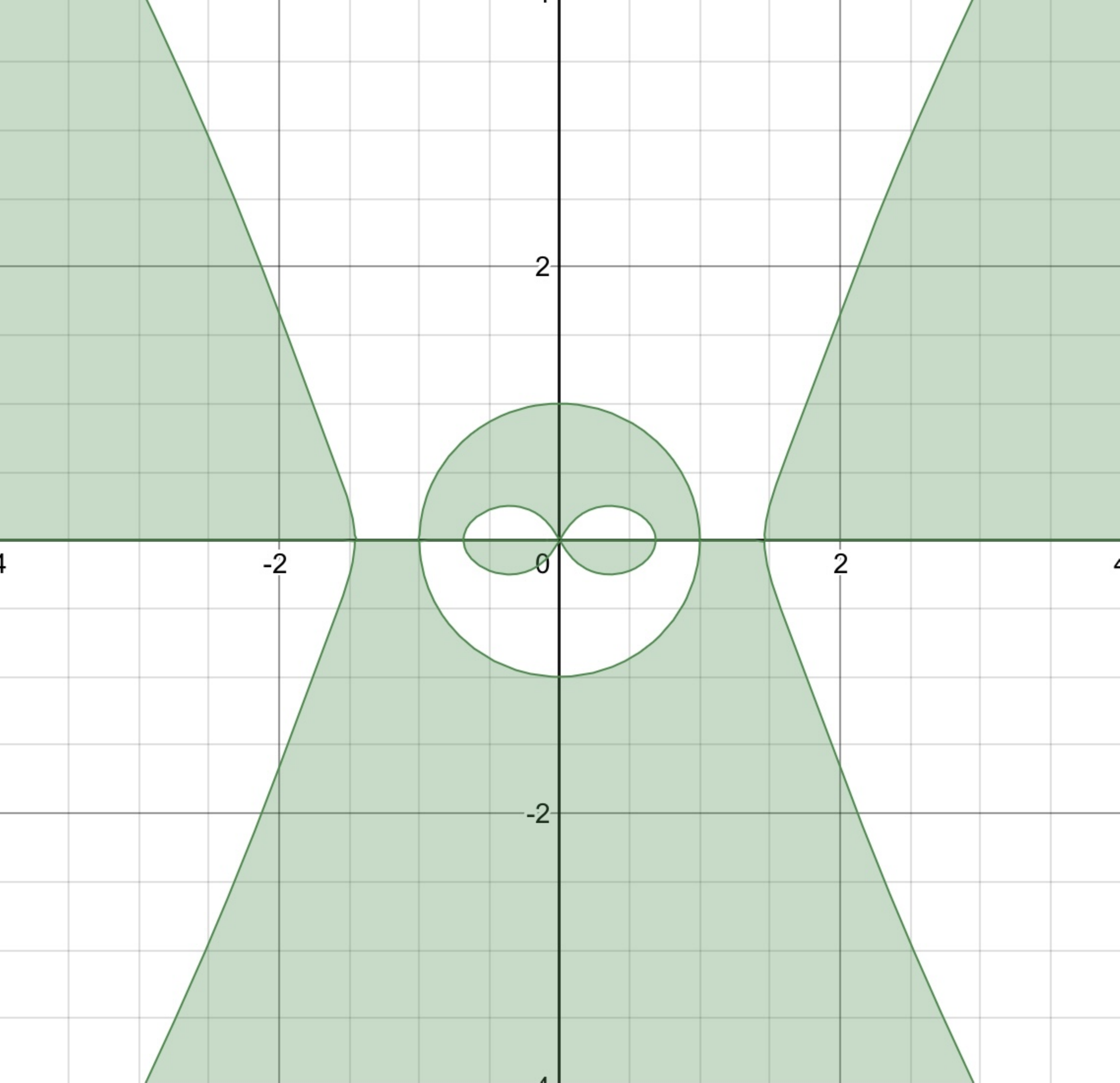}\hspace{0.5cm}\
	\label{Re2ithetaxiaoyu}}
	\subfigure[]{\includegraphics[width=0.3\linewidth]{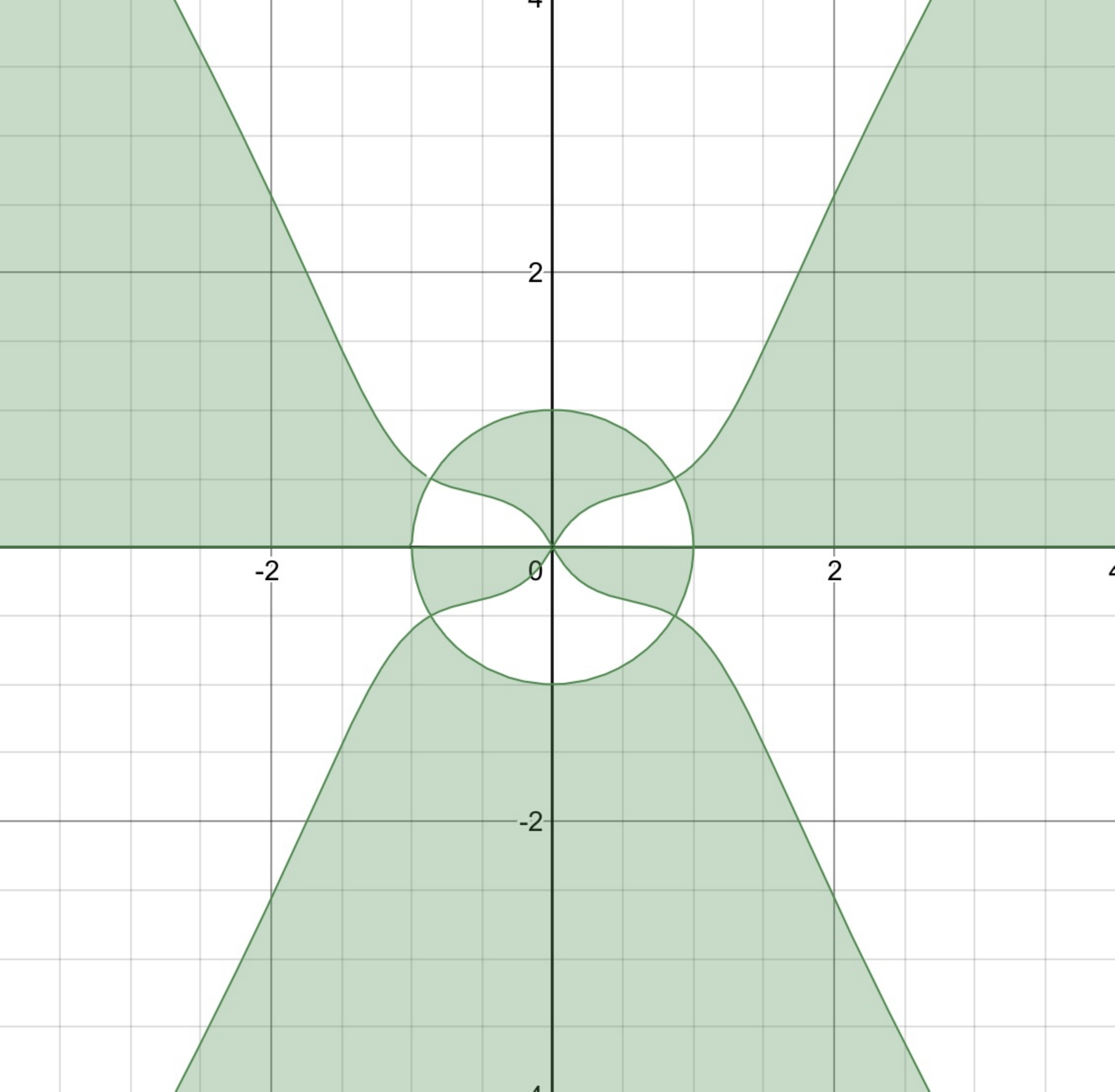}
	\label{Re2ithetajieyu}}	
	\subfigure[]{\includegraphics[width=0.3\linewidth]{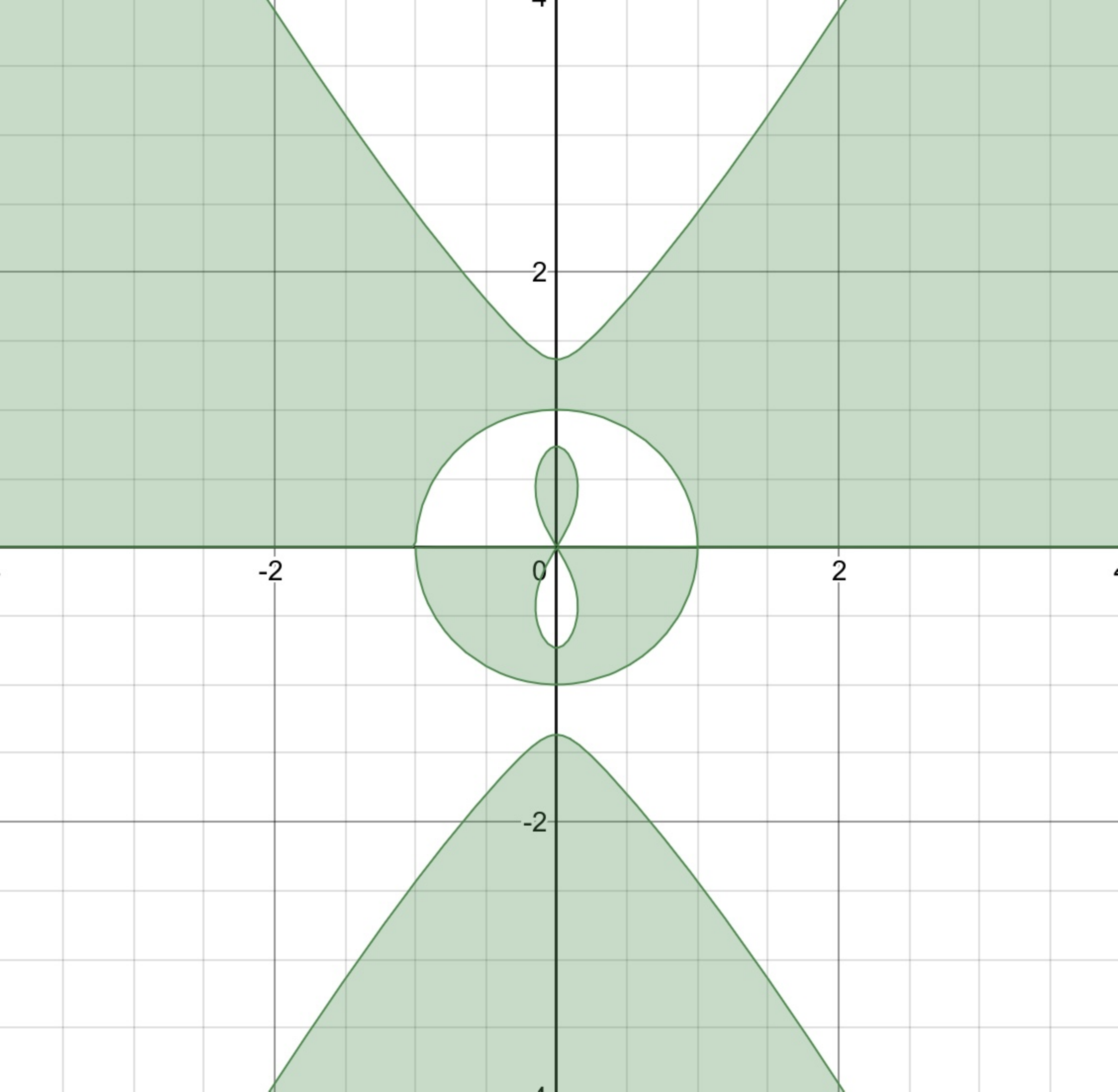}
	\label{Re2ithetadayu}}
	\caption{\footnotesize Signature table of ${\rm Re}(2it\theta)$ with different $\xi$:
    $\textbf{(a)}$ $\xi<-6$,
    $\textbf{(b)}$ $-6<\xi<6$,
    $\textbf{(c)}$   $\xi>6$. ${\rm Re}(2it\theta)<0$ in the green region and  ${\rm Re}(2it\theta)>0$ in
     the white region. In other words, $|e^{2it\theta}|\rightarrow 0$  as $t\rightarrow\infty$ in the green region and  $|e^{-2it\theta} |\rightarrow 0$ as $t\rightarrow\infty$  in the white region. Moreover, ${\rm Re}(2it\theta)=0$ on the green curve.}
	\label{figtheta}
\end{figure}
Further, the decaying regions of ${\rm Re}[2\mathrm{i}t\theta(z)]$ are shown in Figure \ref{figtheta}.
{\bf  We will mainly discuss this case (i) and (iii) in the present paper}.
\begin{remark}
According to \eqref{phasefunc}, $\theta(z)$ allows the following symmetry:
\begin{equation}
\theta(-z^{-1})=-\theta(z), \quad \theta(-\bar{z})=-\overline{\theta(z)}, \quad \theta(-\bar{z}^{-1})=-\overline{\theta(z)}.
\end{equation}
\end{remark}
\hspace*{\parindent}

\section{Deformation of  RH Problem in region  $\xi<-6$}\label{openlens1}

\subsection{Jump matrix  factorizations}\label{1stdeform}
In order to perform the long time analysis using the $\bar{\partial}$ steepest descent method, we need to perform two essential operations:
    \begin{itemize}
       \item[(i)] decompose the jump matrix $v(z)$ into appropriate upper/lower triangular factorizations so that the oscillating factor $e^{\pm2\mathrm{i}\theta(z)}$ are decaying in corresponding region respectively;
        \item[(ii)] interpolate the poles by trading them for jumps along small closed loops enclosing
          each pole \cite{CJ}.
    \end{itemize}
The first step is aided by two well known factorizations of the jump matrix $v(z)$
\begin{align}
 v(z)=\left\{
        \begin{aligned}
    &\begin{bmatrix} 1 & -\tilde{\rho}(z)e^{2\mathrm{i}t\theta(z)} \\  0 & 1 \end{bmatrix}\begin{bmatrix} 1 & 0 \\  \rho(z) e^{-2\mathrm{i}t\theta(z)} & 0 \end{bmatrix}, \quad z\in \tilde{\Gamma}\\
    &\begin{bmatrix} 1 & 0 \\  \frac{\rho(z)}{1-\rho(z)\tilde{\rho}(z)}e^{-2\mathrm{i}t\theta(z)} & 1 \end{bmatrix}
    \begin{bmatrix} 1-\rho(z)\tilde{\rho}(z) & 0 \\ 0 & \frac{1}{1-\rho(z)\tilde{\rho}(z)} \end{bmatrix}
    \begin{bmatrix} 1 & -\frac{\tilde{\rho}(z)}{1-\rho(z)\tilde{\rho}(z)}e^{2\mathrm{i}t\theta(z)} \\ 0  & 1 \end{bmatrix}, \quad z\in \Gamma,
    \end{aligned}
        \right.
\end{align}
where $\zeta_{j}, j=1,2,\cdot\cdot\cdot,6$ are six stationary phase points and
\begin{equation}\label{Gamma}
\begin{split}
& \tilde{\Gamma}=\{z\in \mathbb{C}: |z|=1\}\cup (\zeta_6,\zeta_2)\cup(\zeta_1,\zeta_5);\\
& \Gamma=\Sigma/\tilde{\Gamma}=(-\infty,\zeta_6)\cup(\zeta_2,0)\cup(0,\zeta_1)\cup(\zeta_5,+\infty).
\end{split}
\end{equation}
To remove the diagonal matrix in the middle of the second factorization, we introduce a scalar RH problem.
\begin{RHP}\label{scalarrhp}
Find a scalar function $\delta(z):= \delta(z;\xi)$, which is defined by the following properties:
\begin{itemize}
\item[*] Analyticity: $\delta(z)$ is analytical in $\mathbb{C}\backslash \Gamma$.
\item[*] Jump relation:
\begin{equation}
\begin{split}
& \delta{+}(z)=\delta{-}(z)(1-\rho(z)\tilde{\rho}(z)), \quad z\in {\Gamma};\\
& \delta{+}(z)=\delta{-}(z), \quad z\in\tilde{\Gamma}.
\end{split}
\end{equation}
\item[*] Asymptotic behavior:
\begin{equation}
\delta(z)\rightarrow 1, z\rightarrow\infty
\end{equation}
\end{itemize}
\end{RHP}
Utilizing the Plemelj's formula, we are arriving
\begin{equation}
  \delta(z)={\rm exp}\left[-\frac{1}{2\pi i}\int_{\Gamma}{\rm log}\left(1-\rho(s)\tilde{\rho}(s)\right)\frac{1}{s-z}ds\right].
\end{equation}
Taking $\nu(z)=-\frac{1}{2\pi}{\rm log}(1-\rho(z)\tilde{\rho}(z))$, then we can express
\begin{equation}\label{delta}
    \delta(z)={\rm exp}\left(i\int_{\Gamma}\frac{\nu(s)}{s-z}ds\right).
\end{equation}
\begin{remark}
    $\nu(\zeta_i)$ are complex-valued and
    \begin{equation}
    {\rm Im}\nu(\zeta_i)=-\frac{1}{2\pi}arg(1-\rho(\zeta_i)\tilde{\rho}(\zeta_i)), \quad i=1,2,5,6.
    \end{equation}
Assuming that $|{\rm Im}\nu(z)|<\frac{1}{2}$ for all $z\in \mathbb{R}$, it means that $\log{(1-\rho(z)\tilde{\rho}(z))}$ is single-valued and the singularity of $\delta(z)$ at $z=\zeta_i$ are square integrable. In a word, $\delta(z)$ is bounded at $z=\zeta_i $ \cite{YD}.
\end{remark}

For brevity, we denote $\mathcal{N}=\{1,2,\cdot\cdot\cdot,2N_{1}+N_{2}\}$. Moreover, we introduce a small positive constant $\varrho$:
\begin{equation}
\varrho=\frac{1}{2}{\rm min}\left\{\underset{k\in\mathcal{N}}{\rm min}\left\{|{\rm Im}(\eta_{k})|, |{\rm Im}(\hat{\eta}_{k})|\right\},{\underset{\lambda,\mu\in\mathcal{Z}\cup\hat{\mathcal{Z}},\lambda\neq\mu}{\rm min}|\lambda-\mu|},{\underset{\lambda\in\mathcal{Z}\cup\hat{\mathcal{Z}},i=1,2,3,4}{\rm min}|\lambda-\zeta_{i}|}\right\}.
\end{equation}
Taking $\delta_{0}<\varrho$, we define $\triangle, \nabla$ and $\Lambda$ of $\mathcal{N}$ as follows:
\begin{equation}\label{symbol}
\begin{split}
&\triangle=\{k\in\mathcal{N}:{\rm Re}(2\mathrm{i}\theta(\eta_{k}))<0\}, \quad \nabla=\{k\in\mathcal{N}:{\rm Re}(2\mathrm{i}\theta(\eta_{k}))>0\}, \\ &\Lambda=\{k\in\mathcal{N}:|{\rm Re}(2\mathrm{i}\theta(\eta_{k}))|<\delta_{0}\}.
\end{split}
\end{equation}
To distinguish different type of zeros, we further give
\begin{equation}
\begin{split}
&\triangle_{1}=\{k\in\{1,2,\cdot\cdot\cdot,N_{1}\}:{\rm Re}(2\mathrm{i}\theta(z_{k}))<0\},\\
&\nabla_{1}=\{k\in\{1,2,\cdot\cdot\cdot,N_{1}\}:{\rm Re}(2\mathrm{i}\theta(z_{k}))>0\}, \\
&\triangle_{2}=\{l\in\{1,2,\cdot\cdot\cdot,N_{2}\}:{\rm Re}(2\mathrm{i}\theta(\mathrm{i}\omega_{l}))<0\}, \\ &\nabla_{2}=\{l\in\{1,2,\cdot\cdot\cdot,N_{2}\}:{\rm Re}(2\mathrm{i}\theta(\mathrm{i}\omega_{l})))>0\}, \\
&\Lambda_{1}=\{k\in\{1,2,\cdot\cdot\cdot,N_{1}\}:|{\rm Re}(2\mathrm{i}\theta(z_{k}))|<\delta_{0}\}, \\ &\Lambda_{2}=\{l\in\{1,2,\cdot\cdot\cdot,N_{2}\}:|{\rm Re}(2\mathrm{i}\theta(\mathrm{i}\omega_{l}))|<\delta_{0}\}.\nonumber
\end{split}
\end{equation}
Define the function
\begin{equation}\label{T}
T(z):=T(z;\xi)=\prod_{k\in\triangle_{1}}\prod_{l\in\triangle_{2}}\frac{(z+z_{k}^{-1})(z-\bar{z}_{k}^{-1})(z-\mathrm{i}\omega_{l}^{-1})}
{(zz_{k}^{-1}-1)(z\bar{z}_{k}^{-1}+1)(\mathrm{i}\omega_{l}^{-1}z+1)}{\rm exp}[\mathrm{i}\int_{\Gamma}\nu(s)(\frac{1}{s-z}-\frac{1}{2s})ds].
\end{equation}
\begin{proposition}
The function defined by Eq. \eqref{T} has following properties:
    \begin{itemize}
        \item[(a)] $T(z)$ is meromorphic in $\mathbb{C}\backslash \Gamma$, and for each $k\in\triangle_{1}$, $l\in\triangle_{2}$, $z_{k}$, $-\bar{z}_{k}$, $\mathrm{i}\omega_{l}$ are simple poles and $-z_{k}^{-1}$, $\bar{z}_{k}^{-1}$, $\mathrm{i}\omega_{l}^{-1}$ are simple zeros of $T(z)$.
        \item[(b)]$T(z)=-[T(-z^{-1})]^{-1}$.
        \item[(c)] For $z\in\Gamma$,
        \begin{equation}
        \frac{T_{+}(z)}{T_{-}(z)}=1-\rho(z)\tilde{\rho}(z).
        \end{equation}
        \item[(d)] For $i=1,2,5,6$, as $z\rightarrow \zeta_{i}$ along any ray $\zeta_{i}+re^{\mathrm{i}\phi}$ with $r>0$ and $|\phi|<\pi$,
        \begin{align}\label{estT}
        &|T(z)-T_{i}(\zeta_{i})(z-\zeta_{i})^{-\mathrm{i}\nu(\zeta_{i})}|\leq c\left\|{\rm log}(1-\rho(z)\tilde{\rho}(z))\right\|_{H^{1}(\Gamma)}|z-\zeta_{i}|^{\frac{1}{2}-{\rm Im}\nu(\zeta_i)},\\
        &|T(z)-T(\pm 1)|\lesssim |z\mp 1|^{\frac{1}{2}},
        \end{align}
         where
         \begin{equation}
         \begin{split}
         &T_{i}(\zeta_i)=\prod_{k\in\triangle_{1}}\prod_{l\in\triangle_{2}}\frac{(\zeta_{i}+z_{k}^{-1})(\zeta_{i}-\bar{z}_{k}^{-1})(\zeta_{i}
         -\mathrm{i}\omega_{l}^{-1})}{(\zeta_{i}z_{k}^{-1}-1)(\zeta_{i}\bar{z}_{k}^{-1}+1)(\mathrm{i}\omega_{l}^{-1}\zeta_{i}+1)} {\rm exp}[\mathrm{i}\beta_{i}(\zeta_{i},\xi)],\\
         &\beta_{1}(z)=\left(\int_{-\infty}^{\zeta_6}+\int_{0}^{\zeta_2}+\int_{\zeta_5}^{+\infty}\right)\frac{\nu(s)}{s-z}ds+
         \int_{0}^{\zeta_1}\frac{\nu(s)+\nu(\zeta_1)}{s-z}ds+\nu(\zeta_1)\log{z}-\int_{\Gamma}\frac{\nu(s)}{2s}ds,\\
         &\beta_{2}(z)=\left(\int_{-\infty}^{\zeta_6}+\int_{0}^{\zeta_1}+\int_{\zeta_6}^{+\infty}\right)\frac{\nu(s)}{s-z}ds+
         \int_{\zeta_2}^{0}\frac{\nu(s)-\nu(\zeta_2)}{s-z}ds+\nu(\zeta_2)\log{(-z)}-\int_{\Gamma}\frac{\nu(s)}{2s}ds,\\
         &\beta_{5}(z)=\left(\int_{-\infty}^{\zeta_6}+\int_{\zeta_2}^{0}+\int_{0}^{\zeta_1}\right)\frac{\nu(s)}{s-z}ds+
         \int_{\zeta_5}^{+\infty}\frac{\nu(s)-\chi_5(s)\nu(\zeta_5)}{s-z}ds+\nu(\zeta_5)\log{(\zeta_5+1-z)}-\int_{\Gamma}\frac{\nu(s)}{2s}ds,\\
          &\beta_{6}(z)=\left(\int_{\zeta_2}^{0}+\int_{0}^{\zeta_1}+\int_{\zeta_5}^{+\infty}\right)\frac{\nu(s)}{s-z}ds+
         \int_{-\infty}^{\zeta_6}\frac{\nu(s)+\chi_6(s)\nu(\zeta_6)}{s-z}ds+\nu(\zeta_6)\log{(z-\zeta_6+1)}-\int_{\Gamma}\frac{\nu(s)}{2s}ds.
        \end{split}
        \end{equation}
    \end{itemize}
\end{proposition}
To trade the poles for jumps on small contours encircling each pole, we introduce a  matrix-valued function $G(z)$:
\begin{align}\label{G}
   G(z)=\left\{
        \begin{aligned}
    &\begin{bmatrix} 1 & 0 \\ -\frac{A[\eta_k]e^{-2\mathrm{i}t\theta(\eta_k)}}{z-\eta_k} & 1 \end{bmatrix}, \quad |z-\eta_k|<\varrho, \quad k\in \nabla/\Lambda,\\
    &\begin{bmatrix} 1 & -\frac{z-\eta_k}{A[\eta_k]e^{-2\mathrm{i}t\theta(\eta_k)}} \\ 0 & 1 \end{bmatrix}, \quad |z-\eta_k|<\varrho, \quad k\in \triangle/\Lambda,\\
    &\begin{bmatrix} 1 & -\frac{A[\hat{\eta}_k]e^{2\mathrm{i}t\theta(\hat{\eta}_k)}}{z-\hat{\eta}_k} \\ 0 & 1 \end{bmatrix}, \quad |z-\hat{\eta}_k|<\varrho, \quad k\in \nabla/\Lambda,\\
    &\begin{bmatrix} 1 & 0 \\ -\frac{z-\hat{\eta}_k}{A[\hat{\eta}_k]e^{-2\mathrm{i}t\theta(\hat{\eta}_k)}} & 1 \end{bmatrix}, \quad |z-\hat{\eta}_k|<\varrho, \quad k\in \triangle/\Lambda,\\
    &I, \quad elsewhere.
    \end{aligned}
        \right.
    \end{align}

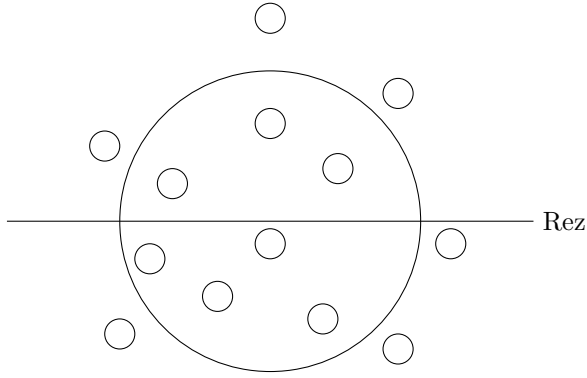
\begin{figure}[H]
\begin{center}
\begin{tikzpicture}[node distance=2cm]
\draw [](-3.5,0)--(3.5,0)  node[right, scale=1] {Rez};
draw (-0.7,1) circle [radius=0.2];
\draw (0.9,0.7) circle [radius=0.2];
\draw (0,1.3) circle [radius=0.2];
\draw  (-2.2,1) circle [radius=0.2];
\draw (1.7,1.7) circle [radius=0.2];
\draw (0,2.7) circle [radius=0.2];
\draw  (-1.3,0.5) circle [radius=0.2];

\draw (0,-0.3) circle [radius=0.2];
\draw (-0.7,-1) circle [radius=0.2];
\draw (0.7,-1.3) circle [radius=0.2];
\draw  (-2,-1.5) circle [radius=0.2];
\draw (1.7,-1.7) circle [radius=0.2];
\draw  (-1.6,-0.5) circle [radius=0.2];
\draw  (2.4,-0.3) circle [radius=0.2];

\draw (0,0)circle(2cm);
\end {tikzpicture}
\end{center}
\caption{The contours defining the interpolating transformation \eqref{defm1}. Around each of the
poles $\eta_k\in D_+$ and  $\hat{\eta}_k\in D_-$ we insert a small disk, oriented counterclockwise in $D_+$ and
clockwise in $D_-$, of fixed radius $\varrho$ sufficiently small such that the disks intersect neither each other nor the
jump contour $\Sigma$. If a pair of poles $\eta_k$, $\hat{\eta}_k$ satisfy ${\rm Re}[2\mathrm{i}t\theta(\eta_k)]=0$ or ${\rm Re}[2\mathrm{i}t\theta(\hat{\eta}_k)]=0$, we leave that pair uninterpolated.}
\label{sigma1}
\end{figure}

Consider the following contour
\begin{equation}\label{defm1}
\Sigma^{(1)}=\Sigma \cup \Sigma_{cir},
\end{equation}
where $\Sigma_{cir}=\underset{k\in \mathcal{N}/\Lambda}\cup\{z\in \mathbb{C}: |z-\eta_k|=\varrho \ \mbox{or} \ |z-\hat{\eta}_k|=\varrho \}$ depicted in Figure \ref{sigma1}.

Define the following transformation:
\begin{equation}\label{defm1}
m^{(1)}(z)=T(\infty)^{\sigma_{3}}m(z)G(z)T(z)^{-\sigma_{3}},
\end{equation}
which satisfies the following RH problem.
\begin{RHP}\label{rhp1}
Find a $2\times 2$ matrix-valued function $m^{(1)}(z):=m^{(1)}(x,t;z)$ such that
\begin{itemize}
    \item[*] $m^{(1)}(z)$ is meromorphic in $\mathbb{C}\backslash \Sigma^{(1)}$.
    \item[*] The boundary value $m^{(1)}_{\pm}(z)$ at $\Sigma^{(1)}$ satisfies the jump condition $m_+^{(1)}(z)=m_-^{(1)}(z)v^{(1)}(z)$, where
    \begin{align}\label{rhp1jump}
    v^{(1)}(z)=\left\{
        \begin{aligned}
    &\begin{bmatrix} 1 & -\tilde{\rho}(z)T^{2}(z)e^{2\mathrm{i}t\theta(z)} \\  0 & 1 \end{bmatrix}\begin{bmatrix} 1 & 0 \\  \rho(z) T^{-2}(z) e^{-2\mathrm{i}t\theta(z)} & 0 \end{bmatrix}, \quad z\in \tilde{\Gamma}\\
    &\begin{bmatrix} 1 & 0 \\  \frac{\rho(z)}{1-\rho(z)\tilde{\rho}(z)}T_{-}^{-2}(z)e^{-2\mathrm{i}t\theta(z)} & 1 \end{bmatrix}
    \begin{bmatrix} 1 & -\frac{\tilde{\rho}(z)}{1-\rho(z)\tilde{\rho}(z)}T_{+}^{2}(z)e^{2\mathrm{i}t\theta(z)} \\ 0  & 1 \end{bmatrix}, \quad z\in \Gamma,\\
     &\begin{bmatrix} 1 & 0 \\ -\frac{A[\eta_k]e^{-2\mathrm{i}t\theta(\eta_k)}}{z-\eta_k} & 1 \end{bmatrix}, \quad |z-\eta_k|=\varrho, \quad k\in \nabla/\Lambda,\\
    &\begin{bmatrix} 1 & -\frac{z-\eta_k}{A[\eta_k]e^{-2\mathrm{i}t\theta(\eta_k)}} \\ 0 & 1 \end{bmatrix}, \quad |z-\eta_k|=\varrho, \quad k\in \triangle/\Lambda,\\
    &\begin{bmatrix} 1 & -\frac{A[\hat{\eta}_k]e^{2\mathrm{i}t\theta(\hat{\eta}_k)}}{z-\hat{\eta}_k} \\ 0 & 1 \end{bmatrix}, \quad |z-\hat{\eta}_k|=\varrho, \quad k\in \nabla/\Lambda,\\
    &\begin{bmatrix} 1 & 0 \\ -\frac{z-\hat{\eta}_k}{A[\hat{\eta}_k]e^{-2\mathrm{i}t\theta(\hat{\eta}_k)}} & 1 \end{bmatrix}, \quad |z-\hat{\eta}_k|=\varrho, \quad k\in \triangle/\Lambda,\\
    &I, \quad elsewhere.
    \end{aligned}
        \right.
    \end{align}
    \item[*] Asymptotic behavior
    \begin{equation}\label{rhp1aym}
    m^{(1)}(z)=I+\mathcal{O}(z^{-1}), \quad  z\rightarrow\infty.
    \end{equation}
    \item[*] Residue conditions
    \begin{align}
    \underset{z=\eta_k}{\rm Res}m^{(1)}(z)=\left\{
    \begin{aligned}\label{rhp1resa}
    &\lim_{z\rightarrow\eta_{k}}m^{(1)}(z)\begin{bmatrix} 0 & 0 \\ A[\eta_{k}]T^{-2}(\eta_{k}) e^{-2it\theta(\eta_{k})} & 0 \end{bmatrix}, \quad k\in\nabla\cap\Lambda,\\
    &\lim_{z\rightarrow\eta_{k}}m^{(1)}(z)\begin{bmatrix} 0 & \frac{1}{A[\eta_{k}]}[(\frac{1}{T})'(\eta_{k})]^{-2}e^{2it\theta(\eta_{k})} \\ 0 & 0 \end{bmatrix}, \quad k\in\triangle\cap\Lambda,
    \end{aligned}
        \right.
    \end{align}
    \begin{align}\label{rhp1resb}
    \underset{z=\hat{\eta}_k}{\rm Res}m^{(1)}(z)=\left\{
    \begin{aligned}
    &\lim_{z\rightarrow\hat{\eta}_k}m^{(1)}(z)\begin{bmatrix} 0 & A[\hat{\eta}_k]T^{2}(\hat{\eta}_{k}) e^{2it\theta(\hat{\eta}_{k})} \\ 0 & 0 \end{bmatrix}, \quad k\in\nabla\cap\Lambda,\\
    &\lim_{z\rightarrow\hat{\eta}_k}m^{(1)}(z)\begin{bmatrix} 0 & 0 \\ \frac{1}{A[\hat{\eta}_k]}\frac{1}{[T'(\hat{\eta}_k)]^{2}} e^{-2it\theta(\hat{\eta}_{k})} & 0 \end{bmatrix}, \quad k\in\triangle\cap\Lambda.
    \end{aligned}
        \right.
    \end{align}
\end{itemize}
\end{RHP}
\subsection{Set up and decomposition of a mixed $\bar{\partial}$-RH problem}\label{mixedRHP}
In this section, we want to remove the jump from the original jump contour $\Sigma$ in such a way that the new problem takes advantage of the decay of $e^{2\mathrm{i}t\theta(z)}$ or $e^{-2\mathrm{i}t\theta(z)}$ for $z\notin\Sigma$. Additionally, we hope to open the lenses in such a way that the lenses are bounded away from the disks containing poles introduced in Figure \ref{sigma1}.
\subsubsection{Characteristic lines and some estimates for ${\rm Re}[2\mathrm{i}\theta(z)]$}
We fix an angle $\theta_{0}>0$ sufficiently small such that the set $\{z\in \mathbb{C}: \left| \frac{{\rm Re}z}{z}\right|>\cos{\theta_0}, \left| \frac{{\rm Re}z-\zeta_i}{z}\right|>\cos{\theta_0}, i=1,\cdots,6\}$ does not intersect the discrete spectrums set $\mathcal{Z}\cup\hat{\mathcal{Z}}$. For any $\xi<-6$, let
\begin{equation}
\phi(\xi)={\rm min}\left\{\theta_0, \frac{\pi}{4}\right\}.
\end{equation}

Since the phase function \eqref{phasefunc} has six critical points at $\zeta_i, i=1,2,\cdots,6$, our new contour
is chosen to be
\begin{equation}\label{Sigmajump}
\Sigma_{jum}=\underset{j=1,2,3,4}\cup \left(\Sigma_{0j} \cup\Big(\underset{i=1,2,\cdots,8} \cup \Sigma_{ij}\Big)\right)
\end{equation}
shown in Figure \ref{alljunpcontours}, which consists of rays of the form $\zeta_i+re^{\mathrm{i}\phi}$ with $r>0$ and other line segments or arcs.
\begin{lemma}\label{theta1}
Set $\xi=\frac{x}{t}$  and let $\xi<-6$. Then for $z=|z|e^{\mathrm{i}w}\in\Omega_{0j}, j=1,2,3,4$, the
phase $\theta(z)$ defined in \eqref{phasefunc} satisfies
\begin{equation}
\begin{split}
& {\rm Re}[2\mathrm{i}\theta(z)]\geq c|{\rm sin}w|(|z|^{-1}-|z|), \quad z\in\Omega_{01}, \Omega_{02},\\
& {\rm Re}[2\mathrm{i}\theta(z)]\leq -c|{\rm sin}w|(|z|^{-1}-|z|), \quad z\in\Omega_{03}, \Omega_{04},
\end{split}
\end{equation}
where $c=c(\xi)>0$.
\end{lemma}
\begin{proof}
We give a proof for $z=|z|e^{\mathrm{i}w}\in\Omega_{01}$, the others are similar.
\begin{equation}
{\rm Re}[2\mathrm{i}\theta(z)]=-{\rm sin}w(|z|-|z|^{-1})\left[\xi-3+(F^2(|z|)-1)(1+2{\rm cos}2w)\right],
\end{equation}
where $F(|z|)=|z|+|z|^{-1}$. Let
\begin{equation}
g(z)=\xi-3+(F^2(|z|)-1)(1+2{\rm cos}2w),
\end{equation}
From $g(|z|)=0$, we have
\begin{equation}
F^2(|z|)=1+\frac{3-\xi}{1+2{\rm cos}2w}:=\alpha>4.
\end{equation}
Further we have $F(|z|)=|z|+|z|^{-1}=\sqrt{\alpha}$, which leads to two solutions
\begin{equation}
|z|_1=\frac{\sqrt{\alpha}-\sqrt{\alpha-4}}{2},\quad |z|_2=\frac{\sqrt{\alpha}+\sqrt{\alpha-4}}{2}.
\end{equation}
It's easy to check that $h(|z|)=|z|^2-\sqrt{\alpha}|z|+1$ is decreasing in $(-\infty,|z|_1)$ and $F(|z|)>\sqrt{\alpha}$. Therefore,
\begin{equation}
{\rm Re}[2\mathrm{i}\theta(z)]>{\rm sin}w(|z|^{-1}-|z|)\left[\xi-3+(\alpha-1)(1+2{\rm cos}2w)\right]=0.
\end{equation}
\end{proof}
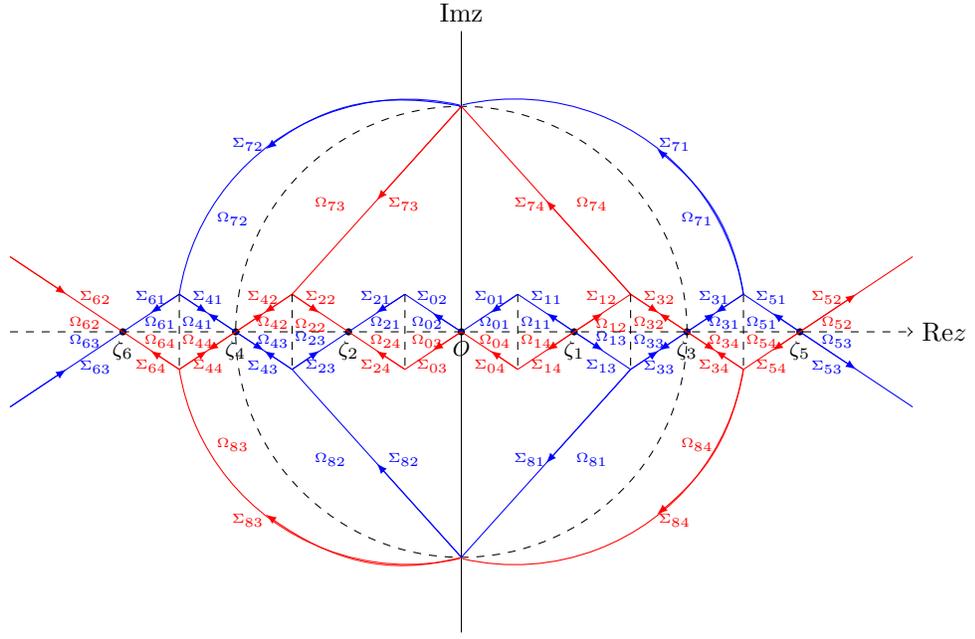
\begin{figure}[H]
        \begin{center}
        \begin{tikzpicture}[node distance=2cm]
          \draw [->,dashed](-6,0)--(6,0)node[right]{ \textcolor{black}{${\rm Re} z$}};;
          \draw [](0,-4)--(0,4)  node[above, scale=1] {{\rm Im}z};
          \draw [->,dashed](0,0)circle(3cm);

          \node  [below]  at (1.5,0) {\footnotesize $\zeta_{1}$};
          \node  [below]  at (-1.5,0) {\footnotesize $\zeta_{2}$};
          \node  [below] at (3,0) {\footnotesize $\zeta_{3}$};
          \node  [below]  at (-3,0) {\footnotesize $\zeta_{4}$};
          \node  [below]  at (4.5,0) {\footnotesize $\zeta_{5}$};
          \node  [below] at (-4.5,0) {\footnotesize $\zeta_{6}$};
          \node  [below]  at (0,0) {\footnotesize $O$};

          \draw[fill] (1.5,0) circle [radius=0.04];
          \draw[fill] (4.5,0) circle [radius=0.04];
          \draw[fill] (3,0) circle [radius=0.04];
          \draw[fill] (-4.5,0) circle [radius=0.04];
          \draw[fill] (-1.5,0) circle [radius=0.04];
          \draw[fill] (-3,0) circle [radius=0.04];
          \draw[fill] (0,0) circle [radius=0.04];

          \draw [color=blue](0,0)--(0.75,0.5)  node[right, scale=1] {};
          \draw [color=blue](0.75,0.5)--(1.5,0)  node[right, scale=1] {};
          \draw [color=blue](1.5,0)--(2.25,-0.5)  node[right, scale=1] {};
          \draw [color=blue](3.75,0.5)--(4.5,0) node[right, scale=1] {};
          \draw [color=blue](4.5,0)--(6,-1) node[right, scale=1] {};

          \draw [color=blue](0,0)--(-0.75,0.5)  node[right, scale=1] {};
          \draw [color=blue](-0.75,0.5)--(-1.5,0)  node[right, scale=1] {};
          \draw [color=blue](-1.5,0)--(-2.25,-0.5)  node[right, scale=1] {};
          \draw [color=blue](-3.75,0.5)--(-4.5,0) node[right, scale=1] {};
          \draw [color=blue](-4.5,0)--(-6,-1) node[right, scale=1] {};

          \draw [color=red](0,0)--(0.75,-0.5)  node[right, scale=1] {};
          \draw [color=red](0.75,-0.5)--(1.5,0)  node[right, scale=1] {};
          \draw [color=red](1.5,0)--(2.25,0.5)  node[right, scale=1] {};
          \draw [color=red](3.75,-0.5)--(4.5,0) node[right, scale=1] {};
          \draw [color=red](4.5,0)--(6,1) node[right, scale=1] {};

          \draw [color=red](0,0)--(-0.75,-0.5)  node[right, scale=1] {};
          \draw [color=red](-0.75,-0.5)--(-1.5,0)  node[right, scale=1] {};
          \draw [color=red](-1.5,0)--(-2.25,0.5)  node[right, scale=1] {};
          \draw [color=red](-3.75,-0.5)--(-4.5,0) node[right, scale=1] {};
          \draw [color=red](-4.5,0)--(-6,1) node[right, scale=1] {};

          \draw [color=red](0,3)--(2.25,0.5)  node[right, scale=1] {};
          \draw [color=red](0,3)--(-2.25,0.5)  node[right, scale=1] {};
          \draw [color=blue](0,-3)--(2.25,-0.5)  node[right, scale=1] {};
          \draw [color=blue](0,-3)--(-2.25,-0.5)  node[right, scale=1] {};

          \draw [color=red](2.25,0.5)--(3.75,-0.5)  node[right, scale=1] {};
          \draw [color=red](-2.25,0.5)--(-3.75,-0.5)  node[right, scale=1] {};
          \draw [color=blue](2.25,-0.5)--(3.75,0.5)  node[right, scale=1] {};
          \draw [color=blue](-2.25,-0.5)--(-3.75,0.5)  node[right, scale=1] {};

          \draw[color=blue] (3.75,0.5) arc (9:103:3.08);
          \draw[color=blue] (-3.75,0.5) arc (171:77:3.076);
          \draw[color=red] (-3.75,-0.5) arc (189:283:3.076);
          \draw[color=red] (3.75,-0.5) arc (351:257:3.08);

          \draw [color=blue] [-latex]  (0.75,0.5) to (0.385,0.25);
          \draw [color=blue] [-latex]  (1.5,0) to (1.135,0.25);
          \draw [color=blue] [-latex]  (1.5,0) to (1.885,-0.25);
          \draw [color=blue] [-latex]  (4.5,0) to (4.135,0.25);
          \draw [color=blue] [-latex]  (4.5,0) to (5.25,-0.5);

          \draw [color=red] [-latex]  (0.75,-0.5) to (0.385,-0.25);
          \draw [color=red] [-latex]  (1.5,0) to (1.135,-0.25);
          \draw [color=red] [-latex]  (1.5,0) to (1.885,0.25);
          \draw [color=red] [-latex]  (4.5,0) to (4.135,-0.25);
          \draw [color=red] [-latex]  (4.5,0) to (5.25,0.5);

          \draw [color=blue] [-latex]  (0,0) to (-0.385,0.25);
          \draw [color=blue] [-latex]  (-0.75,0.5) to (-1.135,0.25);
          \draw [color=blue] [-latex]  (-2.25,-0.5) to (-1.885,-0.25);
          \draw [color=blue] [-latex]  (-3.75,0.5) to (-4.135,0.25);
          \draw [color=blue] [-latex]  (-6,-1) to (-5.25,-0.5);

          \draw [color=red] [-latex]  (0,0) to (-0.385,-0.25);
          \draw [color=red] [-latex]  (-0.75,-0.5) to (-1.135,-0.25);
          \draw [color=red] [-latex]  (-2.25,0.5) to (-1.885,0.25);
          \draw [color=red] [-latex]  (-3.75,-0.5) to (-4.135,-0.25);
          \draw [color=red] [-latex]  (-6,1) to (-5.25,0.5);

          \draw [color=red] [-latex]  (2.25,0.5) to (1.125,1.75);
          \draw [color=red] [-latex]  (0,3) to (-1.125,1.75);
          \draw [color=blue] [-latex]  (2.25,-0.5) to (1.125,-1.75);
          \draw [color=blue] [-latex]  (0,-3) to (-1.125,-1.75);

          \draw [color=red] [-latex]  (3,0) to (2.625,0.25);
          \draw [color=red] [-latex]  (3,0) to (3.375,-0.25);
          \draw [color=blue] [-latex]  (3,0) to (2.625,-0.25);
          \draw [color=blue] [-latex]  (3,0) to (3.375,0.25);
          \draw [color=red] [-latex]  (2.25,0.5) to (2.625,0.25);
          \draw [color=red] [-latex]  (3.75,-0.5) to (3.375,-0.25);
          \draw [color=blue] [-latex]  (2.25,-0.5) to (2.625,-0.25);
          \draw [color=blue] [-latex]  (3.75,0.5) to (3.375,0.25);

          \draw [color=red] [-latex]  (-3,0) to (-2.625,0.25);
          \draw [color=red] [-latex]  (-3,0) to (-3.375,-0.25);
          \draw [color=blue] [-latex]  (-3,0) to (-2.625,-0.25);
          \draw [color=blue] [-latex]  (-3,0) to (-3.375,0.25);
          \draw [color=red] [-latex]  (-2.25,0.5) to (-2.625,0.25);
          \draw [color=red] [-latex]  (-3.75,-0.5) to (-3.375,-0.25);
          \draw [color=blue] [-latex]  (-2.25,-0.5) to (-2.625,-0.25);
          \draw [color=blue] [-latex]  (-3.75,0.5) to (-3.375,0.25);

          \draw [color=blue] [-latex]  (3.75,0.5) to  [out=99, in=320]  (2.6,2.43);
          \draw [color=blue] [-latex]  (0,3) to  [out=169, in=37]  (-2.6,2.43);
          \draw [color=red] [-latex]  (0,-3) to  [out=195, in=325]  (-2.6,-2.43);
          \draw [color=red] [-latex]  (3.75,-0.5) to  [out=260, in=40]  (2.6,-2.43);

          \draw[dashed](-0.75,0.5)--(-0.75,-0.5);
          \draw[dashed](0.75,0.5)--(0.75,-0.5);
          \draw[dashed](-2.25,0.5)--(-2.25,-0.5);
          \draw[dashed](2.25,0.5)--(2.25,-0.5);
          \draw[dashed](-3.75,0.5)--(-3.75,-0.5);
          \draw[dashed](3.75,0.5)--(3.75,-0.5);

         \node [above]  at (0.385,0.25) {\textcolor{blue}{\tiny $\Sigma_{01}$}};
         \node [above]  at (-0.385,0.25) {\textcolor{blue}{\tiny $\Sigma_{02}$}};
         \node [below]  at (-0.385,-0.25) {\textcolor{red}{\tiny $\Sigma_{03}$}};
         \node [below]  at (0.385,-0.25) {\textcolor{red}{\tiny $\Sigma_{04}$}};

         \node [below]  at (0.45,0.33) {\textcolor{blue}{\tiny $\Omega_{01}$}};
         \node [below]  at (-0.45,0.33) {\textcolor{blue}{\tiny $\Omega_{02}$}};
         \node [above]  at (-0.45,-0.33) {\textcolor{red}{\tiny $\Omega_{03}$}};
         \node [above]  at (0.45,-0.33) {\textcolor{red}{\tiny $\Omega_{04}$}};

         \node [above]  at (1.135,0.25) {\textcolor{blue}{\tiny $\Sigma_{11}$}};
         \node [below]  at (1.135,-0.25) {\textcolor{red}{\tiny $\Sigma_{14}$}};
         \node [above]  at (1.865,0.25) {\textcolor{red}{\tiny $\Sigma_{12}$}};
         \node [below]  at (1.865,-0.25) {\textcolor{blue}{\tiny $\Sigma_{13}$}};

         \node [below]  at (1,0.33) {\textcolor{blue}{\tiny $\Omega_{11}$}};
         \node [above]  at (1,-0.33) {\textcolor{red}{\tiny $\Omega_{14}$}};
         \node [below]  at (2,0.3) {\textcolor{red}{\tiny $\Omega_{12}$}};
         \node [above]  at (2,-0.3) {\textcolor{blue}{\tiny $\Omega_{13}$}};

         \node [above]  at (3.365,0.25) {\textcolor{blue}{\tiny $\Sigma_{31}$}};
         \node [above]  at (2.635,0.25) {\textcolor{red}{\tiny $\Sigma_{32}$}};
         \node[below]  at (2.635,-0.26) {\textcolor{blue}{\tiny $\Sigma_{33}$}};
         \node [below]  at (3.365,-0.25) {\textcolor{red}{\tiny $\Sigma_{34}$}};

         \node [below]  at (3.5,0.33) {\textcolor{blue}{\tiny $\Omega_{31}$}};
         \node [below]  at (2.5,0.33) {\textcolor{red}{\tiny $\Omega_{32}$}};
         \node[above]  at (2.5,-0.33) {\textcolor{blue}{\tiny $\Omega_{33}$}};
         \node [above]  at (3.5,-0.33) {\textcolor{red}{\tiny $\Omega_{34}$}};

         \node [above]  at (4.125,0.25) {\textcolor{blue}{\tiny $\Sigma_{51}$}};
         \node [above]  at (4.865,0.25) {\textcolor{red}{\tiny $\Sigma_{52}$}};
         \node[below]  at (4.865,-0.26) {\textcolor{blue}{\tiny $\Sigma_{53}$}};
         \node [below]  at (4.125,-0.25) {\textcolor{red}{\tiny $\Sigma_{54}$}};

         \node [below]  at (4,0.33) {\textcolor{blue}{\tiny $\Omega_{51}$}};
         \node [below]  at (5,0.33) {\textcolor{red}{\tiny $\Omega_{52}$}};
         \node[above]  at (5,-0.33) {\textcolor{blue}{\tiny $\Omega_{53}$}};
         \node [above]  at (4,-0.33) {\textcolor{red}{\tiny $\Omega_{54}$}};

        \node [above]  at (-1.135,0.25) {\textcolor{blue}{\tiny $\Sigma_{21}$}};
        \node [below]  at (-1.135,-0.25) {\textcolor{red}{\tiny $\Sigma_{24}$}};
        \node [above]  at (-1.865,0.25) {\textcolor{red}{\tiny $\Sigma_{22}$}};
        \node [below]  at (-1.865,-0.25) {\textcolor{blue}{\tiny $\Sigma_{23}$}};

        \node [below]  at (-1,0.33) {\textcolor{blue}{\tiny $\Omega_{21}$}};
        \node [above]  at (-1,-0.33) {\textcolor{red}{\tiny $\Omega_{24}$}};
        \node [below]  at (-2,0.3) {\textcolor{red}{\tiny $\Omega_{22}$}};
        \node [above]  at (-2,-0.3) {\textcolor{blue}{\tiny $\Omega_{23}$}};

        \node [above]  at (-3.365,0.25) {\textcolor{blue}{\tiny $\Sigma_{41}$}};
        \node [above]  at (-2.635,0.25) {\textcolor{red}{\tiny $\Sigma_{42}$}};
        \node[below]  at (-2.635,-0.26) {\textcolor{blue}{\tiny $\Sigma_{43}$}};
        \node [below]  at (-3.365,-0.25) {\textcolor{red}{\tiny $\Sigma_{44}$}};

        \node [below]  at (-3.5,0.33) {\textcolor{blue}{\tiny $\Omega_{41}$}};
        \node [below]  at (-2.5,0.33) {\textcolor{red}{\tiny $\Omega_{42}$}};
        \node[above]  at (-2.5,-0.33) {\textcolor{blue}{\tiny $\Omega_{43}$}};
        \node [above]  at (-3.5,-0.33) {\textcolor{red}{\tiny $\Omega_{44}$}};

        \node [above]  at (-4.125,0.25) {\textcolor{blue}{\tiny $\Sigma_{61}$}};
        \node [above]  at (-4.865,0.25) {\textcolor{red}{\tiny $\Sigma_{62}$}};
        \node[below]  at (-4.865,-0.26) {\textcolor{blue}{\tiny $\Sigma_{63}$}};
        \node [below]  at (-4.125,-0.25) {\textcolor{red}{\tiny $\Sigma_{64}$}};

         \node [below]  at (-4,0.33) {\textcolor{blue}{\tiny $\Omega_{61}$}};
         \node [below]  at (-5,0.33) {\textcolor{red}{\tiny $\Omega_{62}$}};
         \node[above]  at (-5,-0.33) {\textcolor{blue}{\tiny $\Omega_{63}$}};
         \node [above]  at (-4,-0.33) {\textcolor{red}{\tiny $\Omega_{64}$}};

          \node [right]  at (2.5,2.5) {\textcolor{blue}{\tiny $\Sigma_{71}$}};
          \node [left]  at (-2.5,2.5) {\textcolor{blue}{\tiny $\Sigma_{72}$}};
          \node [right]  at (-1.1,1.7) {\textcolor{red}{\tiny $\Sigma_{73}$}};
          \node [left]  at (1.25,1.7) {\textcolor{red}{\tiny $\Sigma_{74}$}};

          \node [right]  at (2.8,1.5) {\textcolor{blue}{\tiny $\Omega_{71}$}};
          \node [left]  at (-2.7,1.5) {\textcolor{blue}{\tiny $\Omega_{72}$}};
          \node [left]  at (-1.4,1.7) {\textcolor{red}{\tiny $\Omega_{73}$}};
          \node [right]  at (1.4,1.7) {\textcolor{red}{\tiny $\Omega_{74}$}};

          \node [right]  at (2.5,-2.5) {\textcolor{red}{\tiny $\Sigma_{84}$}};
          \node [left]  at (-2.5,-2.5) {\textcolor{red}{\tiny $\Sigma_{83}$}};
          \node [right]  at (-1.1,-1.7) {\textcolor{blue}{\tiny $\Sigma_{82}$}};
          \node [left]  at (1.25,-1.7) {\textcolor{blue}{\tiny $\Sigma_{81}$}};

          \node [right]  at (2.8,-1.5) {\textcolor{red}{\tiny $\Omega_{84}$}};
          \node [left]  at (-2.7,-1.5) {\textcolor{red}{\tiny $\Omega_{83}$}};
          \node [left]  at (-1.4,-1.7) {\textcolor{blue}{\tiny $\Omega_{82}$}};
          \node [right]  at (1.4,-1.7) {\textcolor{blue}{\tiny $\Omega_{81}$}};

          \end {tikzpicture}
          \caption{\footnotesize Deformation from $\Sigma$ to $\Sigma_{jump}$. The blue curves are the opening contours in region $\{z\in\mathbb{C}: |e^{-2\mathrm{i}t\theta(z)}|\rightarrow0\}$ while the red curves are the opening contours in region $\{z\in\mathbb{C}: |e^{2\mathrm{i}t\theta(z)}|\rightarrow0\}$. These arrows represent directions of jump contours.}
        \label{alljunpcontours}
        \end{center}
\end{figure}
\begin{corollary}
For $z=|z|e^{\mathrm{i}w}=u+\mathrm{i}v\in\Omega_{0j}, \ j=1,2,3,4$,
\begin{equation}
\begin{split}
& {\rm Re}[2\mathrm{i}\theta(z)]\geq c|v|, \quad z\in\Omega_{01}, \Omega_{02},\\
& {\rm Re}[2\mathrm{i}\theta(z)]\leq -c|v|, \quad z\in\Omega_{03}, \Omega_{04},
\end{split}
\end{equation}
where $c=c(\xi)>0$.
\end{corollary}
\begin{lemma}\label{theta2}
For $z\in\Omega_{ij}, i=1,2,5,6, j=1,2,3,4$,
\begin{equation}
\begin{split}
& {\rm Re}[2\mathrm{i}\theta(z)]\geq cv^2, \quad z\in\Omega_{i1}, \Omega_{i3}, i=1,2,5,6,\\
& {\rm Re}[2\mathrm{i}\theta(z)]\leq -cv^2, \quad z\in\Omega_{i2}, \Omega_{i4}, i=1,2,5,6,
\end{split}
\end{equation}
where $c=c(\xi)>0$.
\end{lemma}
\begin{proof}
We take $z\in\Omega_{12}$ as an example. For $z=\zeta_1+u+\mathrm{i}v\in\Omega_{12}$, $0<v<1$, $\zeta_1<|z|<\frac{1}{2}\sqrt{(\zeta_1+\zeta_3)^2+(\zeta_1-\zeta_3)^2\tan{\phi}}$,  we have
\begin{equation}
\begin{split}
{\rm Re}[2\mathrm{i}\theta(z)]&=(|z|^{-2}-1)v\left[\xi-3+3(1+|z|^2+|z|^{-2})-4v^2(1+|z|^{-2}+|z|^{-4})\right]\\
&\lesssim c\left[\xi-3+3(1+|z|^2+|z|^{-2})-4v^2(1+|z|^{-2}+|z|^{-4})\right].
\end{split}
\end{equation}
Let $\tau=|z|^2\in\left(\zeta_1^2,\frac{(\zeta_1+1)^2}{4}\right)$ and define
\begin{equation}
h(\tau)=\xi-3+3(1+\tau+\tau^{-1})-4v^2(1+\tau^{-1}+\tau^{-2}).
\end{equation}
Due to $v\ll\tau$, then
\begin{equation}
h'(\tau)=3-3\tau^{-2}+4v^2\tau^{-2}+8v^2\tau^{-3}\leq 0
\end{equation}
and $h(\tau)$ is decreasing in $\left(\zeta_1^2,\frac{(\zeta_1+1)^2}{4}\right)$. Thus,
\begin{equation}\label{h}
\begin{split}
h(\tau)&\leq h(\zeta_1^2)=\xi-3+3(1+\zeta_1^2+\zeta_1^{-2})-4v^2(1+\zeta_1^{-2}+\zeta_1^{-4})\\
&\overset{\xi_1=\frac{1}{\xi_6}}{=}\xi-3+3(1+\zeta_1^2+\zeta_6^2)-4v^2(1+\zeta_6^2+\zeta_6^4).
\end{split}
\end{equation}
Recall $\theta'(\zeta_6)=0$, we obtain
\begin{equation}
3\zeta_6^4+\xi\zeta_6^2+3=0
\end{equation}
and then
\begin{equation}\label{xi}
\xi-3=-3(\zeta_6^{-2}+\zeta_6^2+1)=-3(\zeta_1^2+\zeta_6^2+1).
\end{equation}
Substituting \eqref{xi} into \eqref{h}, it comes to
\begin{equation}
h(\tau)\leq h(\zeta_1^2)=-4v^2(1+\zeta_6^2+\zeta_6^4).
\end{equation}
As a consequence,
\begin{equation}
{\rm Re}[2\mathrm{i}\theta(z)]\leq -cv^2.
\end{equation}
\end{proof}
\begin{lemma}\label{theta3}
For $z\in\Omega_{ij}, i=3,4,7,8, j=1,2,3,4$,
\begin{equation}
\begin{split}
& {\rm Re}[2\mathrm{i}\theta(z)]\geq c\left|1-|z|^{-2}\right|v^2, \quad z\in\Omega_{i1}, \Omega_{i3}, \Omega_{k1}, \Omega_{k2}, i=3,4, k=7,8,\\
& {\rm Re}[2\mathrm{i}\theta(z)]\leq -c\left|1-|z|^{-2}\right|v^2, \quad z\in\Omega_{i2}, \Omega_{i4}, \Omega_{k3}, \Omega_{k4}, i=3,4, k=7,8,
\end{split}
\end{equation}
where $c=c(\xi)>0$.
\end{lemma}
\begin{proof}
We take $z\in\Omega_{74}$ as an example. For $z=u+\mathrm{i}v\in\Omega_{74}$, we have $\zeta_1<|z|<1$ and $0<v<1$, thus
\begin{equation}
{\rm Re}[2\mathrm{i}\theta(z)]\leq c(|z|^{-2}-1)\left[\xi-3+3(1+|z|^2+|z|^{-2})-4v^2(1+|z|^{-2}+|z|^{-4})\right].
\end{equation}
Let $f(|z|)=|z|^2+|z|^{-2}$, it is obvious that $f(|z|)$ is decreasing and $f(|z|)<f(\zeta_1)=\zeta_1^2+\zeta_1^{-2}$. Recall $\theta'(\zeta_1)=0$, we obtain
\begin{equation}
\xi-3=-3(\zeta_1^2+\zeta_1^{-2}+1).
\end{equation}
Therefore,
\begin{equation}
\begin{split}
{\rm Re}[2\mathrm{i}\theta(z)]&\leq c(|z|^{-2}-1)\left[-3(1+\zeta_1^2+\zeta_1^{-2})+3(1+\zeta_1^2+\zeta_1^{-2})-4v^2(1+|z|^{-2}+|z|^{-4})\right]\\
&\leq -c(\xi)(|z|^{-2}-1)v^2.
\end{split}
\end{equation}
\end{proof}
\subsubsection{Opening lenses}\label{422}
The estimates in Lemma \ref{theta1}, \ref{theta2}, \ref{theta3} suggest that we should open lenses using (modified versions of) factorization for $z\in\Gamma$ and for $z\in\tilde{\Gamma}$ shown in \eqref{rhp1jump}. To do so, we need to define extensions of the off-diagonal entries of factorization matrices off $\Sigma$, which is the content of this subsection.

To be brief, we introduce the following notations and functions:
\begin{equation}
\begin{split}
&l_0^+=\left(0,\frac{\zeta_1}{2}\right), \quad l_0^-=\left(\frac{\zeta_2}{2},0\right),\\
&l_i^+=\left(\zeta_i,\frac{\zeta_i+\zeta_{i+2}}{2}\right), \quad l_i^-=\left(\frac{\zeta_i}{2},\zeta_i\right), \quad i=1,2,\\
&l_i^+=\left(\zeta_i,\frac{\zeta_i+\zeta_{i+2}}{2}\right), \quad l_i^-=\left(\frac{\zeta_{i-2}+\zeta_i}{2},\zeta_i\right), \quad i=3,4,\\
&l_5^+=\left(\zeta_5,+\infty\right), \quad l_6^+=\left(-\infty,\zeta_6\right), \quad l_i^-=\left(\frac{\zeta_{i-2}+\zeta_i}{2},\zeta_i\right), \quad i=5,6,\\
&\gamma_k=\left\{z\in \mathbb{C}: z=e^{\mathrm{i}w}, \frac{(k-1)\pi}{2}\leq w\leq\frac{k\pi}{2}\right\}, \quad k=1,2,3,4,\\
&r_1(z)=\rho(z), \quad r_2(z)=\tilde{\rho}(z),\\
&r_3(z)=\frac{\rho(z)}{1-\rho(z)\tilde{\rho}(z)}, \quad r_4(z)=\frac{\tilde{\rho}(z)}{1-\rho(z)\tilde{\rho}(z)}.
\end{split}
\end{equation}
Further, we choose $\mathcal{R}^{(2)}(z):=\mathcal{R}^{(2)}(z;\xi)$ as:
\begin{align}
\mathcal{R}^{(2)}(z)=\left\{
        \begin{aligned}
    &\begin{bmatrix} 1 & 0 \\  R_{ij}e^{-2\mathrm{i}t\theta} & 1 \end{bmatrix}, \quad z\in\Omega_{ij}, \quad i=0,7,8, j=1,2,\\
    &\begin{bmatrix} 1 & -R_{ij}e^{2\mathrm{i}t\theta} \\ 0 & 1 \end{bmatrix}^{-1}, \quad z\in\Omega_{ij}, \quad i=0,7,8, j=3,4,\\
    &\begin{bmatrix} 1 & 0 \\ R_{i1}e^{-2\mathrm{i}t\theta} & 1 \end{bmatrix}, \quad z\in\Omega_{i1}, \quad   i=1,2,\cdots,6,\\
    &\begin{bmatrix} 1 & -R_{i2}e^{2\mathrm{i}t\theta} \\ 0 & 1 \end{bmatrix}^{-1}, \quad z\in\Omega_{i2}, \quad   i=1,2,\cdots,6,\\
    &\begin{bmatrix} 1 & 0 \\ R_{i3}e^{-2\mathrm{i}t\theta} & 1 \end{bmatrix}, \quad z\in\Omega_{i3}, \quad   i=1,2,\cdots,6,\\
    &\begin{bmatrix} 1 & -R_{i4}e^{2\mathrm{i}t\theta} \\ 0 & 1 \end{bmatrix}^{-1}, \quad z\in\Omega_{i4}, \quad   i=1,2,\cdots,6,\\
    &I, \quad elsewhere,
    \end{aligned}
        \right.
\label{mu}
\end{align}
where the functions $R_{0j}$, $R_{ij}$ are defined as the following two propositions.
\begin{proposition}\label{R0j}
$R_{0j}: \overline{\Omega}_{0j}\rightarrow \mathbb{C}$ are continuous on $\overline{\Omega}_{0j}, j=1,2,3,4$. Their boundary values are as follows:
\begin{align}
   R_{0j}(z)=
   &\left\{
        \begin{aligned}
    &r_3(z)T_{-}^{-2}(z), \quad z\in l_0^+, l_0^-, \\
    &0, \quad z\in\Sigma_{0j}, \quad j=1,2,
    \end{aligned}
        \right.\\
   R_{0j}(z)=
    &\left\{
        \begin{aligned}
    &r_4(z)T_{+}^{2}(z), \quad z\in  l_0^+, l_0^-,\\
    &0, \quad z\in\Sigma_{0j}, \quad j=3,4.
    \end{aligned}
        \right.
\end{align}
Moreover, $R_{0j}$ have following property:
\begin{equation}
\begin{split}
& |\bar{\partial}R_{0j}(z)|\lesssim |z|^{-\frac{1}{2}}+\left|r_3'({\rm Re} z)\right|, \quad j=1,2,\\
& |\bar{\partial}R_{0j}(z)|\lesssim |z|^{-\frac{1}{2}}+\left|r_4'({\rm Re} z)\right|, \quad j=3,4.
\end{split}
\end{equation}
\end{proposition}
\begin{proof}
    The proof is the similar with \cite[Lemma 6.5]{CJ}.
\end{proof}

\begin{proposition}\label{estRij}
$R_{ij}: \overline{\Omega}_{ij}\rightarrow \mathbb{C}$, $i=1,2,\cdots,8, j=1,2,3,4$ are continuous on $\overline{\Omega}_{ij}$ with boundary values:
\begin{itemize}
\item[(\romannumeral1)]
For $i=1,2$,
\begin{align}
&R_{i1}(z)=\left\{
        \begin{aligned}
    &r_3(z)T_-^{-2}(z), \quad z\in l_i^-,\\
    &r_3(\zeta_i)T_i^{-2}(\zeta_i)(z-\zeta_i)^{2\mathrm{i}\nu(\zeta_{i})}, \quad z\in\Sigma_{i1},
    \end{aligned}
        \right.
\end{align}
\begin{align}
&R_{i2}(z)=\left\{
        \begin{aligned}
    &r_2(z)T^{2}(z), \quad z\in l_i^+,\\
    &r_2(\zeta_i)T_i^{2}(\zeta_i)(z-\zeta_{i})^{-2\mathrm{i}\nu(\zeta_{i})}, \quad z\in\Sigma_{i2},
    \end{aligned}
        \right.\\
&R_{i3}(z)=\left\{
        \begin{aligned}
    &r_1(z)T^{-2}(z), \quad z\in l_i^+,\\
    &r_1(\zeta_i)T_i^{-2}(\zeta_{i})(z-\zeta_{i})^{2\mathrm{i}\nu(\zeta_{i})}, \quad z\in\Sigma_{i3},
    \end{aligned}
        \right.\\
&R_{i4}(z)=\left\{
        \begin{aligned}
    &r_4(z)T_+^{2}(z), \quad z\in l_i^-,\\
    &r_4(\zeta_{i})T_i^{2}(\zeta_{i})(z-\zeta_{i})^{-2\mathrm{i}\nu(\zeta_{i})}, \quad z\in\Sigma_{i4}.
    \end{aligned}
        \right.
\end{align}
\item[(\romannumeral2)]
For $i=3,4$,
\begin{align}
&R_{ij}(z)=\left\{
        \begin{aligned}
    &r_1(z)T^{-2}(z), \quad z\in l_i^-, l_i^+,\\
    &r_1(\zeta_i)T^{-2}(\zeta_i), \quad z\in\Sigma_{ij}, j=1,3,
    \end{aligned}
        \right.\\
&R_{ij}(z)=\left\{
        \begin{aligned}
    &r_2(z)T^{2}(z), \quad z\in l_i^-, l_i^+,\\
    &r_2(\zeta_{i})T^{2}(\zeta_{i}), \quad z\in\Sigma_{ij}, j=2,4.
    \end{aligned}
        \right.
\end{align}
\item[(\romannumeral3)]
For $i=5,6$,
\begin{align}
&R_{i1}(z)=\left\{
        \begin{aligned}
    &r_1(z)T^{-2}(z), \quad z\in l_i^-,\\
    &r_1(\zeta_i)T_i^{-2}(\zeta_i)(z-\zeta_i)^{2\mathrm{i}\nu(\zeta_{i})}, \quad z\in\Sigma_{i1},
    \end{aligned}
        \right.\\
&R_{i2}(z)=\left\{
        \begin{aligned}
    &r_4(z)T_+^{2}(z), \quad z\in l_i^+,\\
    &r_4(\zeta_i)T_i^{2}(\zeta_i)(z-\zeta_{i})^{-2\mathrm{i}\nu(\zeta_{i})}, \quad z\in\Sigma_{i2},
    \end{aligned}
        \right.\\
&R_{i3}(z)=\left\{
        \begin{aligned}
    &r_3(z)T_-^{-2}(z), \quad z\in l_i^+,\\
    &r_3(\zeta_i)T_i^{-2}(\zeta_{i})(z-\zeta_{i})^{2\mathrm{i}\nu(\zeta_{i})}, \quad z\in\Sigma_{i3},
    \end{aligned}
        \right.\\
&R_{i4}(z)=\left\{
        \begin{aligned}
    &r_2(z)T^{2}(z), \quad z\in l_i^-,\\
    &r_2(\zeta_{i})T_i^{2}(\zeta_{i})(z-\zeta_{i})^{-2\mathrm{i}\nu(\zeta_{i})}, \quad z\in\Sigma_{i4}.
    \end{aligned}
        \right.
\end{align}
\item[(\romannumeral4)]
For $i=7,8$,
\begin{align}
&R_{i1}(z)=\left\{
        \begin{aligned}
    &r_1(z)T^{-2}(z), \quad z\in\gamma_1, \gamma_4,\\
    &r_1(1)T^{-2}(1)(1-\chi_{\mathcal{Z}}(z)), \quad z\in\Sigma_{i1},
    \end{aligned}
        \right.\\
&R_{i2}(z)=\left\{
        \begin{aligned}
    &r_1(z)T^{-2}(z), \quad z\in\gamma_2, \gamma_3,\\
    &r_1(-1)T^{-2}(-1)(1-\chi_{\mathcal{Z}}(z)), \quad z\in\Sigma_{i2},
    \end{aligned}
        \right.\\
&R_{i3}(z)=\left\{
        \begin{aligned}
    &r_2(z)T^2(z), \quad z\in\gamma_2, \gamma_3,\\
    &r_2(-1)T^2(-1)(1-\chi_{\mathcal{Z}}(z)), \quad z\in\Sigma_{i3},
    \end{aligned}
        \right.\\
&R_{i4}(z)=\left\{
        \begin{aligned}
    &r_2(z)T^2(z), \quad z\in\gamma_1, \gamma_4,\\
    &r_2(1)T^2(1)(1-\chi_{\mathcal{Z}}(z)), \quad z\in\Sigma_{i4},
    \end{aligned}
        \right.
\end{align}
where
\begin{align}
\chi_{\mathcal{Z}}(z)=\left\{
        \begin{aligned}
    &1, \quad {\rm dist(z,\mathcal{Z}\cup\mathcal{\hat{Z}})}<\varrho/3\\
    &0, \quad {\rm dist(z,\mathcal{Z}\cup\mathcal{\hat{Z}})}>2\varrho/3.
    \end{aligned}
        \right.
\end{align}
\end{itemize}
And $R_{ij}(z), i=1,2,\cdots,8, \ j=1,2,3,4$ have following properties:
\begin{equation}\label{4.52}
\begin{split}
&|R_{ij}(z)|\leq c_1+c_2|1+z^{2}|^{-\frac{1}{4}}, \quad z\in\Omega_{ij},\\
& |\bar{\partial}R_{ij}(z)|\leq c_1+c_2|z-\zeta_i|^{\frac{1}{2}}+c_3|z-\zeta_i|^{-\alpha_i}, \quad z\in\Omega_{ij},
\end{split}
\end{equation}
where
\begin{align}
\alpha_i=\left\{
        \begin{aligned}
    &\frac{1}{2}+{\rm Im}\nu(\zeta_i), \quad 0<{\rm Im}\nu(\zeta_i)<\frac{1}{2},\\
    &\frac{1}{2}, \quad -\frac{1}{2}<{\rm Im}\nu(\zeta_i)\leq0.
    \end{aligned}
        \right.
\end{align}
Moreover, when $z\rightarrow\pm \mathrm{i}$
\begin{equation}\label{453}
\begin{split}
& |\bar{\partial}R_{ij}(z)|\leq c|z-\mathrm{i}|, \quad z\in\Omega_{7j}, \quad j=1,2,3,4,\\
& |\bar{\partial}R_{ij}(z)|\leq c|z+\mathrm{i}|, \quad z\in\Omega_{8j}, \quad j=1,2,3,4.
\end{split}
\end{equation}
\end{proposition}
\begin{proof}
    The proof is the similar with \cite[Lemma 4.1]{BJ}.
\end{proof}

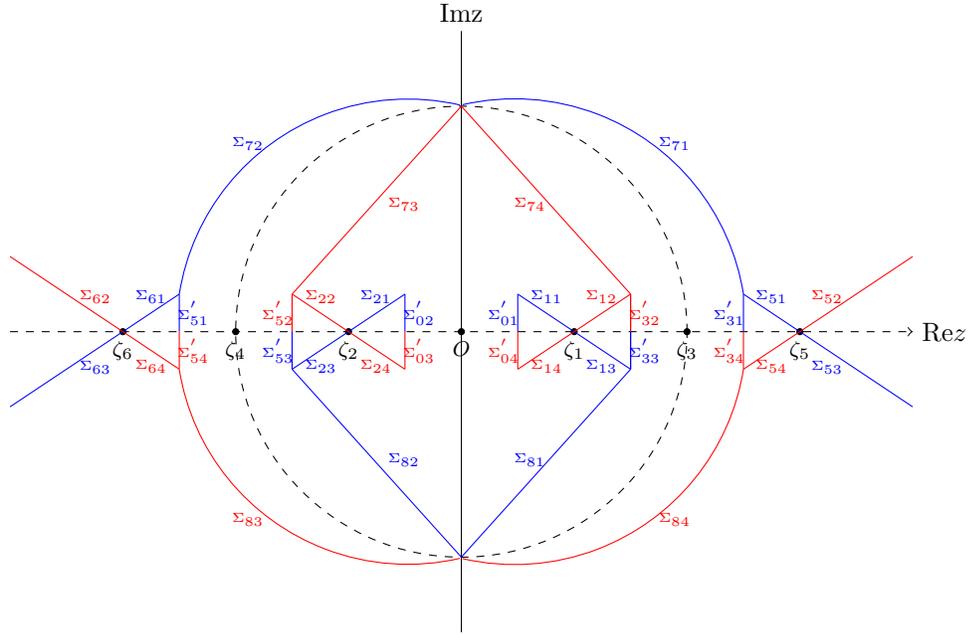
\begin{figure}[H]
\begin{center}
\begin{tikzpicture}[node distance=2cm]
\draw [->,dashed](-6,0)--(6,0)node[right]{ \textcolor{black}{${\rm Re} z$}};;
\draw [](0,-4)--(0,4)  node[above, scale=1] {{\rm Im}z};
\draw [->,dashed](0,0)circle(3cm);

          \node  [below]  at (1.5,0) {\footnotesize $\zeta_{1}$};
          \node  [below]  at (-1.5,0) {\footnotesize $\zeta_{2}$};
          \node  [below] at (3,0) {\footnotesize $\zeta_{3}$};
          \node  [below]  at (-3,0) {\footnotesize $\zeta_{4}$};
          \node  [below]  at (4.5,0) {\footnotesize $\zeta_{5}$};
          \node  [below] at (-4.5,0) {\footnotesize $\zeta_{6}$};
          \node  [below]  at (0,0) {\footnotesize $O$};

          \draw[fill] (1.5,0) circle [radius=0.04];
          \draw[fill] (4.5,0) circle [radius=0.04];
          \draw[fill] (3,0) circle [radius=0.04];
          \draw[fill] (-4.5,0) circle [radius=0.04];
          \draw[fill] (-1.5,0) circle [radius=0.04];
          \draw[fill] (-3,0) circle [radius=0.04];
          \draw[fill] (0,0) circle [radius=0.04];

          \draw [color=blue](0.75,0.5)--(1.5,0)  node[right, scale=1] {};
          \draw [color=blue](1.5,0)--(2.25,-0.5)  node[right, scale=1] {};
          \draw [color=blue](3.75,0.5)--(4.5,0) node[right, scale=1] {};
          \draw [color=blue](4.5,0)--(6,-1) node[right, scale=1] {};

          \draw [color=blue](-0.75,0.5)--(-1.5,0)  node[right, scale=1] {};
          \draw [color=blue](-1.5,0)--(-2.25,-0.5)  node[right, scale=1] {};
          \draw [color=blue](-3.75,0.5)--(-4.5,0) node[right, scale=1] {};
          \draw [color=blue](-4.5,0)--(-6,-1) node[right, scale=1] {};

          \draw [color=red](0.75,-0.5)--(1.5,0)  node[right, scale=1] {};
          \draw [color=red](1.5,0)--(2.25,0.5)  node[right, scale=1] {};
          \draw [color=red](3.75,-0.5)--(4.5,0) node[right, scale=1] {};
          \draw [color=red](4.5,0)--(6,1) node[right, scale=1] {};

          \draw [color=red](-0.75,-0.5)--(-1.5,0)  node[right, scale=1] {};
          \draw [color=red](-1.5,0)--(-2.25,0.5)  node[right, scale=1] {};
          \draw [color=red](-3.75,-0.5)--(-4.5,0) node[right, scale=1] {};
          \draw [color=red](-4.5,0)--(-6,1) node[right, scale=1] {};

          \draw [color=red](0,3)--(2.25,0.5)  node[right, scale=1] {};
          \draw [color=red](0,3)--(-2.25,0.5)  node[right, scale=1] {};
          \draw [color=blue](0,-3)--(2.25,-0.5)  node[right, scale=1] {};
          \draw [color=blue](0,-3)--(-2.25,-0.5)  node[right, scale=1] {};

          \draw[color=blue] (3.75,0.5) arc (9:103:3.08);
          \draw[color=blue] (-3.75,0.5) arc (171:77:3.076);
          \draw[color=red] (-3.75,-0.5) arc (189:283:3.076);
          \draw[color=red] (3.75,-0.5) arc (351:257:3.08);

         \node [left]  at (0.9,0.25) {\textcolor{blue}{\tiny $\Sigma^{'}_{01}$}};
         \node [right]  at (-0.9,0.25) {\textcolor{blue}{\tiny $\Sigma^{'}_{02}$}};
         \node [right]  at (-0.9,-0.25) {\textcolor{red}{\tiny $\Sigma^{'}_{03}$}};
         \node [left]  at (0.9,-0.25) {\textcolor{red}{\tiny $\Sigma^{'}_{04}$}};

         \node [left]  at (3.9,0.25) {\textcolor{blue}{\tiny $\Sigma^{'}_{31}$}};
         \node [right]  at (2.1,0.25) {\textcolor{red}{\tiny $\Sigma^{'}_{32}$}};
         \node [right]  at (2.1,-0.25) {\textcolor{blue}{\tiny $\Sigma^{'}_{33}$}};
         \node [left]  at (3.9,-0.25) {\textcolor{red}{\tiny $\Sigma^{'}_{34}$}};

         \node [right]  at (-3.9,0.25) {\textcolor{blue}{\tiny $\Sigma^{'}_{51}$}};
         \node [left]  at (-2.1,0.25) {\textcolor{red}{\tiny $\Sigma^{'}_{52}$}};
         \node [left]  at (-2.1,-0.25) {\textcolor{blue}{\tiny $\Sigma^{'}_{53}$}};
         \node [right]  at (-3.9,-0.25) {\textcolor{red}{\tiny $\Sigma^{'}_{54}$}};

          \draw [color=blue] (0.75,0.5)--(0.75,0);
          \draw [color=red] (0.75,0)--(0.75,-0.5);
          \draw [color=blue] (-0.75,0.5)--(-0.75,0);
          \draw [color=red] (-0.75,0)--(-0.75,-0.5);

          \draw [color=red] (2.25,0.5)--(2.25,0);
          \draw [color=blue] (2.25,0)--(2.25,-0.5);
          \draw [color=red] (-2.25,0.5)--(-2.25,0);
          \draw [color=blue] (-2.25,0)--(-2.25,-0.5);

          \draw [color=red] (2.25,0.5)--(2.25,0);
          \draw [color=blue] (2.25,0)--(2.25,-0.5);
          \draw [color=red] (-2.25,0.5)--(-2.25,0);
          \draw [color=blue] (-2.25,0)--(-2.25,-0.5);

          \draw [color=blue] (3.75,0.5)--(3.75,0);
          \draw [color=red] (3.75,0)--(3.75,-0.5);
          \draw [color=blue] (-3.75,0.5)--(-3.75,0);
          \draw [color=red] (-3.75,0)--(-3.75,-0.5);

         \node [above]  at (1.135,0.25) {\textcolor{blue}{\tiny $\Sigma_{11}$}};
         \node [below]  at (1.135,-0.25) {\textcolor{red}{\tiny $\Sigma_{14}$}};
         \node [above]  at (1.865,0.25) {\textcolor{red}{\tiny $\Sigma_{12}$}};
         \node [below]  at (1.865,-0.25) {\textcolor{blue}{\tiny $\Sigma_{13}$}};

         \node [above]  at (4.125,0.25) {\textcolor{blue}{\tiny $\Sigma_{51}$}};
         \node [above]  at (4.865,0.25) {\textcolor{red}{\tiny $\Sigma_{52}$}};
         \node[below]  at (4.865,-0.26) {\textcolor{blue}{\tiny $\Sigma_{53}$}};
         \node [below]  at (4.125,-0.25) {\textcolor{red}{\tiny $\Sigma_{54}$}};

         \node [above]  at (-1.135,0.25) {\textcolor{blue}{\tiny $\Sigma_{21}$}};
         \node [below]  at (-1.135,-0.25) {\textcolor{red}{\tiny $\Sigma_{24}$}};
         \node [above]  at (-1.865,0.25) {\textcolor{red}{\tiny $\Sigma_{22}$}};
         \node [below]  at (-1.865,-0.25) {\textcolor{blue}{\tiny $\Sigma_{23}$}};

         \node [above]  at (-4.125,0.25) {\textcolor{blue}{\tiny $\Sigma_{61}$}};
         \node [above]  at (-4.865,0.25) {\textcolor{red}{\tiny $\Sigma_{62}$}};
         \node[below]  at (-4.865,-0.26) {\textcolor{blue}{\tiny $\Sigma_{63}$}};
         \node [below]  at (-4.125,-0.25) {\textcolor{red}{\tiny $\Sigma_{64}$}};

          \node [right]  at (2.5,2.5) {\textcolor{blue}{\tiny $\Sigma_{71}$}};
          \node [left]  at (-2.5,2.5) {\textcolor{blue}{\tiny $\Sigma_{72}$}};
          \node [right]  at (-1.1,1.7) {\textcolor{red}{\tiny $\Sigma_{73}$}};
          \node [left]  at (1.25,1.7) {\textcolor{red}{\tiny $\Sigma_{74}$}};

          \node [right]  at (2.5,-2.5) {\textcolor{red}{\tiny $\Sigma_{84}$}};
          \node [left]  at (-2.5,-2.5) {\textcolor{red}{\tiny $\Sigma_{83}$}};
          \node [right]  at (-1.1,-1.7) {\textcolor{blue}{\tiny $\Sigma_{82}$}};
          \node [left]  at (1.25,-1.7) {\textcolor{blue}{\tiny $\Sigma_{81}$}};

\end {tikzpicture}
\caption{\footnotesize The jump contour $\Sigma^{'}_{jump}$.}
\label{Sigmajump'}
\end{center}
\end{figure}

Define $\Sigma^{(2)}=\Sigma^{'}_{jum}\cup\Sigma_{cir}$,  where
\begin{equation}
\Sigma^{'}_{jum}=\underset{j=1,2,3,4}\bigcup \left(\Sigma^{'}_{0j} \cup\left(\underset{i=3,4}\cup \Sigma^{'}_{ij}\right)\cup \left(\underset{i=1,2,5,6,7,8} \cup \Sigma_{ij}\right)\right)
\end{equation}
can be referred in Figure \ref{Sigmajump'}.
We now use $\mathcal{R}^{(2)}(z)$ to define a new transformation
\begin{equation}\label{defm2}
m^{(2)}(z)=m^{(1)}(z)\mathcal{R}^{(2)}(z),
\end{equation}
which satisfies the following mixed $\bar{\partial}$-RH problem.
\begin{RHP}\label{rhp2}
Find a $2\times 2$ matrix-valued function $m^{(2)}(z):=m^{(2)}(x,t;z)$ such that
\begin{itemize}
    \item[*] $m^{(2)}(z)$ is continuous in $\mathbb{C}\backslash (\Sigma^{(2)}\cup\mathcal{Z}\cup\hat{\mathcal{Z}})$.
    \item[*] Jump relation: $m^{(2)}_{+}(z)=m^{(2)}_{-}(z)v^{(2)}(z)$, $z\in\Sigma^{(2)}$, where
      \begin{align}
         v^{(2)}=[\mathcal{R}^{(2)}_{-}]^{-1}v^{(1)}\mathcal{R}^{(2)}_{+}=\left\{
        \begin{aligned}
        &\begin{bmatrix} 1 & 0 \\  R_{ij}e^{-2\mathrm{i}t\theta} & 1 \end{bmatrix}, \quad z\in\Sigma_{ij}, \quad i=7,8, j=1,2,\\
        &\begin{bmatrix} 1 & R_{ij}e^{2\mathrm{i}t\theta} \\ 0 & 1 \end{bmatrix}, \quad z\in\Sigma_{ij}, \quad i=7,8, j=3,4,\\
        &\begin{bmatrix} 1 & 0 \\ R_{i1}e^{-2\mathrm{i}t\theta} & 1 \end{bmatrix}, \quad z\in\Sigma_{i1}, \quad   i=1,2,5,6,\\
        &\begin{bmatrix} 1 & R_{i2}e^{2\mathrm{i}t\theta} \\ 0 & 1 \end{bmatrix}, \quad z\in\Sigma_{i2}, \quad   i=1,2,5,6,\\
        &\begin{bmatrix} 1 & 0 \\ R_{i3}e^{-2\mathrm{i}t\theta} & 1 \end{bmatrix}, \quad z\in\Sigma_{i3}, \quad   i=1,2,5,6,\\
        &\begin{bmatrix} 1 & R_{i4}e^{2\mathrm{i}t\theta} \\ 0 & 1 \end{bmatrix}, \quad z\in\Sigma_{i4}, \quad   i=1,2,5,6,\\
        &\begin{bmatrix} 1 & 0 \\ (R_{11}-R_{01})e^{-2\mathrm{i}t\theta} & 1 \end{bmatrix}, \quad z\in\Sigma^{'}_{01},\\
        &\begin{bmatrix} 1 & (R_{21}-R_{02})e^{-2\mathrm{i}t\theta} \\ 0 & 1 \end{bmatrix}, \quad z\in\Sigma^{'}_{02},\\
        &\begin{bmatrix} 1 & 0 \\ (R_{24}-R_{03})e^{2\mathrm{i}t\theta} & 1 \end{bmatrix}, \quad z\in\Sigma^{'}_{03},\\
        &\begin{bmatrix} 1 & (R_{14}-R_{04})e^{2\mathrm{i}t\theta} \\ 0 & 1 \end{bmatrix}, \quad z\in\Sigma^{'}_{04},\\
        &\begin{bmatrix} 1 & 0 \\ (R_{i+2,1}-R_{i1})e^{-2\mathrm{i}t\theta} & 1 \end{bmatrix}, \quad z\in\Sigma^{'}_{i1}, \quad   i=3,4,\\
        &\begin{bmatrix} 1 & (R_{i,2}-R_{i-2,1})e^{2\mathrm{i}t\theta} \\ 0 & 1 \end{bmatrix}, \quad z\in\Sigma^{'}_{i2}, \quad   i=3,4,\\
        &\begin{bmatrix} 1 & 0 \\ (R_{i,3}-R_{i-2,3})e^{-2\mathrm{i}t\theta} & 1 \end{bmatrix}, \quad z\in\Sigma^{'}_{i3}, \quad   i=3,4,\\
        &\begin{bmatrix} 1 & (R_{i+2,4}-R_{i4})e^{2\mathrm{i}t\theta} \\ 0 & 1 \end{bmatrix}, \quad z\in\Sigma^{'}_{i4}, \quad   i=3,4,\\
        &\begin{bmatrix} 1 & 0 \\ -\frac{A[\eta_k]}{z-\eta_k}T^{-2}(z)e^{-2\mathrm{i}t\theta(\eta_k)} & 1 \end{bmatrix}, \quad |z-\eta_k|=\varrho, \quad k\in \nabla/\Lambda,\\
        &\begin{bmatrix} 1 & -\frac{z-\eta_k}{A[\eta_k]}T^2(z)e^{2\mathrm{i}t\theta(\eta_k)} \\ 0 & 1 \end{bmatrix}, \quad |z-\eta_k|=\varrho, \quad k\in \triangle/\Lambda,\\
        &\begin{bmatrix} 1 & -\frac{A[\hat{\eta}_k]}{z-\hat{\eta}_k}T^2(z)e^{2\mathrm{i}t\theta(\hat{\eta}_k)} \\ 0 & 1 \end{bmatrix}, \quad |z-\hat{\eta}_k|=\varrho, \quad k\in \nabla/\Lambda,\\
        &\begin{bmatrix} 1 & 0 \\ -\frac{z-\hat{\eta}_k}{A[\hat{\eta}_k]}T^{-2}(z)e^{-2\mathrm{i}t\theta(\hat{\eta}_k)} & 1 \end{bmatrix}, \quad |z-\hat{\eta}_k|=\varrho, \quad k\in \triangle/\Lambda.
        \end{aligned}
        \right.
     \end{align}
    \item[*]For $z\in\mathbb{C}\backslash (\Sigma^{(2)}\cup\mathcal{Z}\cup\hat{\mathcal{Z}})$,
      \begin{equation}
      \bar{\partial}m^{(2)}=m^{(2)}\bar{\partial}\mathcal{R}^{(2)},
      \end{equation}
      where
      \begin{align}
         \bar{\partial}\mathcal{R}^{(2)}=\left\{
        \begin{aligned}
        &\begin{bmatrix} 0 & 0 \\ \bar{\partial}R_{ij}e^{-2\mathrm{i}t\theta} & 0 \end{bmatrix}, \quad z\in\Omega_{ij}, i=0,7,8, j=1,2,\\
        &\begin{bmatrix} 0 & \bar{\partial}R_{ij}e^{2\mathrm{i}t\theta} \\ 0 & 0 \end{bmatrix}, \quad z\in\Omega_{ij}, i=0,7,8, j=1,2,\\
        &\begin{bmatrix} 0 & \bar{\partial}R_{i1}e^{2\mathrm{i}t\theta} \\ 0 & 0 \end{bmatrix}, \quad z\in\Omega_{i1}, i=1,2,\cdots,6,\\
        &\begin{bmatrix} 0 & 0 \\ \bar{\partial}R_{i2}e^{-2\mathrm{i}t\theta} & 0 \end{bmatrix}, \quad z\in\Omega_{i2}, i=1,2,\cdots,6,\\
        &\begin{bmatrix} 0 & \bar{\partial}R_{i3}e^{2\mathrm{i}t\theta} \\ 0 & 0 \end{bmatrix}, \quad z\in\Omega_{i3}, i=1,2,\cdots,6,\\
        &\begin{bmatrix} 0 & 0 \\ \bar{\partial}R_{i4}e^{-2\mathrm{i}t\theta} & 0 \end{bmatrix}, \quad z\in\Omega_{i4}, i=1,2,\cdots,6,\\
        &\begin{bmatrix} 0 & 0 \\ 0 & 0 \end{bmatrix}, \quad elsewhere.\\
      \end{aligned}
        \right.
      \end{align}
    \item[*] Asymptotic behavior
    \begin{equation}\label{rhp1aym}
    m^{(2)}(x,t;z)=I+\mathcal{O}(z^{-1}), \quad  z\rightarrow\infty.
    \end{equation}
    \item[*] Residue conditions
    \begin{align}\label{resm2}
    \underset{z=\eta_k}{\rm Res}m^{(2)}(z)=\left\{
    \begin{aligned}
    &\lim_{z\rightarrow\eta_{k}}m^{(2)}(z)\begin{bmatrix} 0 & 0 \\ A[\eta_{k}]T^{-2}(\eta_{k}) e^{-2it\theta(\eta_{k})} & 0 \end{bmatrix}, \quad k\in\nabla\cap\Lambda,\\
    &\lim_{z\rightarrow\eta_{k}}m^{(2)}(z)\begin{bmatrix} 0 & \frac{1}{A[\eta_{k}]}[(\frac{1}{T})'(\eta_{k})]^{-2}e^{2it\theta(\eta_{k})} \\ 0 & 0 \end{bmatrix}, \quad k\in\triangle\cap\Lambda,
    \end{aligned}
        \right.
    \end{align}
    \begin{align}\label{rhp1resb}
    \underset{z=\hat{\eta}_k}{\rm Res}m^{(2)}(z)=\left\{
    \begin{aligned}
    &\lim_{z\rightarrow\hat{\eta}_k}m^{(2)}(z)\begin{bmatrix} 0 & A[\hat{\eta}_k]T^{2}(\hat{\eta}_{k}) e^{2it\theta(\hat{\eta}_{k})} \\ 0 & 0 \end{bmatrix}, \quad k\in\nabla\cap\Lambda,\\
    &\lim_{z\rightarrow\hat{\eta}_k}m^{(2)}(z)\begin{bmatrix} 0 & 0 \\ \frac{1}{A[\hat{\eta}_k]}\frac{1}{[T'(\hat{\eta}_k)]^{2}} e^{-2it\theta(\hat{\eta}_{k})} & 0 \end{bmatrix}, \quad k\in\triangle\cap\Lambda.
    \end{aligned}
        \right.
    \end{align}
\end{itemize}
\end{RHP}
\subsubsection{Decomposition of the mixed $\bar\partial$-RH Problem}\label{423}
To solve RH Problem \ref{rhp2}, we decompose it into a model RH problem for $m^{rhp}(z)$  with  $\bar{\partial}\mathcal{R}^{(2)}\equiv0$ and a pure $\bar{\partial}$-problem for $m^{(3)}(z)$ with $\bar{\partial}\mathcal{R}^{(2)}\neq0$, which can be shown as the following structure
\begin{align}
m^{(2)}(z)=\left\{
        \begin{aligned}
    &\bar{\partial}\mathcal{R}^{(2)}\equiv0\longrightarrow m^{rhp}(z),\\
    &\bar{\partial}\mathcal{R}^{(2)}\neq0\longrightarrow m^{(3)}(z)=m^{(2)}(z)[m^{rhp}(z)]^{-1}
    \end{aligned}
        \right.
\end{align}
with $m^{rhp}(z)$ satisfies the following RH problem.
\begin{RHP}\label{m2php}
Find a $2\times 2$ matrix-valued function $m^{rhp}(z):=m^{rhp}(x,t;z)$ such that
\begin{itemize}
    \item[*] $m^{rhp}(z)$ is analytic in $\mathbb{C}\backslash (\Sigma^{(2)}\cup\mathcal{Z}\cup\hat{\mathcal{Z}})$.

    \item[*] Jump relation: \ $m_+^{rhp}(z)=m_-^{rhp}(z)v^{(2)}(z)$, $z\in\Sigma^{(2)}$.

    \item[*]For $z\in\mathbb{C}\backslash (\Sigma^{(2)}\cup\mathcal{Z}\cup\hat{\mathcal{Z}})$,
      \begin{equation}
      \bar{\partial}\mathcal{R}^{(2)}\equiv0.
      \end{equation}

    \item[*] Asymptotic behavior
    \begin{equation}\label{rhp1aym}
    m^{rhp}(z)=I+\mathcal{O}(z^{-1}), \quad  z\rightarrow\infty.
    \end{equation}

    \item[*] $ m^{rhp}(z)$ has same jump matrix and residue conditions as $m^{(2)}(z)$.
\end{itemize}
\end{RHP}
We construct the solution  $m^{rhp}(z)$ of the RH problem \ref{m2php} in the following form
\begin{align}
m^{rhp}(z)=\left\{
        \begin{aligned}
    &E(z)m^{sol}(z)=E(z)m^{err}(z)m^{\Lambda}(z), \quad z\in\mathbb{C}/\left(\underset{i=1,2,5,6}\cup U_{\zeta_i}\right),\\
    &E(z)m^{sol}(z)m^{lo}(z), \quad  z\in\underset{i=1,2,5,6}\cup U_{\zeta_i},
    \end{aligned}
        \right.
\end{align}
where $U_{\zeta_{i}}$ is defined as
\begin{equation}
U_{\zeta_{i}}=\left\{z: |z-\zeta_{i}|<\varrho\right\}.
\end{equation}

In the above, $m^{rhp}(z)$ decomposes into two parts: one part accounts for the solitons determined by discrete spectrums, and the other part accounts for the contribution of phase points. Specifically,
\begin{equation}
m^{\Lambda}(z)=m^{sol}(z)|_{v^{(2)}=I, z\in\Sigma_{cir}}=m^{rhp}(z)|_{v^{(2)}=I, z\in\Sigma_{jum}^{'}}
\end{equation}
will be solved in next subsection \ref{out}; $m^{lo}(z)$ uses parabolic cylinder functions to build a matrix to match jumps of $m^{rhp}(z)$ in a neighborhood $\zeta_{i}, i=1,2,5,6$, which is shown in subsection \ref{in}; $E(z)$ is an error function and a solution of a small norm Riemann-Hilbert problem, which is shown in subsection \ref{erro}.

\subsection{Analysis on the pure RH Problem}\label{purerhp}

\subsubsection{Outer model RH problem}\label{out}

In this subsection, we build a reflectionless case of RHP \ref{rhp0} to show that its solution can approximated with $m^{sol}(z)$.
The following RH problem constructs the $m^{sol}$.
\begin{RHP}\label{rhpmsol}
Find a $2\times 2$ matrix-valued function $m^{sol}(z):=m^{sol}(x,t;z)$ such that
\begin{itemize}
    \item[*] $m^{sol} (z)$ is meromorphic in $\mathbb{C}\backslash\Sigma_{cir}$.
    \item[*] Jump relation: $m^{sol}_{+}(z)=m^{sol}_{-}(z)v^{(2)}(z)$, $z\in\Sigma^{cir}$.
    \item[*] Asymptotic behavior:
    \begin{equation}\nonumber
    m^{sol}(z)=I+\mathcal{O}(z^{-1}), \quad  z\rightarrow\infty.
    \end{equation}
    \item[*] $\bar\partial R^{(2)}=0$.
    \item[*] $m^{sol}(z)$ has the same residue conditions as $m^{(2)}(z)$.
   \end{itemize}
\end{RHP}
\begin{proposition}
If $m^{sol}(z)$ is the solution of RH problem \ref{rhpmsol} with scattering data \begin{equation}
\left\{\rho(z),\tilde{\rho}(z),\{\eta_k,C_k\}_{k=1}^{2N_1+N_2}\right\},
\end{equation}
$m^{sol}(z)$ exists and unique.
\end{proposition}
\begin{proof}
To transform $m^{sol}(z)$ to the solution of RH problem \ref{rhp0}, the jump and poles need to be restored. We reverse the triangularity effected in \eqref{defm1} and \eqref{defm2}:
\begin{align}
\tilde{m}(z)=&(-\mathrm{i}\prod_{k\in\triangle_1}\prod_{k\in\triangle_2}|z_k|^2\omega_l)^{-\sigma_3}m^{sol}(z)T(z)^{\hat{\sigma}_3}G(z)^{-1}\nonumber\\
&\left[ \prod_{k\in\triangle_1}\prod_{k\in\triangle_2}\frac{(z+z_{k}^{-1})(z-\bar{z}_{k}^{-1})(z-\mathrm{i}\omega_{l}^{-1})}
{(zz_{k}^{-1}-1)(z\bar{z}_{k}^{-1}+1)(\mathrm{i}\omega_{l}^{-1}z+1)}\right]^{\sigma_3}
\end{align}
with $G(z)$ defined in \eqref{G}. Fisrt we verify $\tilde{m}(z)$ satisfy ing RH problem \ref{rhp0}. The transformation to $m^{sol}(z)$ preserves the normalization conditions at infinity obviously. And compare to \eqref{defm1}, this transformation restores the jump on $\Sigma_{cir}$ to residue for $k\notin\Lambda$. As for $k\in\Lambda$, take $\eta_k\in\nabla\cap\Lambda$ as an example. Substituting \eqref{resm2} into this transformation:
\begin{align}
&\underset{z=\eta_k}{\rm Res}\tilde{m}(z)\nonumber\\
=&(-\mathrm{i}\prod_{k\in\triangle_1}\prod_{k\in\triangle_2} |z_k|^2\omega_l)^{-\sigma_3}\underset{z=\eta_k}{\rm Res}m^{sol}T^{\hat{\sigma}_3}G^{-1}\left[\prod_{k\in\triangle_1}\prod_{k\in\triangle_2}\frac{(z+z_{k}^{-1})(z-\bar{z}_{k}^{-1})(z-\mathrm{i}\omega_{l}^{-1})}
{(zz_{k}^{-1}-1)(z\bar{z}_{k}^{-1}+1)(\mathrm{i}\omega_{l}^{-1}z+1)}\right]^{\sigma_3}\nonumber\\
=&\lim_{z\rightarrow\eta_k}(-\mathrm{i}\prod_{k\in\triangle_1}\prod_{k\in\triangle_2}|z_k|^2\omega_l)^{-\sigma_3}m^{sol}(z)\begin{bmatrix} 0 & 0\\ A[\eta_k]T^{-2}(\eta_k)e^{-2\mathrm{i}t\theta(\eta_k)} & 0 \end{bmatrix}\nonumber\\
&\left[\prod_{k\in\triangle_1}\prod_{k\in\triangle_2}\frac{(z+z_{k}^{-1})(z-\bar{z}_{k}^{-1})(z-\mathrm{i}\omega_{l}^{-1})}
{(zz_{k}^{-1}-1)(z\bar{z}_{k}^{-1}+1)(\mathrm{i}\omega_{l}^{-1}z+1)}\right]^{\sigma_3}\nonumber\\
=&\lim_{z\rightarrow\eta_k}\tilde{m}(z)\begin{bmatrix} 0 & 0\\ A[\eta_k]\tilde{\delta}(\eta_k)^{-2}(\eta_k)e^{-2\mathrm{i}t\theta(\eta_k)} & 0 \end{bmatrix}
\end{align}
where
\begin{equation}
\tilde{\delta}(\eta_k)=exp\left[\mathrm{i}\int_{\Gamma}\nu(s)\left(\frac{1}{s-\eta_k}-\frac{1}{2s}\right)\right]ds.
\end{equation}
Then $\tilde{m}(z)$ is the solution of RH problem \ref{rhp0} with absence of reflection, whose solution exists and can be obtained as described similarly in \cite{CJ} Appendix A. And its uniqueness comes from Liouville's theorem.
\end{proof}

The main contribution to $m^{sol}(z)$ comes from discrete spectrums $\left\{\eta_k, \hat{\eta}_k, \ k\in\Lambda\right\}$, which can be described by the following PH problem \ref{RHPLAMDA}.
\begin{RHP}\label{RHPLAMDA}
Find a $2\times 2$ matrix-valued function $m^{\Lambda}(z):=m^{\Lambda}(x,t;z)$ such that
\begin{itemize}
  \item[*]$m^{\Lambda}(z)$ is analytic in $\mathbb{C}\backslash\left(\mathcal{Z}\cup\hat{\mathcal{Z}}\right)$.
  \item[*]$m^{\Lambda}(z)\rightarrow I, \  z\rightarrow\infty$; \ $m^{\Lambda}(z)\rightarrow \frac{\mathrm{i}}{z}\sigma\sigma_1, \  z\rightarrow 0$.
  \item[*]$v^{(2)}=I, \ \bar{\partial}\mathcal{R}^{(2)}=0$.
  \item[*]$m^{\Lambda}(z)$ has same residue conditions as $m^{(2)}(z)$.
\end{itemize}
\end{RHP}
The unique solution $m^{\Lambda}(z)$ to the above RH problem \ref{RHPLAMDA} is given in the following proposition.
\begin{proposition}\label{mlambda}
For given scattering data $\left\{0, 0, \{\eta_{k_0}, \tilde{C}_{k_0}\}_{k_0\in\Lambda}\right\}$
, the solution to RH problem \ref{RHPLAMDA} exists uniquely and can be constructed explicitly, where
\begin{align}
\tilde{C}_{k_0}=\left\{
        \begin{aligned}
    &A[\eta_{k_0}]\tilde{\delta}^{-2}(\eta_{k_0}), \quad  k_0\in\nabla\cap\Lambda,\\
    &-\frac{1}{A[\eta_{k_0}]}T'(\eta_{k_0})\tilde{\delta}^{-2}(\eta_{k_0}), \quad k_0\in\triangle\cap\Lambda.
    \end{aligned}
        \right.
\end{align}
\end{proposition}
\begin{proof}
The proof is similar to Lemma 6.7 in \cite{CJ}.
\end{proof}
Denote $m^{err}(z)$ as the error function generated by jumps on $\Sigma_{cir}$, then $m^{err}(z)=m^{sol}(z)[m^\Lambda(z)]^{-1}$ satisfies the following RH problem:
\begin{RHP}\label{merror}
Find a $2\times 2$ matrix-valued function $m^{err}(z):=m^{err}(x,t;z)$ such that
\begin{itemize}
  \item[*] $m^{err}(z)$ is analytic in $\mathbb{C}/\Sigma_{cir}$.
  \item[*] Jump relation: $m^{err}_{+}(z)=m^{err}_{-}(z)v^{err}(z)$, $z\in\Sigma_{cir}$, where
   \begin{equation}
    v^{err}(z)=m^\Lambda(z)v^{(2)}(z)[m^\Lambda(z)]^{-1}.
   \end{equation}
  \item[*] Asymptotic behavior
   \begin{equation}\label{rhp1aym}
    m^{err}(z)=I+\mathcal{O}(z^{-1}), \quad  z\rightarrow\infty.
   \end{equation}
\end{itemize}
\end{RHP}
For $z\in\Sigma_{cir}$, the jump matrix $v^{err}(z)$ satisfies
\begin{equation}
\|v^{err}(z)-I\|_{L^\infty(\Sigma_{cir})}\lesssim e^{-ct}.
\end{equation}
From Beals-Coifman theorem \cite{BC}, we obtain the solution of RH problem \ref{merror}:
\begin{equation}
m^{err}(z)=I+\frac{1}{2\pi\mathrm{i}}\int_{\Sigma_{cir}}\frac{\mu_e(s)(v^{err}(s)-I)}{s-z}ds,
\end{equation}
where $\mu_{e}(s)\in L^2(\Sigma_{cir})$ satisfies
\begin{equation}
(1-C_{w_e})\mu_e(z)=I.
\end{equation}
Thus,
\begin{equation}
\|C_{w_e}\|_{L^2(\Sigma_{cir})}\leq\|C_-\|_{L^2(\Sigma_{cir})}\|v^{err}(z)-I\|_{L^\infty(\Sigma_{cir})}\lesssim e^{-ct},
\end{equation}
which implies $1-C_{w_e}$ invertible for sufficiently large $t$. Furthermore, we can prove the existence and uniqueness for $\mu_e$, $m^{err}(z)$ and we have following proposition.
\begin{proposition}
For any $(x,t)$ such that $\xi=\frac{x}{t}<-6$ and $t\gg1$, uniformly for $z\in\mathbb{C}$ we have
\begin{equation}\label{asymsol}
m^{sol}(z)=m^\Lambda(z)\left[I+\mathcal{O}(e^{-ct})\right],
\end{equation}
and, in particular, for large $z$ we have
\begin{equation}\label{msol}
m^{sol}(z)=m^\Lambda(z)\left[I+z^{-1}\mathcal{O}(e^{-ct})+\mathcal{O}(z^{-2})\right].
\end{equation}
Moreover, the relationship between their corresponding solution for nonlocal mKdV equation \eqref{nmkdv} as follows:
\begin{equation}
q^{sol}(x,t)=q^\Lambda(x,t)+\mathcal{O}(e^{-ct}),
\end{equation}
where
\begin{equation}
q^{sol}(x,t)=-\mathrm{i}\underset{z\rightarrow\infty}\lim{(zm^{sol})}_{12}, \quad q^\Lambda(x,t)=-\mathrm{i}\underset{z\rightarrow\infty}\lim{(zm^\Lambda)}_{12}.
\end{equation}
\end{proposition}
\begin{proof}
The proof is similar to Lemma 6.7 in \cite{CJ}.
\end{proof}
\subsubsection{A local solvable RH model near phase points}\label{in}

In the neighborhood $U_{\xi_{i}}$ of $\xi_{i}, i=1,2,5,6$ , we establish a local model $m^{lo}(z)$ which exactly matches the jumps of $m^{rhp}(z)$ on $\Sigma^{lo}$ for function $E(z)$ and then it has a uniform estimate on the decay of the jump, where $\Sigma^{lo}$ is defined as:
\begin{equation}
\begin{split}
&\Sigma^{lo}=\underset{i=1,2,5,6}\cup\left( L_{i}\cap U_{\zeta_{i}}\right), \\
&\Sigma_{i}^{lo}=L_{i}\cap U_{\zeta_{i}}, \quad i=1,2,5,6,
\end{split}
\end{equation}
where
\begin{equation}
L_{i}=\underset{j=1,2,3,4}\cup\Sigma_{ij}, \quad i=1,2,5,6,
\end{equation}
see in Figure \ref{Sigmalo}.
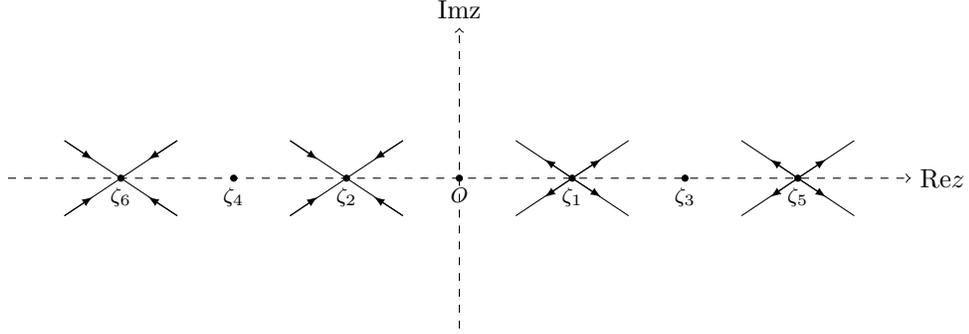
\begin{figure}[H]
  \begin{center}
    \begin{tikzpicture}[node distance=2cm]
          \draw [->,dashed](-6,0)--(6,0)node[right]{ \textcolor{black}{${\rm Re} z$}};;
          \draw [->,dashed](0,-2)--(0,2)  node[above, scale=1] {{\rm Im}z};

          \node  [below]  at (1.5,0) {\footnotesize $\zeta_{1}$};
          \node  [below]  at (-1.5,0) {\footnotesize $\zeta_{2}$};
          \node  [below] at (3,0) {\footnotesize $\zeta_{3}$};
          \node  [below]  at (-3,0) {\footnotesize $\zeta_{4}$};
          \node  [below]  at (4.5,0) {\footnotesize $\zeta_{5}$};
          \node  [below] at (-4.5,0) {\footnotesize $\zeta_{6}$};
          \node  [below]  at (0,0) {\footnotesize $O$};

          \draw[fill] (1.5,0) circle [radius=0.04];
          \draw[fill] (4.5,0) circle [radius=0.04];
          \draw[fill] (3,0) circle [radius=0.04];
          \draw[fill] (-4.5,0) circle [radius=0.04];
          \draw[fill] (-1.5,0) circle [radius=0.04];
          \draw[fill] (-3,0) circle [radius=0.04];
          \draw[fill] (0,0) circle [radius=0.04];

          \draw (0.75,0.5)--(1.5,0)  node[right, scale=1] {};
          \draw (1.5,0)--(2.25,-0.5)  node[right, scale=1] {};
          \draw (3.75,0.5)--(4.5,0) node[right, scale=1] {};
          \draw (4.5,0)--(5.25,-0.5) node[right, scale=1] {};

          \draw (-0.75,0.5)--(-1.5,0)  node[right, scale=1] {};
          \draw (-1.5,0)--(-2.25,-0.5)  node[right, scale=1] {};
          \draw (-3.75,0.5)--(-4.5,0) node[right, scale=1] {};
          \draw (-4.5,0)--(-5.25,-0.5) node[right, scale=1] {};

          \draw (0.75,-0.5)--(1.5,0)  node[right, scale=1] {};
          \draw (1.5,0)--(2.25,0.5)  node[right, scale=1] {};
          \draw (3.75,-0.5)--(4.5,0) node[right, scale=1] {};
          \draw (4.5,0)--(5.25,0.5) node[right, scale=1] {};

          \draw (-0.75,-0.5)--(-1.5,0)  node[right, scale=1] {};
          \draw (-1.5,0)--(-2.25,0.5)  node[right, scale=1] {};
          \draw (-3.75,-0.5)--(-4.5,0) node[right, scale=1] {};
          \draw (-4.5,0)--(-5.25,0.5) node[right, scale=1] {};

          \draw  [-latex]  (1.5,0) to (1.135,0.25);
          \draw  [-latex]  (1.5,0) to (1.885,-0.25);
          \draw  [-latex]  (4.5,0) to (4.135,0.25);
          \draw  [-latex]  (4.5,0) to (4.875,-0.25);

          \draw  [-latex]  (1.5,0) to (1.135,-0.25);
          \draw  [-latex]  (1.5,0) to (1.885,0.25);
          \draw [-latex]  (4.5,0) to (4.135,-0.25);
          \draw  [-latex]  (4.5,0) to (4.875,0.25);

          \draw  [-latex]  (-0.75,0.5) to (-1.135,0.25);
          \draw  [-latex]  (-2.25,-0.5) to (-1.885,-0.25);
          \draw [-latex]  (-3.75,0.5) to (-4.135,0.25);
          \draw [-latex]  (-5.25,-0.5) to (-4.875,-0.25);

          \draw  [-latex]  (-0.75,-0.5) to (-1.135,-0.25);
          \draw [-latex]  (-2.25,0.5) to (-1.885,0.25);
          \draw  [-latex]  (-3.75,-0.5) to (-4.135,-0.25);
          \draw  [-latex]  (-5.25,0.5) to (-4.875,0.25);
          \end {tikzpicture}
          \caption{\footnotesize The jump contours $\Sigma^{lo}$ of $m^{lo}$.}.
        \label{Sigmalo}
    \end{center}
\end{figure}
\begin{RHP}\label{rhplo}
Find a $2\times 2$ matrix-valued function $m^{lo}(z):=m^{lo}(x,t;z)$ such that
\begin{itemize}
\item[*] Analyticity: $m^{lo}(z)$ is analytic in $\mathbb{C}\backslash\Sigma^{lo}$.
\item[*] Jump relation: $m_{+}^{lo}(z)=m_{-}^{lo}(z)v^{(2)}(z), \quad z\in\Sigma^{lo}$.
\item[*] Asymptotic behaviors: $m^{lo}(z)=I+\mathcal{O}(z^{-1}), \quad  z\rightarrow\infty$.
\end{itemize}
\end{RHP}
According to \cite{ZF}, solving for $m^{lo}(z)$ of RH problem \ref{rhplo} can be decomposed into RH problem  of $m^{lo}_{i}(z), i=1,2,5,6$, where $m_{i}^{(lo)}$ can be constructed by parabolic cylinder equation. Further, the solution $m^{lo}$ can be expressed as the following proposition as $t\rightarrow\infty$.
\begin{proposition}
As $t\rightarrow\infty$,
\begin{equation}
m^{lo}(z)=I+\frac{1}{2}t^{-\frac{1}{2}}\sum_{i=1,2,5,6}\frac{m_{i,1}^{pc}(\zeta_{i})}{\sqrt{\theta^{''}(\zeta_{i})}(z-\zeta_{i})}+\mathcal{O}(t^{-1}),
\end{equation}
where
\begin{equation}\label{mpc}
m^{pc}_{i,1}=\begin{bmatrix} 0 & \beta^{\zeta_i}_{12} \\ -\beta^{\zeta_i}_{21} & 0 \end{bmatrix}
\end{equation}
and
\begin{align}\label{beta}
&\beta^{\zeta_i}_{12}=-\frac{\sqrt{2\pi}e^{\frac{\pi}{4}\mathrm{i}}e^{-\frac{\pi}{2}\nu}}{\rho_{\zeta_i}\Gamma(-\mathrm{i}\nu(\zeta_i))}
\triangleq t^{\rm{Im}\nu(\zeta_i)}\tilde{\beta}^{\zeta_i}_{12},\\
&\beta^{\zeta_i}_{21}=\frac{\sqrt{2\pi}e^{-\frac{\pi}{4}\mathrm{i}}e^{-\frac{\pi}{2}\nu}}{\tilde{\rho}_{\zeta_i}\Gamma(\mathrm{i}\nu(\zeta_i))}
\triangleq t^{-\rm{Im}\nu(\zeta_i)}\tilde{\beta}^{\zeta_i}_{21}.
\end{align}
\end{proposition}
\subsubsection{The small norm RH problem for error function}\label{erro}

In this section, we consider the error matrix-function $E(z)$ generated by the jumps on $\Sigma_{jup}^{'}$ outside $U(\xi)$, where
\begin{equation}
U(\xi)=\underset{i=1,2,3,4}\cup U_{\zeta_{i}},
\end{equation}
and $E(z)$ allows the following RHP:
\begin{RHP}\label{rhpE}
    Find a $2\times 2$ matrix-valued function $E(z)$ such that
\begin{itemize}
    \item[*] $E(z)$ is analytical in $\mathbb{C}\backslash \Sigma^{E}$, where
    \begin{equation}
        \Sigma^{E}=\partial U(\xi) \cup \left(\Sigma_{jum}^{'}\backslash U(\xi)\right);
    \end{equation}
    \item[*] $E(z)$ takes continuous boundary values $E_{\pm}(z)$ on $\Sigma^{E}$ and
    \begin{equation}
        E_{+}(z)=E_{-}(z)v^{E}(z),
    \end{equation}
    where
    \begin{equation}
      v^{E}(z)=\left\{
        \begin{aligned}
        &m^{sol}(z)v^{(2)}(z)\left[{m^{sol}(z)}\right]^{-1}, \quad z\in\Sigma_{jum}^{'}\backslash U(\xi),\\
        &m^{sol}(z)m^{lo}(z)\left[{m^{sol}(z)}\right]^{-1}. \quad z\in \partial U(\xi).
        \end{aligned}
        \right.
    \end{equation}
    \item[*] asymptotic behavior
    \begin{equation}
     E(z)=I+\mathcal{O}(z^{-1}), \quad z\rightarrow \infty.
    \end{equation}
\end{itemize}
\end{RHP}
\begin{figure}
        \begin{center}
        \begin{tikzpicture}[node distance=2cm]
          \draw [->,dashed](-6,0)--(6,0)node[right]{ \textcolor{black}{${\rm Re} z$}};;
          \draw [](0,-4)--(0,4)  node[above, scale=1] {{\rm Im}z};
          \draw [->,dashed](0,0)circle(3cm);

          \draw(1.5,0)circle(0.5cm);
          \draw(-1.5,0)circle(0.5cm);
          \draw(-4.5,0)circle(0.5cm);
          \draw(4.5,0)circle(0.5cm);

          \draw (0.75,0.5)--(1.5,0)  node[right, scale=1] {};
          \draw (1.5,0)--(2.25,-0.5)  node[right, scale=1] {};
          \draw (3.75,0.5)--(4.5,0) node[right, scale=1] {};
          \draw (4.5,0)--(6,-1) node[right, scale=1] {};

          \draw (-0.75,0.5)--(-1.5,0)  node[right, scale=1] {};
          \draw (-1.5,0)--(-2.25,-0.5)  node[right, scale=1] {};
          \draw (-3.75,0.5)--(-4.5,0) node[right, scale=1] {};
          \draw (-4.5,0)--(-6,-1) node[right, scale=1] {};

          \draw (0.75,-0.5)--(1.5,0)  node[right, scale=1] {};
          \draw (1.5,0)--(2.25,0.5)  node[right, scale=1] {};
          \draw (3.75,-0.5)--(4.5,0) node[right, scale=1] {};
          \draw (4.5,0)--(6,1) node[right, scale=1] {};

          \draw (-0.75,-0.5)--(-1.5,0)  node[right, scale=1] {};
          \draw (-1.5,0)--(-2.25,0.5)  node[right, scale=1] {};
          \draw (-3.75,-0.5)--(-4.5,0) node[right, scale=1] {};
          \draw (-4.5,0)--(-6,1) node[right, scale=1] {};

          \draw (0,3)--(2.25,0.5)  node[right, scale=1] {};
          \draw (0,3)--(-2.25,0.5)  node[right, scale=1] {};
          \draw (0,-3)--(2.25,-0.5)  node[right, scale=1] {};
          \draw (0,-3)--(-2.25,-0.5)  node[right, scale=1] {};
          \draw (-3.75,-0.5) arc (189:283:3.076);

          \draw(3.75,0.5) arc (9:103:3.08);
          \draw(-3.75,0.5) arc (171:77:3.076);
          \draw (3.75,-0.5) arc (351:257:3.08);

           \draw [fill=white] (1.5,0) circle [radius=0.5];
          \draw [fill=white] (-1.5,0) circle [radius=0.5];
          \draw [fill=white] (4.5,0) circle [radius=0.5];
          \draw [fill=white] (-4.5,0) circle [radius=0.5];

          \node  [below]  at (1.5,0) {\footnotesize $\zeta_{1}$};
          \node  [below]  at (-1.5,0) {\footnotesize $\zeta_{2}$};
          \node  [below] at (3,0) {\footnotesize $\zeta_{3}$};
          \node  [below]  at (-3,0) {\footnotesize $\zeta_{4}$};
          \node  [below]  at (4.5,0) {\footnotesize $\zeta_{5}$};
          \node  [below] at (-4.5,0) {\footnotesize $\zeta_{6}$};
          \node  [below]  at (0,0) {\footnotesize $O$};

          \draw[fill] (1.5,0) circle [radius=0.04];
          \draw[fill] (4.5,0) circle [radius=0.04];
          \draw[fill] (3,0) circle [radius=0.04];
          \draw[fill] (-4.5,0) circle [radius=0.04];
          \draw[fill] (-1.5,0) circle [radius=0.04];
          \draw[fill] (-3,0) circle [radius=0.04];
          \draw[fill] (0,0) circle [radius=0.04];

          \draw  (0.75,0.5)--(0.75,0);
          \draw  (0.75,0)--(0.75,-0.5);
          \draw  (-0.75,0.5)--(-0.75,0);
          \draw  (-0.75,0)--(-0.75,-0.5);

          \draw  (2.25,0.5)--(2.25,0);
          \draw  (2.25,0)--(2.25,-0.5);
          \draw  (-2.25,0.5)--(-2.25,0);
          \draw  (-2.25,0)--(-2.25,-0.5);

          \draw  (2.25,0.5)--(2.25,0);
          \draw  (2.25,0)--(2.25,-0.5);
          \draw  (-2.25,0.5)--(-2.25,0);
          \draw  (-2.25,0)--(-2.25,-0.5);

          \draw  (3.75,0.5)--(3.75,0);
          \draw  (3.75,0)--(3.75,-0.5);
          \draw  (-3.75,0.5)--(-3.75,0);
          \draw  (-3.75,0)--(-3.75,-0.5);
          \end {tikzpicture}
          \caption{\footnotesize The jump contours $\Sigma^{E}$ of $E(z)$.}.
        \label{SigmaE}
        \end{center}
\end{figure}
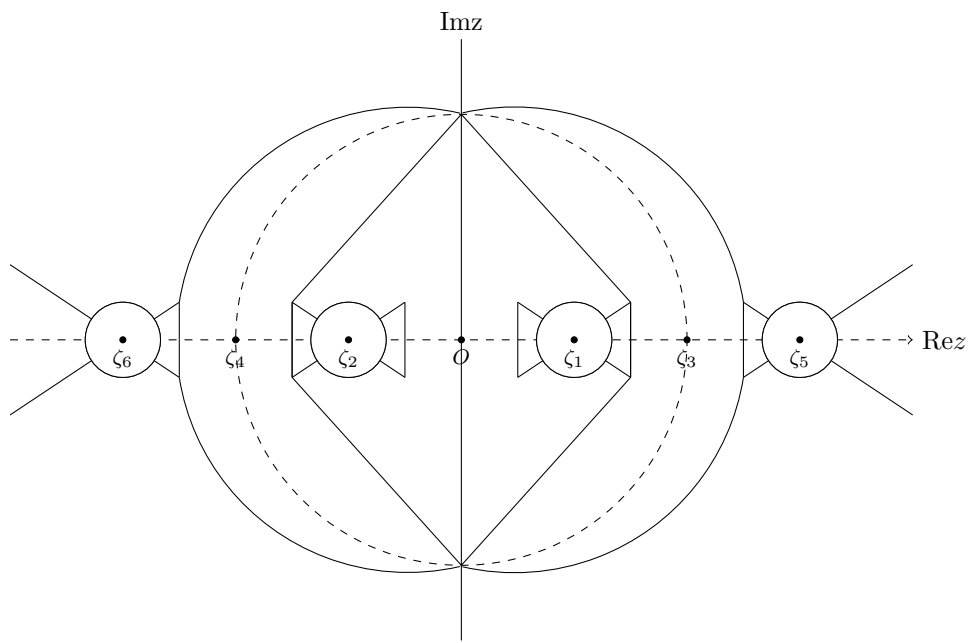
By simple calculation, we have the following estimates of $v^{E}$:
\begin{align}
\|v^{E}-I\|_{L^\infty(\Sigma^E)}=\left\{
        \begin{aligned}
    &\mathcal{O}(e^{-ct}), \quad z\in\Sigma_{jum}^{'}\backslash U(\xi),\\
    &\mathcal{O}\Big(t^{-\frac{1}{2}+\underset{i=1,2,5,6}\max|\rm{Im}\nu(\zeta_i)|}\Big), \quad z\in \partial U(\xi).
    \end{aligned}
        \right.
\end{align}
\begin{proposition}
RHP \ref{rhpE} has an unique solution $E(z)$.
\end{proposition}
\begin{proof}
According to Beals-Coifman theory, the solution for RHP \ref{rhpE} can be written as:
\begin{equation}\label{E}
E(z)=I+\frac{1}{2\pi\mathrm{i}}\int_{\Sigma^{E}}\frac{\mu_{E}(v^{E}(s)-I)}{s-z}ds,
\end{equation}
where $\mu_{E}\in L^{2}(\Sigma^{E})$, and satisfies
\begin{equation}
(I-C_{w_{E}})\mu_{E}=I.
\end{equation}
We can obtain the following estimate by the definition of $C_{w}$:
\begin{equation}
\|C_{w_{E}}\|_{L^{2}(\Sigma^{E})}\leq\|C_{-}\|_{L^{2}(\Sigma^{E})}\|v^{E}-I\|_{L^{\infty}(\Sigma^{E})}
\lesssim\mathcal{O}\Big(t^{-\frac{1}{2}+\underset{i=1,2,5,6}\max|\rm{Im}\nu(\zeta_i)|}\Big),
\end{equation}
which implies $I-C_{w_{E}}$ is invertible for sufficiently large $t$. Furthermore, the existence and uniqueness for $\mu$ and $E(z)$ cab be proved.
\end{proof}
In order to reconstruct the solution $q(x,t)$ of \eqref{nmkdv}, we need the asymptotic behavior of $E(z)$ as $z\rightarrow\infty$.
\begin{proposition}
As $z\rightarrow\infty$, we have
\begin{equation}
E(z)=I+\frac{E_{1}}{z}+\mathcal{O}(z^{-2}),
\end{equation}
where
\begin{equation}
E_{1}=\sum_{i=1,2,5,6}t^{-\frac{1}{2}+\rm{Im}\nu(\zeta_i)}f_i+
\left(\mathcal{O}\Big(t^{-1+\underset{i=1,2,5,6}\max|\rm{Im}\nu(\zeta_i)|-\underset{i=1,2,5,6}\min\rm{Im}\nu(\zeta_i)}\Big), \mathcal{O}\Big(t^{-1+\underset{i=1,2,5,6}\max|\rm{Im}\nu(\zeta_i)|+\underset{i=1,2,5,6}\max\rm{Im}\nu(\zeta_i)}\Big)\right),
\end{equation}
\begin{equation}
f_i=\frac{m_{11}^2(\zeta_i)\tilde{\beta_{12}^{\zeta_i}}-m_{12}^2(\zeta_i)\tilde{\beta_{21}^{\zeta_i}}}
{2\sqrt{\theta^{''}(\zeta_{i})}\det{m^{sol}(\zeta_i)}}, \quad
m^{sol}(\zeta_i)=\begin{bmatrix} m_{11}(\zeta_i) & m_{12}(\zeta_i) \\  m_{21}(\zeta_i) & m_{22}(\zeta_i) \end{bmatrix}.
\end{equation}
\end{proposition}
\begin{proof}
Recall \eqref{E}, we obtain
\begin{equation}
E_{1}=-\frac{1}{2\pi\mathrm{i}}\int_{\Sigma^{E}}\mu_{E}(s)(v^{E}(s)-I)ds=I_{1}+I_{2}+I_{3},
\end{equation}
where
\begin{align}
        &I_{1}=-\frac{1}{2\pi i}\oint_{\partial U(\xi)}\left(v^{E}(s)-I\right)ds,\\
        &I_{2}=-\frac{1}{2\pi i}\int_{\Sigma^{E}\backslash U(\xi)}\left(v^{E}(s)-I\right)ds,\\
        &I_{3}=-\frac{1}{2\pi i}\int_{\Sigma^{E}}(\mu_{E}(s)-I)({v^{E}}(s)-I)ds.
\end{align}
For $I_2$ and $I_3$, we have
\begin{align}
& I_2=\mathcal{O}(e^{-ct}),\\
&I_3=\mathcal{O}\Big(t^{-\frac{1}{2}+\underset{i=1,2,5,6}\max|\rm{Im}\nu(\zeta_i)|}\Big)
\left(\mathcal{O}\Big(t^{-\frac{1}{2}-\underset{i=1,2,5,6}\min\rm{Im}\nu(\zeta_i)}\Big), \mathcal{O}\Big(t^{-\frac{1}{2}+\underset{i=1,2,5,6}\max\rm{Im}\nu(\zeta_i)}\Big)\right).\\
\end{align}
As for $I_1$, we deduce
\begin{equation}
I_{1}=-\frac{1}{2\pi\mathrm{i}}\sum_{i=1,2,5,6}\oint_{\partial U_{\zeta_i}}\frac{t^{-1/2}}{2\sqrt{\theta^{''}(\zeta_{i})}(z-\zeta_{i})}m^{(out)}(s)m_{i,1}^{(pc)}{m^{(out)}(s)}^{-1}ds,
\end{equation}
Substitute \eqref{mpc} into the above formula to get the conclusion.
\end{proof}
\subsection{Analysis on the pure $\bar{\partial}$-Problem}\label{puredbar}
Now we define the function
\begin{equation}\label{m3}
    m^{(3)}(z)=m^{(2)}(z)(m^{rhp}(z))^{-1}.
\end{equation}
Then $m^{(3)}$ satisfies the following pure $\bar{\partial}$-Problem.
\begin{Dbar}\label{dbarproblem}
    Find a $2\times 2$ matrix-valued function $m^{(3)}(z):=m(x,t;z)$ such that
    \begin{itemize}
        \item[*] $m^{(3)}(z)$ is continuous in $\mathbb{C}$ and analytic in $\mathbb{C}\backslash\left(\underset{i,j=1,2,\cdots,8}\cup\Omega_{ij}\right)$;
        \item[*] asymptotic behavior
        \begin{equation}
              m^{(3)}(z)=I+\mathcal{O}(z^{-1}), \quad z\rightarrow\infty ;
        \end{equation}
        \item[*]For $z\in \mathbb{C}$, we have
        \begin{equation}
            \bar{\partial}m^{(3)}(z)=m^{(3)}(z)W^{(3)}(z);
        \end{equation}
        where $W^{(3)}=m^{rhp}(z)\bar{\partial}R^{(2)}(z)(m^{rhp}(z))^{-1}$,
        \begin{align}
          W^{(3)}(z)=\left\{
        \begin{aligned}
        &m^{rhp}(z)\begin{bmatrix} 1 & 0 \\  \bar{\partial}R_{0j}e^{-2\mathrm{i}t\theta} & 1 \end{bmatrix}(m^{rhp}(z))^{-1}, \quad z\in\Omega_{ij}, \quad i=0,7,8, \quad j=1,2,\\
        &m^{rhp}(z)\begin{bmatrix} 1 & \bar{\partial}R_{ij}e^{2\mathrm{i}t\theta} \\ 0 & 1 \end{bmatrix}(m^{rhp}(z))^{-1}, \quad z\in\Omega_{0j}, \quad i=0,7,8, \quad j=3,4,\\
        &m^{rhp}(z)\begin{bmatrix} 0 & \bar{\partial}R_{i1}e^{-2\mathrm{i}t\theta} \\ 0 & 0 \end{bmatrix}(m^{rhp}(z))^{-1}, \quad z\in\Omega_{i1}, i=1,2,\cdots,6,\\
        &m^{rhp}(z)\begin{bmatrix} 0 & 0 \\ \bar{\partial}R_{i2}e^{2\mathrm{i}t\theta} & 0 \end{bmatrix}(m^{rhp}(z))^{-1}, \quad z\in\Omega_{i2},  i=1,2,\cdots,6,\\
         &m^{rhp}(z)\begin{bmatrix} 0 & \bar{\partial}R_{i3}e^{-2\mathrm{i}t\theta} \\ 0 & 0 \end{bmatrix}(m^{rhp}(z))^{-1}, \quad z\in\Omega_{i3}, i=1,2,\cdots,6,\\
         &m^{rhp}(z)\begin{bmatrix} 0 & 0 \\ \bar{\partial}R_{i4}e^{2\mathrm{i}t\theta} & 0 \end{bmatrix}(m^{rhp}(z))^{-1}, \quad z\in\Omega_{i4},  i=1,2,\cdots,6,\\
         & \begin{bmatrix} 0 & 0 \\ 0 & 0 \end{bmatrix}, \quad elsewhere.
        \end{aligned}
        \right.
\end{align}
    \end{itemize}
\end{Dbar}
Now we consider the long time asymptotic behavior of $m^{(3)}$. The solution of $\bar{\partial}$-Problem \ref{dbarproblem} can be solved
by the following integral equation
\begin{equation}\label{intm3}
    m^{(3)}(z)=I-\frac{1}{\pi}\iint_{\mathbb{C}}\frac{m^{(3)}(s)W^{(3)}(s)}{s-z}dA(s),
\end{equation}
where $A(s)$ is the Lebesgue measure on $\mathbb{C}$. Denote $S$ as the Cauchy-Green integral operator
\begin{equation}\label{C-Gop}
    S[f](z)=-\frac{1}{\pi}\iint\frac{f(s)W^{(3)}(s)}{s-z}dA(s),
\end{equation}
then \eqref{intm3} can be written as the following operator equation
\begin{equation}
    (1-S)m^{(3)}(z)=I.
\end{equation}
To prove the existence of the operator at large time, we present the following lemma.
\begin{lemma}\label{estS}
The norm of the integral operator $S$ decay to zero as $t\rightarrow\infty$, and
\begin{equation}
\|S\|_{L^{\infty}\rightarrow L^{\infty}}=\mathcal{O}(t^{-\frac{1}{4}})+\mathcal{O}(t^{\frac{\alpha}{2}-\frac{1}{2}}),
\end{equation}
where $\alpha=\underset{i=1,2,\cdots,6}\max\alpha_i$.
\end{lemma}
\begin{proof}
For any $f\in L^{\infty}$, we have
\begin{equation}
\Vert Sf\Vert_{L^{\infty}}\leqslant\Vert f \Vert_{L^{\infty}}\frac{1}{\pi}\iint_{\mathbb{C}}\frac{|W^{(3)}(s)|}{|s-z|}dA(s),
\end{equation}
where
\begin{equation}
\vert W^{(3)}(s)\vert \leqslant \vert m^{rhp}(s)\vert^{2}\vert 1+s^{-2} \vert^{-1} \vert\bar{\partial} R^{(2)}(s)\vert.
\end{equation}
Recall $\bar{\partial}R^{(2)}(s)=0, s\in\mathbb{C}/\Omega$, so we only need to consider this estimate in $\Omega$, where
\begin{equation}
\Omega=\underset{j=1,2,3,4}\cup\left[\Omega_{0j}\cup\left(\underset{i=1,2,\cdots,8}\cup\Omega_{ij}\right)\right].
\end{equation}
Since $m^{rhp}(s)=m^{out}(I+s^{-1}E_1+\mathcal{O}(s^{-2}))$, we can bound $m^{rhp}$
\begin{equation}
    \vert m^{rhp}(s) \vert\lesssim 1+|s|^{-1}=c\sqrt{(1+|s|^{-1})^2}\lesssim \sqrt{1+|s|^{-2}}=|s|^{-1}\sqrt{1+|s|^2}=|s|^{-1}\langle s\rangle.
\end{equation}
Then we have
\begin{equation}\label{<s>est}
\vert m^{rhp}(s)\vert^{2}\vert 1+s^{-2} \vert^{-1}\lesssim\frac{|s|^{-2}\langle s\rangle^2}{\vert 1+s^{-2} \vert}
=\left\{
        \begin{aligned}
        &\mathcal{O}(1), \quad z\in\Omega_{ij}, \quad i=0,1,2,\cdots,6, j=1,2,3,4,\\
        &\frac{\langle s\rangle}{\vert s-{\rm i}\vert}, \quad z\in\Omega_{7j}, \quad j=1,2,3,4,\\
        &\frac{\langle s\rangle}{\vert s+{\rm i}\vert}, \quad z\in\Omega_{8j}, \quad j=1,2,3,4.\\
        \end{aligned}
        \right.
\end{equation}
Taking $z\in\Omega_{01}, \Omega_{11}, \Omega_{74}$ for examples, the other cases are similar.

We introduce an inequality which plays an vital role in our analysis. Make $s=z_{0}+u+iv$, $z=\alpha+i\eta$, $u,v,\alpha,\eta>0$ we have the inequality
\begin{align}\label{vital role inequ}
\left\Vert \frac{1}{s-z} \right\Vert_{L^{q}(v,\infty)}&\lesssim\left(\int_{\mathbb{R^+}}
\left[1+\left(\frac{u+z_0-\alpha}{v-\eta}\right)^2\right]^{-\frac{q}{2}}(v-\eta)^{-q}du\right)^{\frac{1}{q}}\nonumber\\
&=|v-\eta|^{\frac{1}{q}-1}\left(\int_{\mathbb{R^+}}
\left[1+\left(\frac{u+z_0-\alpha}{v-\eta}\right)^2\right]^{-\frac{q}{2}}d\left(\frac{u+z_{0}-\alpha}{v-\eta}\right)\right)^{1/q}\nonumber\\
&\overset{q>1}{\lesssim} \vert v-\eta \vert^{1/q-1}.
\end{align}

\underline{For $z\in \Omega_{01}$}, we make $z=\alpha+i\eta$, $s=0+u+iv$. Thanks to \eqref{<s>est}, we have
\begin{equation}
\begin{split}
&\frac{1}{\pi}\iint_{\Omega_{01}}\frac{|\bar{\partial}R_{01}||e^{-2{\rm i}t\theta}|}{|s-z|}dA(s)\nonumber\\
\lesssim &\iint_{\Omega_{01}}\frac{|\tilde{\rho}'({\rm Res})|e^{-ctv}}{|s-z|}dA(s)+\iint_{\Omega_{01}}\frac{|s|^{-\frac{1}{2}}e^{-ctv}}{|s-z|}dA(s)\nonumber\\
\lesssim & t^{-\frac{1}{2}}.
\end{split}
\end{equation}

\underline{For $z\in \Omega_{11}$}, For $s=\zeta_1+u+{\rm i}v\in\Omega_{11}$, $|s-\zeta_1|^{\frac{1}{2}}$ is bounded, then
\begin{align}
&\frac{1}{\pi}\iint_{\Omega_{11}}\frac{|\bar{\partial}R_{11}||e^{-2{\rm i}t\theta}|}{|s-z|}dA(s)\nonumber\\
\lesssim & \iint_{\Omega_{11}}\frac{e^{-ctv^2}}{|s-z|}dA(s)+\iint_{\Omega_{11}}\frac{|s-\zeta_1|^{-\alpha_1}e^{-ctv^2}}{|s-z|}dA(s)\nonumber\\
\lesssim & I_1+I_2.
\end{align}
First $I_1$ can be estimated as follows:
\begin{align}
I_1\lesssim &\int_{0}^{\frac{\zeta_1}{2}\tan{\phi}}e^{-ctv^2}\left\|\frac{1}{s-z}\right\|_{L^2}dv\nonumber\\
\lesssim &\int_0^{+\infty}\frac{e^{-ctv^2}}{\sqrt{v-\eta}}dv\nonumber\\
=&\int_0^\eta\frac{e^{-ctv^2}}{\sqrt{v-\eta}}dv+\int_{\eta}^{+\infty}\frac{e^{-ctv^2}}{\sqrt{\eta-v}}dv\nonumber\\
=&I_{11}+I_{12},
\end{align}
where
\begin{equation}
I_{11}\overset{w=\frac{v}{\eta}}=\int_0^1\frac{\sqrt{\eta}e^{-ctw^2\eta^2}}{\sqrt{1-w}}dw
\lesssim\int_0^1\frac{t^{-\frac{1}{4}}w^{-\frac{1}{2}}}{\sqrt{1-w}}dw\lesssim t^{-\frac{1}{4}}
\end{equation}
and
\begin{equation}
I_{12}\overset{w-\eta}=\int_0^{+\infty}\frac{e^{-ctw^2}}{\sqrt{w}}dw
\overset{\lambda=t^{\frac{1}{2}}w}=t^{-\frac{1}{4}}\int_0^{+\infty}\frac{e^{-c\lambda^2}}{\sqrt{\lambda}}d\lambda\lesssim t^{-\frac{1}{4}}.
\end{equation}
As for $I_2$, we have
\begin{align}
I_2&=\int_0^{\frac{\zeta_1}{2}\tan{\phi}}e^{-ctv^2}\left(\int_{v-\zeta_1}^{\zeta_1}\frac{|z-\zeta_1|^{-\alpha_1}}{|s-z|}du\right)dv\nonumber\\
&\lesssim\int_0^{+\infty}e^{-ctv^2}\left\|\frac{1}{s-z}\right\|_{L^q}\left\||z-\zeta_1|^{-\alpha_1}\right\|_{L^p}dv\nonumber\\
&\lesssim\int_0^{+\infty}e^{-ctv^2}|v-\eta|^{\frac{1}{q}-1}v^{\frac{1}{p}-\alpha_1}dv\nonumber\\
&=\int_0^{\eta}e^{-ctv^2}(\eta-v)^{\frac{1}{q}-1}v^{\frac{1}{p}-\alpha_1}dv+
\int_{\eta}^{+\infty}e^{-ctv^2}(v-\eta)^{\frac{1}{q}-1}v^{\frac{1}{p}-\alpha_1}dv\nonumber\\
&=I_{21}+I_{22},
\end{align}
where
\begin{align}
I_{21}&\overset{w=\frac{v}{\eta}}=\int_0^1 e^{-ct\eta^2w^2}(\eta-w\eta)^{\frac{1}{q}-1}(w\eta)^{\frac{1}{p}-\alpha_1}dw\nonumber\\
&=\int_0^1e^{-ct\eta^2w^2}\eta^{1-\alpha_1}(1-w)^{\frac{1}{q}-1}w^{\frac{1}{p}-\alpha_1}dw\nonumber\\
&\lesssim t^{\frac{\alpha_1}{2}-\frac{1}{2}}\lesssim t^{\frac{\alpha}{2}-\frac{1}{2}}
\end{align}
and
\begin{align}
I_{22}&\overset{w=\eta-v}=\int_0^{+\infty}e^{-ct(\eta+w)^2}w^{\frac{1}{q}-1}(\eta+w)^{\frac{1}{p}-\alpha_1}dw
\lesssim\int_0^{+\infty}e^{-ctw^2}w^{-\alpha_1}dw\nonumber\\
&\overset{y=tw^2}=\int_0^{+\infty}e^{-cy}t^{\frac{\alpha_1}{2}-\frac{1}{2}}y^{-\frac{\alpha_1}{2}-\frac{1}{2}}dy\lesssim t^{\frac{\alpha_1}{2}-\frac{1}{2}}\lesssim t^{\frac{\alpha}{2}-\frac{1}{2}}.
\end{align}

\underline{For $z=u+\mathrm{i}v\in \Omega_{74}$}, we have
\begin{equation}
\begin{split}
&\frac{1}{\pi}\iint_{\Omega_{74}}\frac{\langle s\rangle|\bar{\partial}R_{74}||e^{2{\rm i}t\theta}|}{|s-{\rm i}||s-z|}dA(s)\\
\lesssim &\iint_{\Omega_{74}}\frac{\langle s\rangle|\bar{\partial}R_{74}||e^{2{\rm i}t\theta}|\chi_{U({\rm i};\epsilon)}(|s|)}{|s-{\rm i}||s-z|}dA(s)+\iint_{\Omega_{74}}\frac{\langle s\rangle|\bar{\partial}R_{74}||e^{2{\rm i}t\theta}|\chi_{\Omega_{74}/U({\rm i};\epsilon)}(|s|)}{|s-{\rm i}||s-z|}dA(s)\\
=&I_3+I_4,\\
\end{split}
\end{equation}
where $\chi_{U({\rm i};\epsilon)}(|s|)+\chi_{\Omega_{74}/U({\rm i};\epsilon)}(|s|)$ is the partition of unity.
Note $1-v<u<\sqrt{1-v^2}$, we have
\begin{equation}
e^{-ct\left||s|^{-2}-1\right|v^2}
=e^{-ct\frac{|1-u^2-v^2|}{u^2+v^2}v^2}
\leq e^{-ct|1-u-v^2|v^2}
\leq e^{-ct|1-\sqrt{1-v^2}-v^2|v^2}.
\end{equation}
For $I_3$, the singularity at $z=\mathrm{i}$ can be balanced by \eqref{453} and $\langle s \rangle$ is bounded for $z\in\Omega_{74}$. Utilizing $e^{-z}\lesssim z^{-\frac{1}{4}}$, we have
\begin{equation}
\begin{split}
I_3&\lesssim\int_0^1\int_{1-v}^1\frac{1}{|s-z|}e^{-ct|1-\sqrt{1-v^2}-v^2|v^2}dudv\\
&\lesssim\int_0^1\left\|\frac{1}{|s-z|}\right\|_{L^2}\left(\int_{1-v}^1e^{-2ct|1-\sqrt{1-v^2}-v^2|v^2}du\right)^\frac{1}{2}dv\\
&\lesssim\int_0^1\frac{1}{\sqrt{|v-\eta|}}v^{\frac{1}{2}}e^{-ct|1-\sqrt{1-v^2}-v^2|v^2}dv\\
&\lesssim t^{-\frac{1}{4}}\int_0^1\frac{(\sqrt{1-v^2}-1+v^2)^{-\frac{1}{4}}}{\sqrt{|v-\eta|}}dv\\
&\lesssim t^{-\frac{1}{4}}.
\end{split}
\end{equation}
Thanks to \eqref{<s>est}, we obtain
\begin{equation}
\begin{split}
I_4&\lesssim\iint_{\Omega_{74}}\frac{c_1+c_2|z-1|^{\frac{1}{2}}}{|s-z|}|e^{2\mathrm{i}t\theta}|dA(s)+\iint_{\Omega_{74}}\frac{|z-1|^{-\frac{1}{2}}}{|s-z|}|e^{2\mathrm{i}t\theta}|dA(s)\\
&\lesssim I_{41}+I_{42},
\end{split}
\end{equation}
where $I_{41}$ is similar to $I_3$,
\begin{equation}
\begin{split}
I_{41}&\lesssim\int_0^1\int_{1-v}^1\frac{1}{|s-z|}e^{-ct|1-\sqrt{1-v^2}-v^2|v^2}dudv\\
&\lesssim t^{-\frac{1}{4}}
\end{split}
\end{equation}
and
\begin{equation}
\begin{split}
I_{42}&\lesssim\int_0^1\int_{1-v}^1\frac{|z-1|^{-\frac{1}{2}}}{|s-z|}e^{-ct|\sqrt{1-v^2}-1+v^2|v^2}dudv\\ &\lesssim\int_0^1\left\||z-1|^{-\frac{1}{2}}\right\|_{L^p}\left\|\frac{1}{|s-z|}\right\|_{L^q}e^{-ct|\sqrt{1-v^2}-1+v^2|v^2}dv\\
&\lesssim\int_0^1v^{\frac{1}{p}-\frac{1}{2}}|v-\eta|^{\frac{1}{q}-1}e^{-ct|\sqrt{1-v^2}-1+v^2|v^2}dv\\
&\lesssim\int_0^{+\infty}v^{\frac{1}{p}-\frac{1}{2}}|v-\eta|^{\frac{1}{q}-1}e^{-ct|\sqrt{|1-v^2|}-1+v^2|v^2}dv\\
&\overset{v=\eta w}\lesssim\int_0^{+\infty}w^{\frac{1}{p}+\frac{1}{2}}|1-w|^{\frac{1}{q}-1}e^{-ctw^2|\sqrt{|1-\eta w^2|}-1+\eta w^2|}dw\\
&\lesssim t^{-\frac{1}{4}}\int_0^{+\infty}w^{\frac{1}{p}}|1-w|^{\frac{1}{q}-1}|\sqrt{|1-\eta w^2|}-1+\eta w^2|^{-\frac{1}{4}}dw\\
&\lesssim t^{-\frac{1}{4}}.
\end{split}
\end{equation}
\end{proof}
Based on the above discussion, we have the following proposition.
\begin{proposition}
As $t\rightarrow\infty$, $(I-S)^{-1}$ exists, which implies $\bar{\partial}$ Problem \ref{dbarproblem} has an unique solution.
\end{proposition}
Aim at the asymptotic behavior of $m^{(3)}$, we make the asymptotic expansion
\begin{equation}
    m^{(3)}=I+z^{-1}m^{(3)}_{1}(x,t)+\mathcal{O}(z^{-2}), \quad {\rm as} \quad z\rightarrow \infty
\end{equation}
where
\begin{equation}\label{m31}
    m^{(3)}_{1}(x,t)=\frac{1}{\pi}\iint_{\mathbb{C}}m^{(3)}(s)W^{(3)}(s)dA(s).
\end{equation}
To recover the solution of \eqref{nmkdv}, we shall discuss the asymptotic behavior of $m^{(3)}_{1}(x,t)$, thus we have
the following proposition.
\begin{proposition} \label{m31guji}As $t\rightarrow \infty$,
    \begin{equation}
    \vert m^{(3)}_{1}(x,t) \vert \lesssim t^{-\frac{3}{4}}+t^{\frac{\alpha}{2}-1}.
    \end{equation}
\end{proposition}
\begin{proof} Take $z\in \Omega_{01}, \Omega_{11}, \Omega_{74}$ as examples.

\underline{For $z\in \Omega_{01}$}, we make $z=\alpha+i\eta$, $s=0+u+iv$. Thus,
\begin{equation}
\begin{split}
|m^{(3)}_{1}|&\lesssim \iint_{\Omega_{01}}|\tilde{\rho}'({\rm Res})|e^{-ctv}dA(s)+\iint_{\Omega_{01}}|s|^{-\frac{1}{2}}e^{-ctv}dA(s)\\
&=\int_{\mathbb{R}^{+}}\left\Vert \tilde{\rho}'({\rm Res}) \right\Vert_{L^{2}(\mathbb{R}^{+})}e^{-ctv}dv+\int_{\mathbb{R}^+} \left\Vert|s|^{-1/2} \right\Vert_{L^{p}(\mathbb{R}^{+})}\left\Vert e^{-ctv}\right\Vert_{L^{q}(\mathbb{R}^{+})}dv,\\
&\lesssim t^{-1}.
\end{split}
\end{equation}

\underline{For $z\in \Omega_{11}$}, we make $z=\alpha+i\eta$, $s=\zeta_1+u+{\rm i}v$. Then
\begin{equation}
|m^{(3)}_1(z)|\lesssim\iint_{\Omega_{11}}e^{-ctv^2}dA(s)+\iint_{\Omega_{11}}|z-\zeta_1|^{-\alpha_1}e^{-ctv^2}dA(s)=I_1+I_2,
\end{equation}
where
\begin{equation}
\begin{split}
\iint_{\Omega_{11}}e^{-ctv^2}dA(s)
&=\int_0^{\frac{\zeta_1}{2}\tan{\phi}}\int_{\frac{\zeta_1}{2}}^{\frac{\zeta_1}{2}+\frac{v}{\tan{\phi}}}e^{-ctv^2}dudv\\
&\lesssim
\int_0^{\frac{\zeta_1}{2}\tan{\phi}}\left(\int_{\frac{\zeta_1}{2}}^{\frac{\zeta_1}{2}+\frac{v}{\tan{\phi}}}e^{-2ctv^2}du\right)^{\frac{1}{2}}dv\\
&\lesssim \int_0^{+\infty}v^{\frac{1}{2}}e^{-ctv^2}dv\\
&\lesssim t^{-\frac{3}{4}}
\end{split}
\end{equation}
and
\begin{equation}
\begin{split}
I_2
&\lesssim \int_0^{+\infty}\left\||z-\zeta_1|^{-\alpha_1}\right\|_{L^p}\left\|e^{-ctv^2}\right\|_{L^q}dv\\
&\lesssim \int_0^{+\infty}v^{-\alpha_1+\frac{1}{p}}v^{\frac{1}{q}}e^{-ctv^2}dv\\
&\lesssim \int_0^{+\infty}v^{1-\alpha_1}e^{-ctv^2}dv\\
&\overset{w=t^{\frac{1}{2}}v}=\int_0^{+\infty}w^{1-\alpha_1}t^{-\frac{1}{2}(1-\alpha_1)}e^{-w^2}t^{-\frac{1}{2}}dw\\
&\lesssim t^{\frac{\alpha_1}{2}-1}\lesssim t^{\frac{\alpha}{2}-1}.
\end{split}
\end{equation}

\underline{For $z\in \Omega_{74}$}, for $s=u+{\rm i}v$, we have
\begin{equation}
\begin{split}
|m^{(3)}_1|\lesssim|&\frac{1}{\pi}\iint_{\Omega_{74}}\frac{\langle s\rangle|\bar{\partial}R_{74}||e^{2{\rm i}t\theta}|}{|s-{\rm i}|}dA(s)\\
\lesssim &\iint_{\Omega_{74}}\frac{\langle s\rangle|\bar{\partial}R_{74}||e^{2{\rm i}t\theta}|\chi_{U(1;\epsilon)}(|s|)}{|s-{\rm i}|}dA(s)+\iint_{\Omega_{74}}\frac{\langle s\rangle|\bar{\partial}R_{74}||e^{2{\rm i}t\theta}|\chi_{\Omega_{74}/U(1;\epsilon)}(|s|)}{|s-{\rm i}|}dA(s)\\
=&\hat{I}_3+\hat{I}_4,\\
\end{split}
\end{equation}
where $\chi_{U({\rm i};\epsilon)}(|s|)+\chi_{\Omega_{74}/U({\rm i};\epsilon)}(|s|)$ is the partition of unity.
With the help of \eqref{453}, we can balance the singularity at $z={\rm i}$. Utilizing $e^{-z}\lesssim z^{-\frac{3}{4}}$, we derive
\begin{equation}
\begin{split}
\hat{I}_3&\lesssim\int_0^1\int_{1-v}^1e^{-ct|\sqrt{1-v^2}-1+v^2|v^2}dudv\\
&\lesssim\int_0^1ve^{-ct|\sqrt{1-v^2}-1+v^2|v^2}dv\\
&\lesssim\int_0^{+\infty}ve^{-ct|\sqrt{|1-v^2|}-1+v^2|v^2}dv\\
&\overset{w=t^{\frac{1}{2}}v}\lesssim t^{-1}\int_0^{+\infty}we^{-ctw^2|\sqrt{|1-t^{-1}w^2|}-1+t^{-1}w^2|}dw\\
&\lesssim t^{-1}.
\end{split}
\end{equation}
Notice $\frac{\langle s\rangle}{\vert s-{\rm i}\vert}=\mathcal{O}(1)$ for $s\in\Omega_{74}/U({\rm i};\epsilon)$, we obtain
\begin{equation}
\begin{split}
\hat{I}_4&\lesssim\iint_{\Omega_{74}}(c_1+c_2|z-1|^{\frac{1}{2}})|e^{2\mathrm{i}\theta}|dA(s)+
\iint_{\Omega_{74}}|z-1|^{-\frac{1}{2}}|e^{2\mathrm{i}\theta}|dA(s)\\
&=\hat{I}_{41}+\hat{I}_{42},
\end{split}
\end{equation}
where
\begin{equation}
\begin{split}
\hat{I}_{41}&\lesssim\int_0^1\int_{1-v}^1e^{-ct(\sqrt{1-v^2}-1+v^2)v^2}dudv\\
&\lesssim t^{-1},
\end{split}
\end{equation}
which is similar to $\hat{I}_3$.
\begin{equation}
\begin{split}
\hat{I}_{42}&\lesssim\int_0^1\int_{1-v}^1|z-1|^{-\frac{1}{2}}e^{-ct|\sqrt{1-v^2}-1+v^2|v^2}dudv\\
&\lesssim\left\||z-1|^{-\frac{1}{2}}\right\|_{L^p}\left(\int_{1-v}^1e^{-cqt|\sqrt{1-v^2}-1+v^2|v^2}
du\right)^{\frac{1}{q}}dv\\
&\lesssim\int_0^1v^{\frac{1}{p}-\frac{1}{2}}v^{\frac{1}{q}}e^{-ct|\sqrt{1-v^2}-1+v^2|v^2}dv\\
&=\int_0^1v^{\frac{1}{2}}e^{-ct|\sqrt{1-v^2}-1+v^2|v^2}dv\\
&\lesssim\int_0^{+\infty}v^{\frac{1}{2}}e^{-ct|\sqrt{|1-v^2|}-1+v^2|v^2}dv\\
&\overset{w=t^{\frac{1}{4}}v^{\frac{1}{2}}}\lesssim t^{-\frac{3}{4}}\int_0^{+\infty}w^2e^{-ctw^4
|\sqrt{|1-t^{-1}w^4|}-1+t^{-1}w|}dw\\
&\lesssim t^{-\frac{3}{4}}.
\end{split}
\end{equation}
\end{proof}
\hspace*{\parindent}

\section{Deformation of  RH Problem in region  $\xi>6$}\label{openlens2}

In this section, we will discuss some results for the case $\xi>6$. Different from the case of $\xi<-6$, the four phase points are not on the jump contours besides $\pm 1$ in this case. It implies that the four phase points will not contribute to long-time behavior. Next, we will discuss below with the similar steps in section \ref{openlens1}.

The jump matrix $v(z)$ allows the following factorization:
\begin{equation}
 v(x,t;z)=
    \begin{bmatrix} 1 & 0 \\  \frac{\rho}{1-\rho\tilde{\rho}}e^{-2\mathrm{i}t\theta} & 1 \end{bmatrix}
    \begin{bmatrix} 1-\rho\tilde{\rho} & 0 \\ 0 & \frac{1}{1-\rho\tilde{\rho}} \end{bmatrix}
    \begin{bmatrix} 1 & -\frac{\tilde{\rho}}{1-\rho\tilde{\rho}}e^{2\mathrm{i}t\theta} \\ 0  & 1 \end{bmatrix}, \quad z\in \Sigma.
\end{equation}
Then we choose
\begin{equation}\
T(z):=T(z;\xi)=\prod_{k\in\triangle_{1}}\prod_{l\in\triangle_{2}}\frac{(z+z_{k}^{-1})(z-\bar{z}_{k}^{-1})(z-\mathrm{i}\omega_{l}^{-1})}
{(zz_{k}^{-1}-1)(z\bar{z}_{k}^{-1}+1)(\mathrm{i}\omega_{l}^{-1}z+1)}{\rm exp}[\mathrm{i}\int_{\Sigma}\nu(s)(\frac{1}{s-z}-\frac{1}{2s})ds],
\end{equation}
which allows the following estimate as $z\rightarrow \pm 1$:
\begin{equation}
\Big|T(z)-T _{\pm 1}(\pm 1)|z\mp 1|^{-2\mathrm{i}\nu(\pm 1)}\Big|\leq c|z\mp 1|^{\frac{1}{2}},
\end{equation}
where
\begin{align}
&T_{\pm 1}(\pm 1)=\prod_{k\in\triangle_{1}}\prod_{l\in\triangle_{2}}\frac{(\pm 1+z_{k}^{-1})(\pm 1-\bar{z}_{k}^{-1})(\pm 1-\mathrm{i}\omega_{l}^{-1})}{(\pm z_{k}^{-1}-1)(\pm\bar{z}_{k}^{-1}+1)(\pm 1\mathrm{i}\omega_{l}^{-1}+1)}{\rm exp}[\mathrm{i}\beta_{\pm 1}(\pm 1)],\\
&\beta_{-1}(z,\xi)=\nu(-1)\log{|z(z+2)|}+\int_{-\infty}^{-1}\frac{\nu(s)+\chi_1(s)\nu(-1)}{s-z}ds+
\int^{+\infty}_{-1}\frac{\nu(s)-\chi_2(s)\nu(-1)}{s-z}ds,\\
&\beta_{+1}(z,\xi)=\nu(1)\log{|z(z-2)|}+\int_{-\infty}^1\frac{\nu(s)+\chi_3(s)\nu(1)}{s-z}ds+\int^{+\infty}_1\frac{\nu(s)-\chi_4(s)\nu(1)}{s-z}ds.
\end{align}
Utilizing transformation
\begin{equation}
m^{(1)}(z)=T(\infty)^{\sigma_{3}}m(z)G(z)T(z)^{-\sigma_{3}},
\end{equation}
it's easy to verify $m^{(1)}(z)$ still satisfies the RHP \ref{rhp1} formally.

\subsection{Opening lenses}
We fix an angle $\theta_{0}>0$ sufficiently small such that the set $\{z\in \mathbb{C}: \left| \frac{{\rm Re}z}{z}\right|, \left| \frac{{\rm Re}z-1}{z}\right|>\cos{\theta_0}\}, $ does not intersect the discrete spectrums set $\mathcal{Z}\cup\hat{\mathcal{Z}}$. For any $-\infty<\xi<6$, let
\begin{equation}
\phi(\xi)={\rm min}\left\{\theta_0, \frac{\pi}{4}\right\}
\end{equation}
and define
\begin{equation}\label{Sigmajump}
\Sigma_{jum}=\underset{i,j=1,2,3,4}\cup\left(\Sigma_{0j}\cup\Sigma_{ij}\right),
\end{equation}
which is shown in Figure \ref{alljunpcontours2} and consists of rays and other line segments or arcs.
\begin{figure}[H]
  \begin{center}
  \begin{tikzpicture}[node distance=2cm]
          \draw [->,dashed](-6,0)--(6,0)node[right]{ \textcolor{black}{${\rm Re} z$}};;
          \draw [](0,-5)--(0,5)  node[above, scale=1] {{\rm Im}z};
          \draw [->,dashed](0,0)circle(3cm);

          \node  [right]  at (0,1.5) {\footnotesize $\zeta_{1}$};
          \node  [right]  at (0,-1.5) {\footnotesize $\zeta_{2}$};
          \node  [right] at (0,4.5) {\footnotesize $\zeta_{3}$};
          \node  [right]  at (0,-4.5) {\footnotesize $\zeta_{4}$};
          \node  [below]  at (3,0) {\footnotesize $1$};
          \node  [below] at (-3,0) {\footnotesize $-1$};
          \node  [below]  at (0,0) {\footnotesize $O$};

          \draw[fill] (0,1.5) circle [radius=0.04];
          \draw[fill] (0,-1.5) circle [radius=0.04];
          \draw[fill] (0,4.5) circle [radius=0.04];
          \draw[fill] (0,-4.5) circle [radius=0.04];
          \draw[fill] (3,0) circle [radius=0.04];
          \draw[fill] (-3,0) circle [radius=0.04];
          \draw[fill] (0,0) circle [radius=0.04];

          \draw [color=blue](0,0)--(1.9,0.4)  node[right, scale=1] {};
          \draw [color=blue](1.9,0.4)--(3,0)  node[right, scale=1] {};
          \draw [color=blue](3,0)--(6,-1)  node[right, scale=1] {};

          \draw [color=red](0,0)--(1.9,-0.4)  node[right, scale=1] {};
          \draw [color=red](1.9,-0.4)--(3,0)  node[right, scale=1] {};
          \draw [color=red](3,0)--(6,1)  node[right, scale=1] {};

          \draw [color=blue](0,0)--(-1.9,0.4)  node[right, scale=1] {};
          \draw [color=blue](-1.9,0.4)--(-3,0)  node[right, scale=1] {};
          \draw [color=blue](-3,0)--(-6,-1)  node[right, scale=1] {};

          \draw [color=red](0,0)--(-1.9,-0.4)  node[right, scale=1] {};
          \draw [color=red](-1.9,-0.4)--(-3,0)  node[right, scale=1] {};
          \draw [color=red](-3,0)--(-6,1)  node[right, scale=1] {};

          \draw[color=red] (3.75,0.27) arc (6:102:3.14);
          \draw[color=red] (-3.75,0.27) arc (174:78:3.14);
          \draw[color=blue] (-3.75,-0.27) arc (186:282:3.14);
          \draw[color=blue] (3.75,-0.27) arc (354:258:3.14);

          \draw [color=blue] [-latex] (0,0) to (0.95,0.2);
          \draw [color=blue] [-latex] (6,-1) to (4.5,-0.5);
          \draw [color=blue] [-latex] (1.9,0.4) to (2.45,0.2);
          \draw [color=blue] [-latex] (3,0) to (2.45,0.2);
          \draw [color=blue] [-latex] (3,0) to (3.375,-0.135);
          \draw [color=blue] [-latex] (3.75,-0.27) to (3.375,-0.135);

          \draw [color=red] [-latex] (0,0) to (0.95,-0.2);
          \draw [color=red] [-latex] (6,1) to (4.5,0.5);
          \draw [color=red] [-latex] (1.9,-0.4) to (2.45,-0.2);
          \draw [color=red] [-latex] (3,0) to (2.45,-0.2);
          \draw [color=red] [-latex] (3,0) to (3.375,0.135);
          \draw [color=red] [-latex] (3.75,0.27) to (3.375,0.135);

          \draw [color=blue] [-latex] (-1.9,0.4) to (-0.95,0.2);
          \draw [color=blue] [-latex] (-1.9,0.4) to (-2.45,0.2);
          \draw [color=blue] [-latex] (-3,0) to (-2.45,0.2);
          \draw [color=blue] [-latex] (-3,0) to (-4.5,-0.5);
          \draw [color=blue] [-latex] (-3,0) to (-3.375,-0.135);
          \draw [color=blue] [-latex] (-3.75,-0.27) to (-3.375,-0.135);

          \draw [color=red] [-latex] (0,0) to (0.95,-0.2);
          \draw [color=red] [-latex] (-1.9,-0.4) to (-0.95,-0.2);
          \draw [color=red] [-latex] (-1.9,-0.4) to (-2.45,-0.2);
          \draw [color=red] [-latex] (-3,0) to (-2.45,-0.2);
          \draw [color=red] [-latex] (-3,0) to (-4.5,0.5);
          \draw [color=red] [-latex] (-3,0) to (-3.375,0.135);
          \draw [color=red] [-latex] (-3.75,0.27) to (-3.375,0.135);

          \draw [color=blue] [-latex]  (0,-3) to  [out=191, in=323]  (-2.5,-2.44);
          \draw [color=blue] [-latex]  (3.75,-0.27) to  [out=259, in=37]  (2.5,-2.44);
          \draw [color=red] [-latex]  (3.75,0.27) to  [out=99, in=323]  (2.5,2.44);
          \draw [color=red] [-latex]  (0,3) to  [out=169, in=37]  (-2.5,2.44);

          \draw[dashed](-1.9,0.4)--(-1.9,-0.4);
          \draw[dashed](1.9,0.4)--(1.9,-0.4);

         \node [above]  at (1,0.2) {\textcolor{blue}{\tiny $\Sigma_{01}$}};
         \node [above]  at (-1,0.2) {\textcolor{blue}{\tiny $\Sigma_{02}$}};
         \node [below]  at (-1,-0.2) {\textcolor{red}{\tiny $\Sigma_{03}$}};
         \node [below]  at (1,-0.2) {\textcolor{red}{\tiny $\Sigma_{04}$}};

         \node [above]  at (2.45,0.2) {\textcolor{blue}{\tiny $\Sigma_{11}$}};
         \node [above]  at (3.4,0.2) {\textcolor{red}{\tiny $\Sigma_{12}$}};
         \node [below]  at (3.4,-0.2) {\textcolor{blue}{\tiny $\Sigma_{13}$}};
         \node [below]  at (2.45,-0.2) {\textcolor{red}{\tiny $\Sigma_{14}$}};

         \node [above]  at (-2.45,0.2) {\textcolor{blue}{\tiny $\Sigma_{21}$}};
         \node [above]  at (-3.4,0.2) {\textcolor{red}{\tiny $\Sigma_{22}$}};
         \node [below]  at (-3.4,-0.2) {\textcolor{blue}{\tiny $\Sigma_{23}$}};
         \node [below]  at (-2.45,-0.2) {\textcolor{red}{\tiny $\Sigma_{24}$}};

         \node [right]  at (1.2,0.15) {\textcolor{blue}{\tiny $\Omega_{01}$}};
         \node [left]  at (-1.2,0.15) {\textcolor{blue}{\tiny $\Omega_{02}$}};
         \node [left]  at (-1.2,-0.15) {\textcolor{red}{\tiny $\Omega_{03}$}};
         \node [right]  at (1.2,-0.15) {\textcolor{red}{\tiny $\Omega_{04}$}};

         \node [right]  at (1.9,0.15) {\textcolor{blue}{\tiny $\Omega_{11}$}};
         \node [left]  at (4.3,0.15) {\textcolor{red}{\tiny $\Omega_{12}$}};
         \node [left]  at (4.3,-0.15) {\textcolor{blue}{\tiny $\Omega_{13}$}};
         \node [right]  at (1.9,-0.1) {\textcolor{red}{\tiny $\Omega_{14}$}};

         \node [left]  at (-1.9,0.1) {\textcolor{blue}{\tiny $\Omega_{21}$}};
         \node [right]  at (-4.3,0.15) {\textcolor{red}{\tiny $\Omega_{22}$}};
         \node [right]  at (-4.3,-0.15) {\textcolor{blue}{\tiny $\Omega_{23}$}};
         \node [left]  at (-1.9,-0.1) {\textcolor{red}{\tiny $\Omega_{24}$}};

         \node [right]  at (-3.3,1.5) {\textcolor{red}{\tiny $\Omega_{33}$}};
         \node [left]  at (3.3,1.5) {\textcolor{red}{\tiny $\Omega_{34}$}};

         \node [left]  at (-3.4,1.5) {\textcolor{red}{\tiny $\Sigma_{33}$}};
         \node [right]  at (3.4,1.5) {\textcolor{red}{\tiny $\Sigma_{34}$}};

         \node [right]  at (3.4,-1.5) {\textcolor{blue}{\tiny $\Sigma_{31}$}};
         \node [left]  at (-3.4,-1.5) {\textcolor{blue}{\tiny $\Sigma_{32}$}};

         \node [left]  at (3.3,-1.5) {\textcolor{blue}{\tiny $\Omega_{31}$}};
         \node [right]  at (-3.3,-1.5) {\textcolor{blue}{\tiny $\Omega_{32}$}};
    \end {tikzpicture}
    \caption{\footnotesize The blue curves are the opening contours in region $\{z\in\mathbb{C}: |e^{-2\mathrm{i}t\theta(z)}|\rightarrow0\}$ while the red curves are the opening contours in region $\{z\in\mathbb{C}: |e^{2\mathrm{i}t\theta(z)}|\rightarrow0\}$. These arrows represent directions of jump contours.}.
    \label{alljunpcontours2}
  \end{center}
\end{figure}
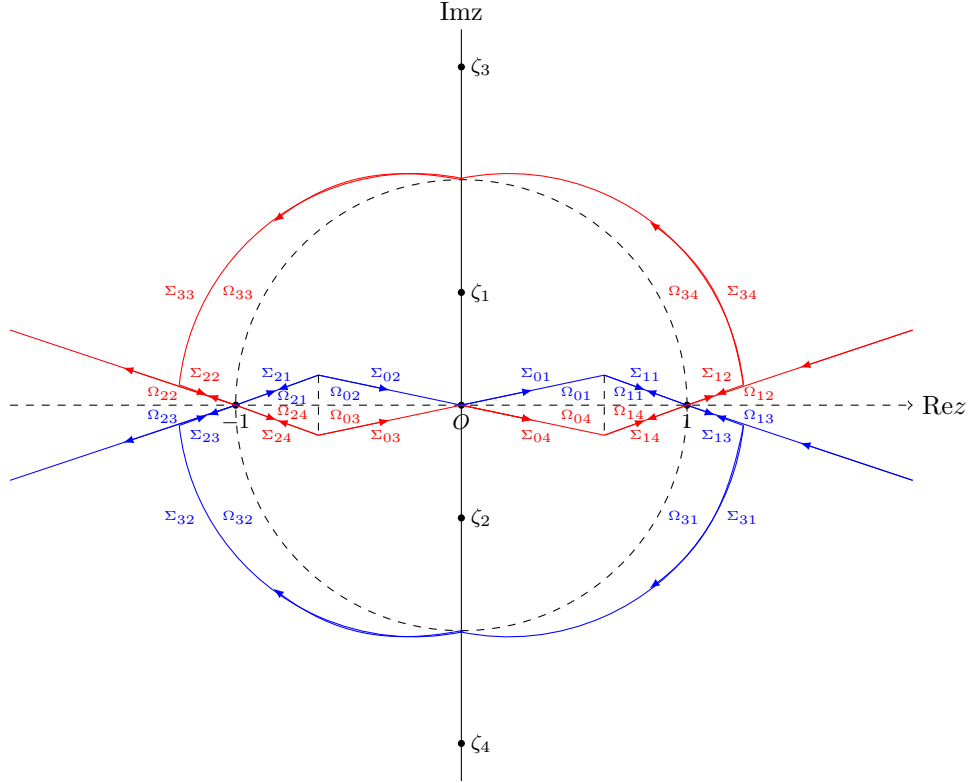

\begin{lemma}\label{theta21}
Set $\xi=\frac{x}{t}$  and let $6<\xi<+\infty$. Then for $z=u+\mathrm{i}v\in\Omega_{0j},\Omega_{ij}, i,j=1,2,3,4$, the
phase $\theta(z)$ defined in \eqref{phasefunc} satisfies
\begin{equation}
\begin{split}
& {\rm Re}[2\mathrm{i}\theta(z)]\geq cv, \quad z\in\Omega_{k1}, \Omega_{k2}, \Omega_{i1}, \Omega_{i3}, k=0,3, i=1,2,\\
& {\rm Re}[2\mathrm{i}\theta(z)]\leq -cv, \quad z\in\Omega_{k3}, \Omega_{k4}, \Omega_{i2}, \Omega_{i4}, k=0,3, i=1,2,
\end{split}
\end{equation}
where $c=c(\xi)>0$.
\end{lemma}
\begin{proof}
Take $z\in\Omega_{01}$ as an example. For $z=u+\mathrm{i}v\in\Omega_{01}$, we have
\begin{equation}
{\rm Re}[2\mathrm{i}\theta(z)]=-v(1-|z|^{-2})[\xi-3+(1+|z|^{-2}+|z|^{-4})(3u^2-v^2)],
\end{equation}
where
\begin{equation}
\xi-3+(1+|z|^{-2}+|z|^{-4})(3u^2-v^2)\geq 3+3(3u^2-v^2)\geq3.
\end{equation}
Thus,
\begin{equation}
{\rm Re}[2\mathrm{i}\theta(z)]\geq cv,
\end{equation}
with $c=c(\xi)>0$.
\end{proof}
Redefine
\begin{equation}
\begin{split}
&l_0^+=(0,z_1), \quad l_0^-=(-z_1,0),\\
&l_1^+=(z_1,1), \quad l_1^-=(1,+\infty), \quad l_2^+=(-1,-z_1), \quad l_1^-=(-\infty,-1),\\
&\gamma_k=\left\{z\in \mathbb{C}: z=e^{\mathrm{i}w}, \frac{(k-1)\pi}{2}\leq w\leq\frac{k\pi}{2}\right\}, \quad k=1,2,3,4,
\end{split}
\end{equation}
and
\begin{align}
\mathcal{R}^{(2)}(z)=\left\{
        \begin{aligned}
    &\begin{bmatrix} 1 & 0 \\  R_{ij}e^{-2\mathrm{i}t\theta} & 1 \end{bmatrix}, \quad z\in\Omega_{ij}, \quad i=0,3, j=1,2,\\
    &\begin{bmatrix} 1 & -R_{ij}e^{2\mathrm{i}t\theta} \\ 0 & 1 \end{bmatrix}^{-1}, \quad z\in\Omega_{ij}, \quad i=0,3, j=3,4,\\
    &\begin{bmatrix} 1 & 0 \\ R_{i1}e^{-2\mathrm{i}t\theta} & 1 \end{bmatrix}, \quad z\in\Omega_{i1}, \quad   i=1,2,\\
    &\begin{bmatrix} 1 & -R_{i2}e^{2\mathrm{i}t\theta} \\ 0 & 1 \end{bmatrix}^{-1}, \quad z\in\Omega_{i2}, \quad   i=1,2,\\
    &\begin{bmatrix} 1 & 0 \\ R_{i3}e^{-2\mathrm{i}t\theta} & 1 \end{bmatrix}, \quad z\in\Omega_{i3}, \quad   i=1,2,\\
    &\begin{bmatrix} 1 & -R_{i4}e^{2\mathrm{i}t\theta} \\ 0 & 1 \end{bmatrix}^{-1}, \quad z\in\Omega_{i4}, \quad   i=1,2,\\
    &I, \quad elsewhere,
    \end{aligned}
        \right.
\end{align}
where the functions $R_{0j}$ is the same as defined in Proposition \ref{R0j} and $R_{ij}, i,j=1,2,3,4$ are defined in the following proposition.
\begin{proposition}
$R_{ij}: \overline{\Omega}_{ij}\rightarrow \mathbb{C}$, $i=1,2,\cdots,8, j=1,2,3,4$ are continuous on $\overline{\Omega}_{ij}$ with boundary values:
\begin{itemize}
\item[(\romannumeral1)]
For $i=1,2$,
\begin{align}
&R_{ij}(z)=\left\{
        \begin{aligned}
    &r_3(z)T_-^{-2}(z), \quad z\in l_i^+, l_i^-,\\
    &r_3(\pm1)T_{\pm1}^{-2}(\pm1)|z\mp 1|^{4\mathrm{i}\nu(\pm1)}, \quad z\in\Sigma_{ij}, \quad j=1,3,
    \end{aligned}
        \right.\\
&R_{ij}(z)=\left\{
        \begin{aligned}
    &r_4(z)T^{2}(z), \quad z\in l_i^+,\\
    &r_4(\pm1)T_{\pm1}^{2}(\pm1)|z\mp 1|^{-4\mathrm{i}\nu(\pm1)}, \quad z\in\Sigma_{ij}, \quad j=2,4,
    \end{aligned}
        \right.
\end{align}
in which take $+1$ when $i=1$ and take $-1$ when $i=2$.
\item[(\romannumeral2)]
For $i=3$,
\begin{align}
&R_{ij}(z)=\left\{
        \begin{aligned}
    &r_3(z)T_-^{-2}(z), \quad z\in \gamma_1, \gamma_4,\\
    &r_3(\pm1)T_{\pm1}^{-2}(\pm1)|z\mp 1|^{4\mathrm{i}\nu(\pm1)}(1-\chi_{\mathcal{Z}}(z)), \quad z\in\Sigma_{3j}, j=1,2,
    \end{aligned}
        \right.\\
&R_{ij}(z)=\left\{
        \begin{aligned}
    &r_4(z)T_+^{2}(z), \quad z\in \gamma_2, \gamma_3,\\
    &r_4(\pm1)T_{\pm1}^{2}(\pm1)|z\mp 1|^{-4\mathrm{i}\nu(\pm1)}(1-\chi_{\mathcal{Z}}(z)), \quad z\in\Sigma_{3j}, j=3,4,
    \end{aligned}
        \right.
\end{align}
in which take $+1$ when $j=1,4$ and take $-1$ when $j=2,3$.
\end{itemize}
Moreover, $R_{ij}(z), i,j=1,2,3,4$ have following properties:
\begin{equation}\label{4.52}
\begin{split}
&|R_{ij}(z)|\leq c_1+c_2|1+z^{2}|^{-\frac{1}{4}}, \quad z\in\Omega_{ij},\\
& |\bar{\partial}R_{ij}(z)|\leq c_1+c_2|z\mp1|^{\frac{1}{2}}+c_3|z\mp1|^{-\frac{1}{2}}, \quad z\in\Omega_{ij},
\end{split}
\end{equation}
and when $z\rightarrow\pm \mathrm{i}$,
\begin{equation}\label{4.53}
\begin{split}
& |\bar{\partial}R_{ij}(z)|\leq c|z-\mathrm{i}|, \quad z\in\Omega_{3j}, \quad j=1,4,\\
& |\bar{\partial}R_{ij}(z)|\leq c|z+\mathrm{i}|, \quad z\in\Omega_{3j}, \quad j=2,3.
\end{split}
\end{equation}
\end{proposition}
\begin{proof}
The proof is similar to Proposition \ref{estRij}.
\end{proof}
Define $\Sigma^{(2)}=\Sigma^{'}_{jum}\cup\Sigma_{cir}$,  where
\begin{equation}
\Sigma^{'}_{jum}=\underset{j=1,2,3,4}\bigcup \left(\Sigma_{3j}\cup (\underset{i=1,2}\cup \Sigma^{'}_{ij})\right),
\end{equation}
can be referred in Figure \ref{Sigmajump'2}.
We get a RH problem of $m^{(2)}(z)$ by transformation
\begin{equation}
m^{(2)}(z)=m^{(1)}(z)\mathcal{R}^{(2)}(z),
\end{equation}
which has the same form as RH problem \ref{rhp2} with different jump matrix as:

\begin{align}
         v^{(2)}=[\mathcal{R}^{(2)}_{-}]^{-1}v^{(1)}\mathcal{R}^{(2)}_{+}=\left\{
        \begin{aligned}
        &\begin{bmatrix} 1 & 0 \\ (R_{11}-R_{01})e^{-2\mathrm{i}t\theta} & 1 \end{bmatrix}, \quad z\in\Sigma^{'}_{11},\\
        &\begin{bmatrix} 1 & (R_{14}-R_{04})e^{2\mathrm{i}t\theta} \\ 0 & 1 \end{bmatrix}, \quad z\in\Sigma^{'}_{14},\\
        &\begin{bmatrix} 1 & 0 \\ (R_{12}-R_{02})e^{-2\mathrm{i}t\theta} & 1 \end{bmatrix}, \quad z\in\Sigma^{'}_{21},\\
        &\begin{bmatrix} 1 & (R_{24}-R_{03})e^{2\mathrm{i}t\theta} \\ 0 & 1 \end{bmatrix}, \quad z\in\Sigma^{'}_{24},\\
        &\begin{bmatrix} 1 & R_{i2}e^{2\mathrm{i}t\theta} \\ 0 & 1 \end{bmatrix}, \quad z\in\Sigma^{'}_{i2}, \quad   i=1,2,\\
        &\begin{bmatrix} 1 & 0 \\ R_{i3}e^{-2\mathrm{i}t\theta} & 1 \end{bmatrix}, \quad z\in\Sigma^{'}_{i3}, \quad   i=1,2,\\
        &\begin{bmatrix} 1 & 0 \\  -R_{ij}e^{-2\mathrm{i}t\theta} & 1 \end{bmatrix}, \quad z\in\Sigma_{3j}, \quad j=1,2,\\
        &\begin{bmatrix} 1 & R_{ij}e^{2\mathrm{i}t\theta} \\ 0 & 1 \end{bmatrix}, \quad z\in\Sigma_{3j}, \quad j=3,4,\\
        &\begin{bmatrix} 1 & 0 \\ -\frac{A[\eta_k]}{z-\eta_k}T^{-2}(z)e^{-2\mathrm{i}t\theta(\eta_k)} & 1 \end{bmatrix}, \quad |z-\eta_k|=\varrho, \quad k\in \nabla/\Lambda,\\
        &\begin{bmatrix} 1 & -\frac{z-\eta_k}{A[\eta_k]}T^2(z)e^{2\mathrm{i}t\theta(\eta_k)} \\ 0 & 1 \end{bmatrix}, \quad |z-\eta_k|=\varrho, \quad k\in \triangle/\Lambda,\\
        &\begin{bmatrix} 1 & -\frac{A[\hat{\eta}_k]}{z-\hat{\eta}_k}T^2(z)e^{2\mathrm{i}t\theta(\hat{\eta}_k)} \\ 0 & 1 \end{bmatrix}, \quad |z-\hat{\eta}_k|=\varrho, \quad k\in \nabla/\Lambda,\\
        &\begin{bmatrix} 1 & 0 \\ -\frac{z-\hat{\eta}_k}{A[\hat{\eta}_k]}T^{-2}(z)e^{-2\mathrm{i}t\theta(\hat{\eta}_k)} & 1 \end{bmatrix}, \quad |z-\hat{\eta}_k|=\varrho, \quad k\in \triangle/\Lambda.
        \end{aligned}
        \right.
     \end{align}

\begin{figure}[H]
  \begin{center}
  \begin{tikzpicture}[node distance=2cm]
          \draw [->,dashed](-6,0)--(6,0)node[right]{ \textcolor{black}{${\rm Re} z$}};;
          \draw [](0,-5)--(0,5)  node[above, scale=1] {{\rm Im}z};
          \draw [->,dashed](0,0)circle(3cm);

          \node  [below]  at (3,0) {\footnotesize $1$};
          \node  [below] at (-3,0) {\footnotesize $-1$};
          \node  [below]  at (0,0) {\footnotesize $O$};

          \draw[fill] (3,0) circle [radius=0.04];
          \draw[fill] (-3,0) circle [radius=0.04];
          \draw[fill] (0,0) circle [radius=0.04];

          \draw [color=blue](3.75,-0.27)--(6,-1)  node[right, scale=1] {};
          \draw [color=red](3.75,0.27)--(6,1)  node[right, scale=1] {};
          \draw [color=blue](-3.75,-0.27)--(-6,-1)  node[right, scale=1] {};
          \draw [color=red](-3.75,0.27)--(-6,1)  node[right, scale=1] {};

          \draw[color=red] (3.75,0.27) arc (6:102:3.14);
          \draw[color=red] (-3.75,0.27) arc (174:78:3.14);
          \draw[color=blue] (-3.75,-0.27) arc (186:282:3.14);
          \draw[color=blue] (3.75,-0.27) arc (354:258:3.14);

          \draw [color=blue](1.9,0.4)--(1.9,0)  node[right, scale=1] {};
          \draw [color=blue](-1.9,0.4)--(-1.9,0)  node[right, scale=1] {};

          \draw [color=red](1.9,-0.4)--(1.9,0)  node[right, scale=1] {};
          \draw [color=red](-1.9,-0.4)--(-1.9,0)  node[right, scale=1] {};

         \node [right]  at (1.8,0.2) {\textcolor{blue}{\tiny $\Sigma^{'}_{11}$}};
         \node [above]  at (4.5,0.45) {\textcolor{red}{\tiny $\Sigma^{'}_{12}$}};
         \node [below]  at (4.5,-0.45)  {\textcolor{blue}{\tiny $\Sigma^{'}_{13}$}};
         \node [right]  at (1.8,-0.2) {\textcolor{red}{\tiny $\Sigma^{'}_{14}$}};

         \node [left]  at (-1.8,0.2) {\textcolor{blue}{\tiny $\Sigma^{'}_{21}$}};
         \node [above]  at (-4.5,0.45) {\textcolor{red}{\tiny $\Sigma^{'}_{22}$}};
         \node [below]  at (-4.5,-0.45) {\textcolor{blue}{\tiny $\Sigma^{'}_{23}$}};
         \node [left]  at (-1.8,-0.2) {\textcolor{red}{\tiny $\Sigma^{'}_{24}$}};

         \node [left]  at (-3.4,1.5) {\textcolor{red}{\tiny $\Sigma_{33}$}};
         \node [right]  at (3.4,1.5) {\textcolor{red}{\tiny $\Sigma_{34}$}};

         \node [right]  at (3.4,-1.5) {\textcolor{blue}{\tiny $\Sigma_{31}$}};
         \node [left]  at (-3.4,-1.5) {\textcolor{blue}{\tiny $\Sigma_{32}$}};
    \end {tikzpicture}
    \caption{\footnotesize The jump contour $\Sigma^{'}_{jump}$.}.
    \label{Sigmajump'2}
  \end{center}
\end{figure}
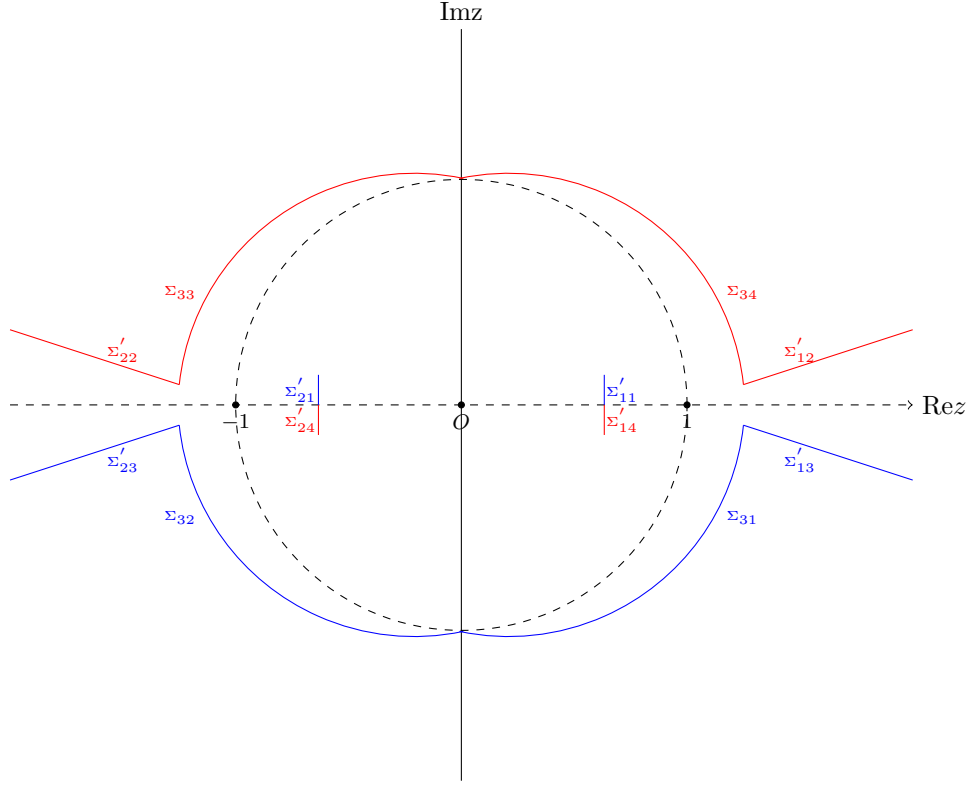

\subsection{Decomposition of mixed $\bar{\partial}$-RH problem}
Like the case $-\infty<\xi<-6$, we decompose $m^{(2)}(z)$ into the following structure
\begin{align}
m^{(2)}(z)=\left\{
        \begin{aligned}
    &\bar{\partial}\mathcal{R}^{(2)}\equiv0\longrightarrow m^{rhp}(z),\\
    &\bar{\partial}\mathcal{R}^{(2)}\neq0\longrightarrow m^{(3)}(z)=m^{(2)}(z)[m^{rhp}(z)]^{-1}.
    \end{aligned}
        \right.
\end{align}
\begin{remark}
In this case, there is no phase point on the jump contour $\Sigma_{jump}^{'}$. Thus, we construct the solution  $m^{rhp}(z)$ of the RH problem as
\begin{equation}
m^{rhp}(z)=E(z)m^{sol}(z)=E(z)m^{err}(z)m^\Lambda(z).
\end{equation}
In which, $m^{sol}(z)$ is given in \eqref{msol}. As for $E(z)$, it satisfies RH problem \ref{rhpE} with jump matrix
\begin{equation}
v^{E}(z)=m^{sol}(z)v^{(2)}(z)[m^{sol}(z)]^{-1}, \quad z\in\Sigma_{jum}^{'},
\end{equation}
and
\begin{equation}
E(z)=I+z^{-1}\mathcal{O}(e^{-ct})+\mathcal{O}(z^{-2}).
\end{equation}
\end{remark}

\subsection{Analysis on the pure $\bar{\partial}$-Problem}
Similar to the case of $-\infty<\xi<-6$, we focus our insights on the estimates for the Cauchy-Green operator $S$ defined by \eqref{C-Gop} and $m^{(3)}_1$ defined by \eqref{m31}. Then we have the following two estimations.
\begin{lemma}
The norm of the integral operator $S$ decay to zero as $t\rightarrow\infty$, and
\begin{equation}
\|S\|_{L^{\infty}\rightarrow L^{\infty}}=\mathcal{O}(t^{-\frac{1}{2}}).
\end{equation}
\end{lemma}
\begin{proof}
The proof is the analogue of Lemma 5.3. in \cite{ZF}.
\end{proof}
\begin{proposition}
As $t\rightarrow \infty$,
    \begin{equation}
    \vert m^{(3)}_{1}(x,t) \vert \lesssim t^{-1}.
    \end{equation}
\end{proposition}
\begin{proof}
The proof is similar to Proposition 5.10. in \cite{ZF}.
\end{proof}
\hspace*{\parindent}

\section{Long time asymptotics for the nonlocal mKdV equation}\label{results}
Now we start to construct the asymptotic solution for the nonlocal mKdV equation \eqref{nmkdv} in the case of $\xi<-6$ and  $\xi>6$ by proving Theorem \ref{th1}.
\begin{proof}
Take $\xi<-6$ as an example, the case of $\xi>6$ can be proved similarly. Denote
\begin{equation}
T_1=2\mathrm{i}\underset{k\in\triangle_{1}}\Sigma({\rm Im}z_k+{\rm Im}z_k^{-1})+\underset{l\in\triangle_{1}}\Sigma({\rm Im}\omega_l+{\rm Im}\omega_l^{-1}),
\end{equation}
then $T(z)$ can be written as:
\begin{equation}
T(z)=T(\infty)\left[1+z^{-1}T_1+\mathcal{O}(z^{-2})\right].
\end{equation}
Recall all the transformations for $m(x,t;z)$, we obtain
\begin{equation}
m(z)=T(\infty)^{-\sigma_{3}}m^{(3)}(z)E(z)m^{sol}(z)T(\infty)^{\sigma_{3}}\left[I+z^{-1}T_1^{\sigma_3}+\mathcal{O}(z^{-2})\right].
\end{equation}
From \eqref{msol}, the asymptotic behaviors can be written as follows:
  \begin{equation}
  m^{sol}(z)=I+z^{-1}\left(m^\Lambda_1+\mathcal{O}(e^{-ct})\right)+\mathcal{O}(z^{-2}).
  \end{equation}
  Thus,
  \begin{equation}
  m(z)=I+z^{-1}T(\infty)^{-\sigma_3}\left[m_1^{(3)}+m_1^\Lambda+E_1+\mathcal{O}(e^{-ct}+T_1^{\sigma_3})\right]T(\infty)^{\sigma_3}
  +\mathcal{O}(z^{-2}).
  \end{equation}
  According to the potential recovering formulae \eqref{resconstructm}, $q(x,t)$ is obtained:
  \begin{equation}
  q(x,t)=T(\infty)^{-2}\left[q^\Lambda(x,t)-\mathrm{i}\underset{i=1,2,5,6}\Sigma t^{-\frac{1}{2}+\rm{Im}\nu(\zeta_i)}f_i+
  \mathcal{O}\Big(t^{-1+\underset{i=1,2,\cdots,6}\max{\rm{Im}\nu(\zeta_i)}+\underset{i=1,2,\cdots,6}\max{|\rm{Im}\nu(\zeta_i)|}}\Big)\right].
  \end{equation}
  Denote
  \begin{equation}
  R(t;\xi)=\mathcal{O}\Big(t^{-1+\underset{i=1,2,\cdots,6}\max{\rm{Im}\nu(\zeta_i)}+\underset{i=1,2,\cdots,6}\max{|\rm{Im}\nu(\zeta_i)|}}\Big).
  \end{equation}
  Compare the powers of $t$ in the secondary main term and the error term, \eqref{R} is derived.
\end{proof}

\noindent\textbf{Acknowledgements}

This work is supported by  the National Natural Science
Foundation of China (Grant No. 11671095,  51879045).


\begin{thebibliography}{99}

    \bibitem{Manakov1974} S. V. Manakov,
    \newblock{Nonlinear Fraunhofer diffraction},
    \newblock{Sov. Phys. JETP}, 38(1974), 693-696.

    \bibitem{Zak-Mana1976}V. E. Zakharov, S. V.  Manakov,
    \newblock {Asymptotic behavior of nonlinear wave systems integrated by the inverse scattering method},
    Soviet Physics JETP, 44(1976), 106-112.

    \bibitem{DZ1993} P. Deift, X. Zhou,
    \newblock{A steepest descent method for oscillatory Riemann-Hilbert prblems. Asymptotics for the MKdV equation},
    Ann. Math., 137(1993), 295-368.

    \bibitem{KDV} K. Grunert and G. Teschl,
    \newblock{Long-time asymptotics for the Korteweg de Vries equation via noninear steepest descent},
    Math. Phys. Anal. Geom., 12(2009), 287-324.

    \bibitem{ZD1994} P. Deift, X. Zhou,
    \newblock{Long-Time Behavior of the Non-focusing Nonlinear Schr\"odinger Equation-A Case Study, Lectures in Mathematical Sciences},
    Graduate School of Mathematical Sciences, University of Tokyo, 1994.

    \bibitem{ZD2003}P. Deift, X. Zhou,
    \newblock{Long-time asymptotics for solutions of the NLS equation with initial data in a weighted Sobolev space},
    Commun. Pure Appl. Math., 56(2003), 1029-1077.

    \bibitem{SG1} P.J. Cheng, S. Venakides, X. Zhou,
    \newblock{Long-time asymptotics for the pure radiation solution of the sine-Gordon equation},
    Commun. Partial Differ. Equ., 24(1999), 1195-1262.

    \bibitem{SG2} L. Huang, J. Lenells,
    \newblock{Nonlinear Fourier transforms for the sine-Gordon equation in the quarter plane},
    J. Differ. Equ., 264(2018), 3445-3499.

    \bibitem{CH} A. Boutet de Monvel, A. Kostenko, D. Shepelsky and G. Teschl,
    \newblock{Long-time asymptotics for the Camassa-Holm equation},
     SIAM J. Math. Anal, 41(2009), 1559-1588.

    \bibitem{DP} A. Boutet de Monvel, J. Lenells and D. Shepelsky,
    \newblock{Long-time asymptotics for the Degasperis-Procesi equation on the half-line},
     Ann. Inst. Fourier, 69(2019), 171-230.

    \bibitem{FL} J. Xu, E. G. Fan,
    \newblock{Long-time asymptotics for the Fokas-Lenells equation
    with decaying initial value problem: Without solitons},
    J. Differential Equations, 259(2015), 1098-1148.

    \bibitem{MM1} K. T. R. McLaughlin and P. D. Miller,
    \newblock{The $\bar{\partial}$-steepest descent method and the asymptotic behavior of polynomials orthogonal on the unit circle with
     fixed and exponentially varying non-analytic weights},
     Int. Math. Res. Not., (2006), Art. ID 48673.

    \bibitem{MM2} K. T. R. McLaughlin and P. D. Miller,
    \newblock{The $\bar{\partial}$-steepest descent method for orthogonal polynomials on the real line with varying weights},
    Int. Math. Res. Not., (2008), Art. ID 075

    \bibitem{DM} M. Dieng, K. D. T. R. McLaughlin,
    \newblock{Dispersive asymptotics for linear and integrable equations by the Dbar steepest descent method,
    Nonlinear dispersive partial differential equations and inverse scattering},
    Fields Inst. Comm., Springer, New York, 2019, 253-291.

    \bibitem{CJ} S. Cuccagna, R. Jekins,
    \newblock{On asymptotic stability $N$-solitons of the defocusing nonlinear Schr\"odinger equation}
    \newblock{Comm. Math. Phys.}, 343(2016), 921-969.

    \bibitem{BJ} M. Borghese, R. Jenkins, K. D. T. R. McLaughlin, P. Miller,
    \newblock{Long-time aysmptotic behavior of the focusing nonlinear Schr\"odinger equation},
    Ann. I. H. Poincar\'e Anal, 35(2018), 997-920.

   \bibitem{JL} R. Jenkins, J. Liu, P. Perry and C. Sulem,
   \newblock{Soliton resolution for the derivative nonlinear Schr\"odinger equation},
   Comm. Math. Phys., 363(2018), 1003-1049.

   \bibitem{YF} Y. L. Yang, E. G. Fan,
   \newblock{On the long-time asymptotics of the modified Camassa-Holm equation in space-time solitonic regions},
   Adv. Math., 402(2022), 108340.

    \bibitem{CF} Q. Y. Cheng, E. G. Fan,
    \newblock{Long-time asymptotics for the focusing Fokas-Lenells equation in the solitonic region of space-time},
    J. Differential Equations, 309(2022), 883-948.

    \bibitem{XZF} T. Y. Xu, Z. C. Zhang and E. G. Fan,
    \newblock{Long time asymptotics for the defocusing mKdV equation with finite density initial data in different solitonic regions}, arXiv: 2108.06284v3, 2021.

    \bibitem{ZXF} Z. C. Zhang, T. Y. Xu and E. G. Fan,
    \newblock{Soliton resolution and asymptotic stability of N-soliton solutions for the defocusing mKdV equation with finite density type initial data},
    arXiv: 2108.03650v3, 2021.

     \bibitem{CL} G. Chen, J. Q. Liu,
     \newblock{Soliton resolution for the focusing modified KdV equation},
     Ann. I. H. Poincar\'e Anal, 38 (2021), 2005-2071.

   \bibitem{nmkdv1} M.J. Ablowitz and Z.H. Musslimani,
   \newblock{Inverse scattering transform for the integrable nonlocal nonlinear Schr\"odinger equation},
   Nonlinearity, 29 (2016), 915-946.

   \bibitem{nmkdv2} M.J. Ablowitz, Z.H. Musslimani,
   \newblock{Integrable nonlocal nonlinear equations},
    Stud. Appl. Math. 139 (2017), 7-59.

   \bibitem{eg} M.J. Ablowitz, P.A. Clarkson,
   \newblock{Soliton, Nonlinear Evolution Equations and Inverse Scattering},
   Cambridge Univeristy Press, Cambridge, 1991.

     \bibitem{PT} C.M. Bender, S. Boettcher,
   \newblock{Real spectra in non-Hermitian Hamiltonians having PT symmetry},
   Phys. Rev. Lett. 80 (1998), 5243-5246.

   \bibitem{eg2} X. Y. Tang, Z. F. Liang and X. Z. Hao,
   \newblock{Nonlinear waves of a nonlocal modified KdV equation in the atmospheric and oceanic dynamical system},
   Comm. Nonl. Sci. Numer. Simul. 60 (2018), 62-71.

   \bibitem{ZY} G. Zhang, Z. Yan,
    \newblock{Inverse scattering transforms and soliton solutions of focusing and
     defocusing nonlocal mKdV equations with non-zero boundary conditions},
    \newblock{\em Pys. D}, 402(2020), 132170.

    \bibitem{Darboux} J. L. Ji, Z. N. Zhu,
    \newblock{On a nonlocal modified Korteweg-de Vries equation: Integrability, Darboux transformation and soliton solutions},
    Comm. Nonl. Sci. Numer. Simul., 42 (2017), 699.

    \bibitem{IST} J. L. Ji, Z. N. Zhu,
    \newblock{Soliton solutions of an integrable nonlocal modified Korteweg-de Vries equation through inverse scattering transform},
     J. Math. Anal. Appl., 453 (2017), 973-984.

    \bibitem{HF} F. J. He, E. G. Fan and J. Xu,
    \newblock{Long-time asymptotics for the nonlocal mKdV equation},
     Comm. Theor. Phys., 71 (2019), 475-488.

     \bibitem{YD} Y. Rybalko, D. Shepelsky,
     \newblock{Long-time asymptotics for the integrable nonlocal nonlinear Schr\"odinger equation},
     J. Math. Phys, 60 (2019), 031504.

     \bibitem{ZF} X. Zhou, E. G. Fan,
     \newblock{Long time asymptotics for the nonlocal mKdV equation with finite density initial data}, arXiv:2111.06567, 2021.

     \bibitem{BC} R. Beals, R. R. Coifman,
     \newblock{Scattering and inverse scattering for first order systems},
     Comm. in Pure and Applied Math., 37(1984), 39-90.





\end{thebibliography}
\end{document}